\pdfoutput=1
\documentclass[11pt]{article}
\usepackage{authblk}

\usepackage{color}
\usepackage{helvet}         
\usepackage{courier}        
\usepackage{type1cm}        
\usepackage{makeidx}         
\usepackage{graphicx}        
\usepackage{multicol}        
\usepackage[bottom]{footmisc}
\usepackage{amsmath}
\usepackage{amssymb}
\usepackage{bbold}
\usepackage{amsthm}
\usepackage{subcaption}
\usepackage{sidecap}
\usepackage{comment}
\usepackage[font=normal]{caption}
\usepackage[inline]{enumitem}
\usepackage{yhmath}
\usepackage{thmtools}
\usepackage{thm-restate}
\usepackage{floatrow}
\usepackage{setspace}

\newtheorem{theorem}{Theorem}
\newtheorem{corollary}[theorem]{Corollary}
\newtheorem{lemma}[theorem]{Lemma}
\newtheorem{conjecture}[theorem]{Conjecture}
\newtheorem{proposition}[theorem]{Proposition}

\newtheorem{remark}[theorem]{Remark}
\newtheorem{definition}[theorem]{Definition}
\usepackage[top=2cm, bottom=2cm, left=2cm, right=2cm]{geometry}

\numberwithin{theorem}{section}
\numberwithin{conjecture}{section}
\numberwithin{corollary}{section}
\numberwithin{lemma}{section}
\numberwithin{proposition}{section}
\numberwithin{remark}{section}
\numberwithin{definition}{section}

\numberwithin{figure}{section}
\numberwithin{equation}{section}

\DeclareMathOperator{\SLE}{SLE}

\DeclareMathOperator{\CLE}{CLE}
\DeclareMathOperator{\GFF}{GFF}
\DeclareMathOperator{\hF}{{}_2F_1}

\begin{document}
\title{Global and Local Multiple SLEs for $\kappa \leq 4$ and \\ Connection Probabilities for Level Lines of GFF}

{\small\author[1]{Eveliina Peltola\thanks{eveliina.peltola@unige.ch. E.P. is supported by the ERC AG COMPASP, the NCCR SwissMAP, and the Swiss~NSF.}}
\author[2]{Hao Wu\thanks{hao.wu.proba@gmail.com. H.W. is supported by the Thousand Talents Plan for Young Professionals (No. 20181710136).}}
\affil[1]{Section de Math\'{e}matiques, Universit\'{e} de Gen\`{e}ve, Switzerland}
\affil[2]{Yau Mathematical Sciences Center, Tsinghua University, China}}
\date{}

%
%
\maketitle
\begin{center}
\begin{minipage}{0.8\textwidth}
\abstract{
This article pertains to the classification of multiple Schramm-Loewner evolutions ($\SLE$).
We construct the pure partition functions of multiple $\SLE_\kappa$ with $\kappa\in (0,4]$ 
and relate them to certain extremal multiple $\SLE$ measures,
thus verifying a conjecture from~\cite{KytolaMultipleSLE, KytolaPeoltolaPurePartitionSLE}. 
We prove that the two approaches to construct multiple $\SLE$s --- the global, configurational construction
of~\cite{KozdronLawlerMultipleSLEs, LawlerPartitionFunctionsSLE} 
and the local, growth process construction
of~\cite{KytolaMultipleSLE, DubedatCommutationSLE, GrahamSLE, KytolaPeoltolaPurePartitionSLE}
--- agree.

The pure partition functions are closely related to crossing probabilities in critical 
statistical mechanics models. With explicit formulas in the special case of $\kappa = 4$, we show that 
these functions give the connection probabilities for the level lines of the Gaussian free field 
($\GFF$) with alternating boundary data. 
We also show that certain functions, known as
conformal blocks, give rise to multiple $\SLE_4$ that can be naturally coupled with the $\GFF$ with appropriate boundary data.
}
\end{minipage}
\end{center}

\newcommand{\eps}{\epsilon}
\newcommand{\ov}{\overline}
\newcommand{\U}{\mathbb{U}}
\newcommand{\T}{\mathbb{T}}
\newcommand{\HH}{\mathbb{H}}
\newcommand{\LA}{\mathcal{A}}
\newcommand{\LB}{\mathcal{B}}
\newcommand{\LC}{\mathcal{C}}
\newcommand{\LD}{\mathcal{D}}
\newcommand{\LE}{\mathcal{E}}
\newcommand{\LF}{\mathcal{F}}
\newcommand{\LK}{\mathcal{K}}
\newcommand{\LG}{\mathcal{G}}
\newcommand{\LL}{\mathcal{L}}
\newcommand{\LM}{\mathcal{M}}
\newcommand{\LQ}{\mathcal{Q}}
\newcommand{\LS}{\mathcal{S}}
\newcommand{\LU}{\mathcal{U}}
\newcommand{\LV}{\mathcal{V}}
\newcommand{\LZ}{\mathcal{Z}}
\newcommand{\LH}{\mathcal{H}}
\newcommand{\R}{\mathbb{R}}
\newcommand{\C}{\mathbb{C}}
\newcommand{\N}{\mathbb{N}}
\newcommand{\Z}{\mathbb{Z}}
\newcommand{\E}{\mathbb{E}}
\newcommand{\PP}{\mathbb{P}}
\newcommand{\QQ}{\mathbb{Q}}
\newcommand{\A}{\mathbb{A}}
\newcommand{\one}{\mathbb{1}}
\newcommand{\bn}{\mathbf{n}}
\newcommand{\MR}{MR}
\newcommand{\cond}{\,|\,}
\newcommand{\la}{\langle}
\newcommand{\ra}{\rangle}
\newcommand{\tree}{\Upsilon}
\newcommand{\dist}{\text{dist}}
\newcommand{\free}{\text{free}}

\global\long\def\ud{\mathrm{d}}
\global\long\def\der#1{\frac{\ud}{\ud#1}}
\global\long\def\pder#1{\frac{\partial}{\partial#1}}
\global\long\def\pdder#1{\frac{\partial^{2}}{\partial#1^{2}}}

\global\long\def\PartF{\mathcal{Z}}
\global\long\def\CobloF{\mathcal{U}}
\global\long\def\chamber{\mathfrak{X}}
\global\long\def\domainofdef{\mathfrak{U}}

\global\long\def\Catalan{\mathrm{C}}
\global\long\def\LP{\mathrm{LP}}
\global\long\def\DP{\mathrm{DP}}
\global\long\def\DPleq{\preceq} 
\global\long\def\DPgeq{\succeq} 
\newcommand{\wedgeat}[1]{\lozenge_#1} 
\newcommand{\upwedgeat}[1]{\wedge^#1}
\newcommand{\downwedgeat}[1]{\vee_#1}
\newcommand{\slopeat}[1]{\times_#1}
\newcommand{\removewedge}[1]{\setminus \wedgeat{#1}}
\newcommand{\removeupwedge}[1]{\setminus \upwedgeat{#1}}
\newcommand{\removedownwedge}[1]{\setminus \downwedgeat{#1}}
\newcommand{\wedgelift}[1]{\uparrow \wedgeat{#1}} 
\global\long\def\Mmat{\mathcal{M}}
\global\long\def\Minv{\mathcal{M}^{-1}}
\global\long\def\link#1#2{\{#1,#2\}}
\global\long\def\removeLink{/}
\global\long\def\nested{\boldsymbol{\underline{\Cap}}}
\global\long\def\unnested{\boldsymbol{\underline{\cap\cap}}}

\newcommand{\KWleq}{\stackrel{\scriptscriptstyle{()}}{\scriptstyle{\longleftarrow}}}
\newcommand{\CItilingsof}{\mathcal{C}}

\global\long\def\Rpos{\R_{> 0}}
\global\long\def\Znn{\Z_{\geq 0}}

\global\long\def\localSLE{\mathsf{P}}

\global\long\def\FKdual{\mathcal{L}}

\global\long\def\Test_space{C_c^\infty}
\global\long\def\Distr_space{(\Test_space)^*}

\global\long\def\im#1{\operatorname{Im}(#1)}

%

\global\long\def\gff{\Gamma}
\global\long\def\Height{\LH}

\newpage
\tableofcontents
\newpage

\section{Introduction}
\label{sec::intro}
Conformal invariance and critical phenomena in two-dimensional statistical physics have been active areas 
of research in the last few decades, both in the mathematics and physics communities.
Conformal invariance can be studied in terms of correlations and interfaces in critical models.
This article concerns conformally invariant probability measures on curves that 
describe scaling limits of interfaces in critical lattice models (with suitable boundary conditions). 

For one chordal curve between two boundary points, such scaling limit
results have been rigorously established for many models: 
critical percolation~\cite{SmirnovPercolationConformalInvariance, CamiaNewmanPercolation},
the loop-erased random walk and the uniform spanning 
tree~\cite{LawlerSchrammWernerLERWUST, Zhan:Scaling_limits_of_planar_LERW_in_finitely_connected_domains}, 
level lines of the discrete Gaussian free 
field~\cite{SchrammSheffieldDiscreteGFF, SchrammSheffieldContinuumGFF}, and
the critical Ising and FK-Ising models~\cite{CDCHKSConvergenceIsingSLE}. 
In this case, the limiting object is a random curve known as the chordal $\SLE_\kappa$ (Schramm-Loewner evolution),
uniquely characterized by a single parameter $\kappa \geq 0$ together with conformal invariance
and a domain Markov property~\cite{SchrammScalinglimitsLERWUST}. 
In general, interfaces of critical lattice models with suitable boundary conditions converge to variants of the $\SLE_\kappa$
(see, e.g.,~\cite{HonglerKytolaIsingFree} for the critical Ising model with plus-minus-free boundary conditions,
and~\cite{Zhan:Scaling_limits_of_planar_LERW_in_finitely_connected_domains} for the loop-erased random walk). 
In particular, multiple interfaces converge to several interacting $\SLE$ 
curves~\cite{IzyurovIsingMultiplyConnectedDomains, WuIsingHyperSLE, BeffaraPeltolaWuUniquenessGloableMultipleSLEs, KemppainenSmirnovFKIsing}. 
These interacting random curves cannot be classified by conformal invariance and the domain Markov property alone, 
but additional data is needed~\cite{KytolaMultipleSLE, DubedatCommutationSLE, GrahamSLE, KozdronLawlerMultipleSLEs, 
LawlerPartitionFunctionsSLE, KytolaPeoltolaPurePartitionSLE}.
Together with results in~\cite{BeffaraPeltolaWuUniquenessGloableMultipleSLEs}, 
the main results of the present article provide with a general classification for $\kappa \leq 4$.


It is also natural to ask questions about the global behavior of the interfaces, such as their crossing or connection probabilities.
In fact, such a crossing probability, known as Cardy's formula, was a crucial ingredient in the 
proof of the conformal invariance of the scaling limit of critical percolation~\cite{SmirnovPercolationConformalInvariance, CamiaNewmanPercolation}.
In Figure~\ref{fig::Ising}, a simulation of the critical Ising model with alternating boundary conditions is depicted.
The figure shows one possible connectivity of the interfaces separating the black and yellow regions, 
but when sampling from the Gibbs measure, other planar connectivities can also arise. 
One may then ask with which probability do the various connectivities occur.
For discrete models, the answer is known only for loop-erased random walks ($\kappa = 2$) and the double-dimer model 
($\kappa = 4$)~\cite{Kenyon-Wilson:Boundary_partitions_in_trees_and_dimers, 
KKP:Correlations_in_planar_LERW_and_UST_and_combinatorics_of_conformal_blocks},
whereas for instance the cases of the Ising model ($\kappa = 3$) and percolation ($\kappa = 6$) are unknown.
However, scaling limits of these connection probabilities are encoded in certain quantities related to multiple $\SLE$s,
known as pure partition functions~\cite{PeltolaWuIsing}. These functions give the Radon-Nikodym derivatives of 
multiple $\SLE$ measures with respect to product measures of independent $\SLE$s. 

\smallbreak

In this article, we construct the pure partition functions of multiple $\SLE_\kappa$ for all $\kappa \in (0,4]$ 
and show that they are smooth, positive, and (essentially) unique. 
We also relate these functions to certain extremal multiple $\SLE$ measures, 
thus verifying a conjecture from~\cite{KytolaMultipleSLE, KytolaPeoltolaPurePartitionSLE}. 
To find the pure partition functions, we give a global construction of multiple $\SLE_\kappa$ measures 
in the spirit of~\cite{KozdronLawlerMultipleSLEs, LawlerPartitionFunctionsSLE, LawlerSLENotes},
but pertaining to the complete classification of these random curves. 
We also prove that, as probability measures on curve segments, 
these ``global'' multiple $\SLE$s agree with another approach to construct and classify interacting $\SLE$ curves, 
known as ``local'' multiple $\SLE$s~\cite{KytolaMultipleSLE, DubedatCommutationSLE, GrahamSLE, KytolaPeoltolaPurePartitionSLE}. 

The $\SLE_4$ processes are known to be realized as level lines of the Gaussian free field ($\GFF$). 
In the spirit of~\cite{Kenyon-Wilson:Boundary_partitions_in_trees_and_dimers, 
KKP:Correlations_in_planar_LERW_and_UST_and_combinatorics_of_conformal_blocks}, 
we find algebraic formulas for the pure partition functions in this case and show that they
give explicitly the connection probabilities for the level lines of the $\GFF$ with alternating boundary data.
We also show that certain functions, known as conformal blocks, give rise to multiple $\SLE_4$ processes that
can be naturally coupled with the $\GFF$ with appropriate boundary data.

\begin{figure} 
\begin{center}
\includegraphics[width=.3\textwidth]{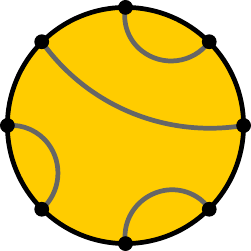} \qquad \qquad \qquad
\includegraphics[width=.3\textwidth]{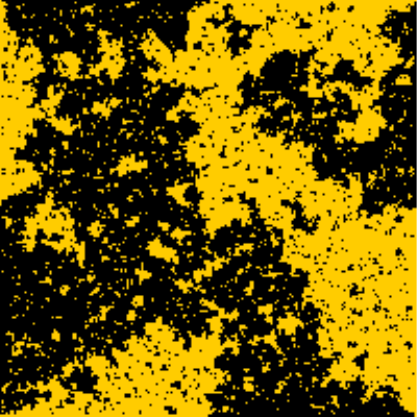}
\end{center}
\caption{Simulation of the critical Ising model with alternating boundary conditions and 
the corresponding link pattern $\alpha \in \LP_4$.}
\label{fig::Ising}
\end{figure}

\subsection{Multiple SLEs and Pure Partition Functions}

One can naturally view interfaces in discrete models as dynamical processes. 
Indeed, in his seminal article~\cite{SchrammScalinglimitsLERWUST}, O.~Schramm defined the $\SLE_\kappa$
as a random growth process (Loewner chain) whose time evolution 
is encoded in an ordinary differential equation (Loewner equation, see Section~\ref{subsec::pre_sle}). 
Using the same idea, one may generate processes of several $\SLE_\kappa$ 
curves by describing their time evolution via a Loewner chain. Such processes are \textit{local multiple} $\SLE$s:
probability measures on curve segments growing from $2N$ fixed boundary points 
$x_1, \ldots, x_{2N} \in \partial \Omega$ of a simply connected domain $\Omega \subset \C$, 
only defined up to a stopping time strictly smaller than the time when the curves touch (we call this localization). 

We prove in Theorem~\ref{thm::global_existence} that,
when $\kappa \leq 4$, localizations of global multiple $\SLE$s give rise to local multiple $\SLE$s.
Then, the $2N$ curve segments form $N$ planar, non-intersecting simple curves connecting 
the $2N$ marked boundary points pairwise, as in Figure~\ref{fig::Ising} for the critical Ising interfaces.
Topologically, these $N$ curves form a planar pair partition, which we call a \textit{link pattern} and denote by
$\alpha = \{ \link{a_1}{b_1}, \ldots, \link{a_N}{b_N} \}$, where $\link{a}{b}$ are the pairs in $\alpha$, called \textit{links}.
The set of link patterns of $N$ links on $\{1, 2, \ldots, 2N\}$ is denoted by $\LP_N$. The number of elements in $\LP_N$ 
is a Catalan number, $\#\LP_N = \Catalan_N = \frac{1}{N+1} \binom{2N}{N}$. We also denote by 
$\LP = \bigsqcup_{N \geq 0} \LP_N$ the set of link patterns of any number of links, where 
we include the empty link pattern $\emptyset \in \LP_0$ in the case $N=0$.

By the results of~\cite{DubedatCommutationSLE,KytolaPeoltolaPurePartitionSLE},
the local $N$-$\SLE_\kappa$ probability measures are classified by smooth functions $\PartF$
of the marked points, called partition functions. 
It is believed that they form a $\Catalan_N$-dimensional space, with basis given by certain special
elements $\PartF_\alpha$, called pure partition functions, indexed by 
the $\Catalan_N$ link patterns $\alpha \in \LP_N$. 
These functions can be related to scaling limits of crossing 
probabilities in discrete models --- 
see~\cite{KKP:Correlations_in_planar_LERW_and_UST_and_combinatorics_of_conformal_blocks, PeltolaWuIsing} 
and Section~\ref{subsec::ising_conjecture} below for discussions on this.
In general, however, even the existence of such functions $\PartF_\alpha$ is not clear.
We settle this problem for all $\kappa \in (0,4]$ in Theorem~\ref{thm::purepartition_existence}.

\smallbreak

To state our results, we need to introduce some definitions and notation. 
Throughout this article, we denote by $\HH=\{z\in\C \colon \im{z}>0\}$ the upper half-plane,
and we use the following real parameters: $\kappa > 0$,
\begin{align*} 
h = \frac{6-\kappa}{2\kappa} , \qquad \qquad \text{and} \qquad \qquad  c = \frac{(3\kappa-8)(6-\kappa)}{2\kappa}. 
\end{align*}

A multiple $\SLE_\kappa$ \textit{partition function} is a positive smooth function 
\begin{align*}
\PartF \colon \chamber_{2N} \to \Rpos
\end{align*}
defined on the configuration space
$\chamber_{2N} :=\; \{ (x_{1},\ldots,x_{2N}) \in \R^{2N} \colon x_{1} < \cdots < x_{2N} \} $
satisfying the following two properties: 
\begin{itemize}[align=left]
\item[$\mathrm{(PDE)}$] \textit{Partial differential equations of second order}: 
We have
\begin{align}\label{eq: multiple SLE PDEs}
\left[ \frac{\kappa}{2}\partial^2_i + \sum_{j\neq i}\left(\frac{2}{x_{j}-x_{i}}\partial_j - 
\frac{2h}{(x_{j}-x_{i})^{2}}\right) \right]
\PartF(x_1,\ldots,x_{2N}) =  0 , \qquad \text{for all } i \in \{1,\ldots,2N\} .
\end{align}
\item[$\mathrm{(COV)}$] \textit{M\"obius covariance}: 
For all M\"obius maps 
$\varphi$ of $\HH$ 
such that $\varphi(x_{1}) < \cdots < \varphi(x_{2N})$, we have
\begin{align}\label{eq: multiple SLE Mobius covariance}
\PartF(x_{1},\ldots,x_{2N}) = 
\prod_{i=1}^{2N} \varphi'(x_{i})^{h} 
\times \PartF(\varphi(x_{1}),\ldots,\varphi(x_{2N})) .
\end{align}
\end{itemize}
Given such a function, one can construct a local $N$-$\SLE_\kappa$ as we discuss in 
Section~\ref{subsec::global_vs_local}.
The above properties $\mathrm{(PDE)}$~\eqref{eq: multiple SLE PDEs} and
$\mathrm{(COV)}$~\eqref{eq: multiple SLE Mobius covariance} guarantee that this
local multiple $\SLE$ process is conformally invariant,
the marginal law of one curve with respect to the joint law of all of the curves is a suitably 
weighted chordal $\SLE_\kappa$, and that the curves enjoy a certain ``commutation'', or 
``stochastic reparameterization invariance'' property --- see~\cite{DubedatCommutationSLE,GrahamSLE,KytolaPeoltolaPurePartitionSLE} for details.  

The \textit{pure partition functions} $\PartF_{\alpha} \colon \chamber_{2N} \to \Rpos$ are indexed by
link patterns $\alpha \in \LP_N$. They are positive solutions to
$\mathrm{(PDE)}$~\eqref{eq: multiple SLE PDEs} and 
$\mathrm{(COV)}$~\eqref{eq: multiple SLE Mobius covariance} singled out by boundary conditions
given in terms of their asymptotic behavior, determined by the link pattern $\alpha$:

\begin{itemize}[align=left]
\item[$\mathrm{(ASY)}$] \textit{Asymptotics}: 
For all $\alpha \in \LP_N$ and for all $j \in \{1, \ldots, 2N-1 \}$ and $\xi \in (x_{j-1}, x_{j+2})$, we have
\begin{align}\label{eq: multiple SLE asymptotics}
\lim_{x_j , x_{j+1} \to \xi} 
\frac{\PartF_\alpha(x_1 , \ldots , x_{2N})}{(x_{j+1} - x_j)^{-2h}} 
=\begin{cases}
0, \quad &
    \text{if } \link{j}{j+1} \notin \alpha, \\
\PartF_{\hat{\alpha}}(x_{1},\ldots,x_{j-1},x_{j+2},\ldots,x_{2N}), &
    \text{if } \link{j}{j+1} \in \alpha ,
\end{cases} 
\end{align}
where 
$\hat{\alpha} = \alpha \removeLink \link{j}{j+1} \in \LP_{N-1}$ denotes
the link pattern obtained from $\alpha$ by removing the link $\link{j}{j+1}$ and relabeling 
the remaining indices by $1, 2, \ldots, 2N - 2$ (see Figure~\ref{fig::link_removal}).
\end{itemize}

\begin{figure}[h]
\centering
\includegraphics[width=\textwidth]{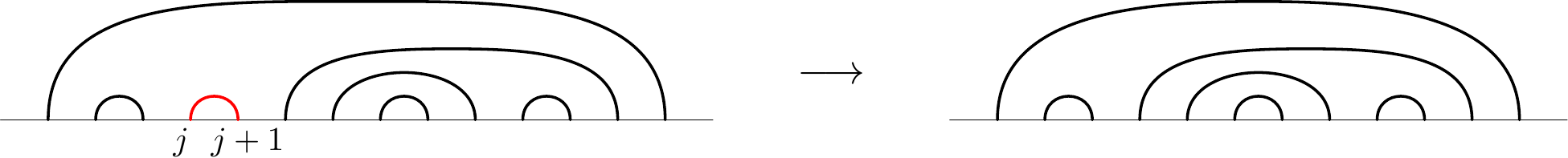}
\captionsetup{width=\textwidth}
\caption{The removal of a link from a link pattern (here $j=4$ and $N=7$).
The left figure is the link pattern 
$\alpha = \{ \link{1}{14}, \link{2}{3}, \link{4}{5}, \link{6}{13}, \link{7}{10}, \link{8}{9}, \link{11}{12} \} \in \LP_7$ 
and the right figure the link pattern
$\alpha \removeLink\link{4}{5} = \{ \link{1}{12}, \link{2}{3}, \link{4}{11}, \link{5}{8}, \link{6}{7}, \link{9}{10} \} \in \LP_6$.}
\label{fig::link_removal} 
\end{figure}

Attempts to find and classify these functions using Coulomb gas techniques 
have been made, e.g., in~\cite{KytolaMultipleSLE, Dubedat:Euler_integrals_for_commuting_SLEs, DubedatCommutationSLE,
FloresKlebanSolutionSpacePDE4, KytolaPeoltolaPurePartitionSLE};
see also~\cite{DF-four_point_correlation_functions,
FSK-Multiple_SLE_connectivity_weights_for_rectangles_hexagons_and_octagons, 
FSKZ-A_formula_for_crossing_probabilities_of_critical_systems_inside_polygons,
Lenells-Viklund:Coulomb_gas_integrals_for_commuting_SLEs1, Lenells-Viklund:Coulomb_gas_integrals_for_commuting_SLEs2}.
The main difficulty in the Coulomb gas approach is to show that the constructed functions are positive, 
whereas smoothness is immediate. On the other hand, as we will see in Lemma~\ref{lem::partitionfunction_positive},
positivity is manifest from the global construction of multiple $\SLE$s, but in this approach, the main obstacle is establishing the 
smoothness\footnote{Recently, another proof for the smoothness with $\kappa<4$  in the global approach appeared in~\cite{LawlerSmoothness}.}.
In this article, we combine the approach of~\cite{KozdronLawlerMultipleSLEs, LawlerPartitionFunctionsSLE} (global construction) with that 
of~\cite{DubedatCommutationSLE, 
Dubedat:SLE_and_Virasoro_representations_localization, 
Dubedat:SLE_and_Virasoro_representations_fusion, KytolaPeoltolaPurePartitionSLE} (local construction and PDE approach),
to show that there exist unique pure partition functions for multiple $\SLE_\kappa$ for all $\kappa \in (0,4]$:

\begin{restatable}{theorem}{purepartitionexistence}
\label{thm::purepartition_existence}
Let $\kappa \in (0,4]$. There exists a unique collection $\{\PartF_{\alpha} \colon \alpha\in \LP\}$ of smooth
functions $\PartF_\alpha \colon \chamber_{2N} \to \R$, for $\alpha \in \LP_N$, 
satisfying the normalization $\PartF_\emptyset = 1$, the power law growth bound given in~\eqref{eqn::powerlawbound} in Section~\ref{sec::pre},
and properties $\mathrm{(PDE)}$~\eqref{eq: multiple SLE PDEs},
$\mathrm{(COV)}$~\eqref{eq: multiple SLE Mobius covariance}, and $\mathrm{(ASY)}$~\eqref{eq: multiple SLE asymptotics}.
These functions have the following further properties:
\begin{enumerate}
\item 
For all 
$\alpha \in \LP_N$, 
we have the stronger power law bound
\begin{align}\label{eqn::partitionfunction_positive}
0<\PartF_{\alpha}(x_1, \ldots, x_{2N}) 
\le \prod_{\link{a}{b} \in \alpha} |x_{b}-x_{a}|^{-2h} .
\end{align} 

\item For each $N \geq 0$, the functions $\{\PartF_{\alpha} \colon \alpha\in \LP_N\}$ are linearly independent. 
\end{enumerate}
\end{restatable}

Next, we make some remarks concerning the above result.

\begin{itemize}
\item The bound~\eqref{eqn::partitionfunction_positive} stated above is very strong.
First of all, together with smoothness, the positivity in~\eqref{eqn::partitionfunction_positive} enables us to construct local multiple $\SLE$s (Corollary~\ref{cor::localmultiplesle}). 
Second, using the upper bound in~\eqref{eqn::partitionfunction_positive}, we will prove 
in Proposition~\ref{prop::loewnerchain_purepartition} 
that the curves in these local multiple $\SLE$s 
are continuous up to and including the continuation threshold, and they connect the marked points in the expected way
--- according to the connectivity $\alpha$. 
Third, the upper bound in~\eqref{eqn::partitionfunction_positive} 
is also crucial in our proof of Theorem~\ref{thm::multiple_sle_4}, 
stated below, concerning the connection probabilities of the level lines of the $\GFF$,
as well as for establishing analogous results for other models~\cite{PeltolaWuIsing}.

\item For $\kappa = 2$, the existence of the functions $\PartF_\alpha$ was already
known before~\cite{KozdronLawlerLERWestimates, KKP:Correlations_in_planar_LERW_and_UST_and_combinatorics_of_conformal_blocks}. 
In this case, the positivity and smoothness can be established by identifying $\PartF_\alpha$ as scaling limits 
of connection probabilities for multichordal loop-erased random walks.

\item In general, it follows from Theorem~\ref{thm::purepartition_existence} that 
the functions $\PartF_{\alpha}$ constructed in the previous works \cite{FloresKlebanSolutionSpacePDE1, KytolaPeoltolaPurePartitionSLE}
are indeed positive, as conjectured, and agree with the functions of Theorem~\ref{thm::purepartition_existence}. 

\item Above, the pure partition functions $\PartF_\alpha$ are only defined for the upper half-plane $\HH$. 
In other simply connected domains $\Omega$, 
when the marked points lie on sufficiently smooth boundary segments,
we may extend the definition of $\PartF_\alpha$ by conformal covariance:
taking any conformal map $\varphi \colon \Omega\to\HH$ such that $\varphi(x_1)<\cdots<\varphi(x_{2N})$, we set
\begin{align}\label{eq: multiple SLE conformal covariance}
\PartF_{\alpha}(\Omega;x_{1},\ldots,x_{2N}) := 
\prod_{i=1}^{2N} |\varphi'(x_{i})|^{h} \times \PartF_{\alpha}(\varphi(x_{1}),\ldots,\varphi(x_{2N})) .
\end{align}
\end{itemize}

Both the global and local definitions of multiple $\SLE$s enjoy conformal invariance and a domain Markov property.
However, only in the case of one curve, these two properties uniquely determine the $\SLE_\kappa$.
With $N \geq 2$, configurations of curves connecting the marked points 
$x_1, \ldots, x_{2N} \in \partial \Omega$ in the simply connected domain $\Omega$ 
have non-trivial conformal moduli, and their probability measures should 
form a convex set of dimension higher than one.
The classification of local multiple $\SLE$s is well established: they are in one-to-one correspondence with 
(normalized) partition functions~\cite{DubedatCommutationSLE, KytolaPeoltolaPurePartitionSLE}.
Thus, we may characterize the convex set of these local $N$-$\SLE_\kappa$ probability measures in the following way:

\begin{restatable}{corollary}{localmultiplesle}
\label{cor::localmultiplesle}
Let $\kappa \in (0,4]$. For any $\alpha \in \LP_N$, there exists a local $N$-$\SLE_\kappa$ with partition function~$\PartF_\alpha$. 
For any $N\geq 1$, the convex hull of the local $N$-$\SLE_\kappa$ 
corresponding to $\{\PartF_{\alpha} \colon \alpha\in \LP_N \}$
has dimension $\Catalan_N-1$.
The $\Catalan_N$ local $N$-$\SLE_\kappa$ probability measures
with pure partition functions $\PartF_\alpha$
are the extremal points of this convex set. 
\end{restatable}

\subsection{Global Multiple SLEs}

To prove Theorem~\ref{thm::purepartition_existence}, we construct the pure partition functions $\PartF_\alpha$ 
from the Radon-Nikodym derivatives of global multiple $\SLE$ measures with respect to product measures of independent $\SLE$s.
To this end, in Theorem~\ref{thm::global_existence}, 
we give a construction of global multiple $\SLE_\kappa$ measures, 
for any number of curves and for all possible topological connectivities, when $\kappa\in (0,4]$.
The construction is not new as such: it was done by M.~Kozdron and G.~Lawler~\cite{KozdronLawlerMultipleSLEs}
in the special case of the rainbow link pattern $\nested_N$, illustrated in Figure~\ref{fig::rainbow}
(see also~\cite[Section~3.4]{Dubedat:Euler_integrals_for_commuting_SLEs}).
For general link patterns, an idea for the construction appeared in~\cite[Section~2.7]{LawlerPartitionFunctionsSLE}. 
However, to prove local commutation of the curves, one needs sufficient regularity that was not established 
in these articles (for this, see~\cite{DubedatCommutationSLE, Dubedat:SLE_and_Virasoro_representations_localization, 
Dubedat:SLE_and_Virasoro_representations_fusion}). 

In the previous works~\cite{KozdronLawlerMultipleSLEs,LawlerPartitionFunctionsSLE}, 
the global multiple $\SLE$s were defined in terms of Girsanov re-weighting of chordal $\SLE$s.
We prefer another definition, where only a minimal amount of characterizing properties are given.
In subsequent work~\cite{BeffaraPeltolaWuUniquenessGloableMultipleSLEs}, 
we prove that 
this definition is optimal in the sense that the global multiple $\SLE$s are uniquely determined by the below stated conditional law property.

\smallbreak

First, we define a \emph{(topological) polygon} to be a $(2N+1)$-tuple $(\Omega; x_1, \ldots, x_{2N})$, 
where $\Omega \subset \C$ is a simply connected domain 
and $x_1, \ldots, x_{2N} \in \partial \Omega$ are $2N$ distinct boundary points appearing in counterclockwise order 
on locally connected boundary segments.
We also say that $U \subset \Omega$ is a \emph{sub-polygon of $\Omega$} if $U$ is simply connected 
and $U$ and $\Omega$ agree in neighborhoods of $x_1, \ldots, x_{2N}$.
When $N=1$, we let $X_0(\Omega; x_1,x_2)$ be the set of continuous simple unparameterized curves in $\Omega$ 
connecting $x_1$ and $x_2$ such that they only touch the boundary $\partial \Omega$ in $\{x_1,x_2\}$. 
More generally, when $N\ge 2$, we consider pairwise disjoint continuous simple curves in $\Omega$ such that each of them connects 
two points among $\{x_1, \ldots, x_{2N}\}$. We encode the connectivities of the curves in link patterns 
$\alpha  = \{ \link{a_1}{b_1}, \ldots, \link{a_N}{b_N} \} \in \LP_N$, and we let $X_0^{\alpha}(\Omega; x_1, \ldots, x_{2N})$ be the set of families $(\eta_1, \ldots, \eta_N)$
of pairwise disjoint curves $\eta_j\in X_0(\Omega; x_{a_j}, x_{b_j})$, for $j\in\{1,\ldots,N\}$.

For any link pattern $\alpha \in \LP_N$, we call a probability measure on 
$(\eta_1, \ldots, \eta_N)\in X_0^{\alpha}(\Omega; x_1, \ldots, x_{2N})$ 
\textit{a~global $N$-$\SLE_{\kappa}$ associated to $\alpha$} if, for each $j\in\{1, \ldots, N\}$, 
the conditional law of 
the curve $\eta_j$ given $\{\eta_1, \ldots, \eta_{j-1}, \eta_{j+1}, \ldots, \eta_N\}$ is the chordal 
$\SLE_{\kappa}$ connecting $x_{a_j}$ and $x_{b_j}$ in the component of the domain
$\Omega \setminus \{\eta_1, \ldots, \eta_{j-1}, \eta_{j+1}, \ldots, \eta_N\}$
that contains the endpoints $x_{a_j}$ and $x_{b_j}$ of $\eta_j$ on its boundary
(see Figure~\ref{fig::malpha_cascade} for an illustration). 
This definition is natural from the point of view of discrete models: it 
corresponds to the scaling limit of interfaces with alternating boundary conditions, 
as described in Sections~\ref{subsec::intro_GFF} and~\ref{subsec::ising_conjecture}. 

\begin{restatable}{theorem}{globalexistencethm}
\label{thm::global_existence}
Let $\kappa\in (0,4]$. Let $(\Omega; x_1, \ldots, x_{2N})$ be a polygon. 
For any $\alpha \in \LP_N$, there exists a global $N$-$\SLE_{\kappa}$ associated to~$\alpha$. 
As a probability measure on the initial segments of the curves,
this global $N$-$\SLE_{\kappa}$
coincides with the local $N$-$\SLE_{\kappa}$ with partition function $\PartF_\alpha$.
It has the following further properties:
\begin{enumerate}
\item 
If $U \subset \Omega$ is a sub-polygon, then the global $N$-$\SLE_{\kappa}$
in $U$ is absolutely continuous with respect to the one 
in $\Omega$, with explicit Radon-Nikodym derivative given in Proposition~\ref{prop::multiplesle_boundary_perturbation}.


\item The marginal law of one curve under this global $N$-$\SLE_{\kappa}$ is absolutely continuous with respect to
the chordal $\SLE_\kappa$, with explicit Radon-Nikodym derivative given in Proposition~\ref{prop::GeneralCascadeProp}.
\end{enumerate}
\end{restatable}

We prove the existence of a global $N$-$\SLE_{\kappa}$ associated to $\alpha$ by constructing it 
in Proposition~\ref{prop::global_existence} in Section~\ref{subsec::multiplesle_existence}. 
The two properties of the measure are proved 
in Propositions~\ref{prop::multiplesle_boundary_perturbation} and~\ref{prop::GeneralCascadeProp}
in Section~\ref{subsec::further_properties_of_Zalpha}.
Finally, in Lemma~\ref{lem::global_is_local} in Section~\ref{subsec::global_vs_local},
we prove that the local and global $\SLE_\kappa$ associated to $\alpha$ agree.

\subsection{$\kappa = 4$: Level Lines of Gaussian Free Field}
\label{subsec::intro_GFF}

In the last sections~\ref{sec::levellines_gff} and~\ref{sec::pure_pf_for_sle4} of this article, 
we focus on the two-dimensional Gaussian free field ($\GFF$).
It can be thought of as a natural 2D time analogue of Brownian motion.
Importantly, the $\GFF$ is conformally invariant and satisfies a certain domain Markov property.
In the physics literature, it is also known as the free bosonic field, a very fundamental and well-understood object,
which plays an important role in conformal field theory, quantum gravity, and statistical physics ---
see, e.g.,~\cite{DuplantierSheffieldLQGKPZ} and references therein.
For instance, the 2D $\GFF$ is the scaling limit of the height function of the dimer model~\cite{KenyonDimer}.

In a series of works~\cite{SchrammSheffieldDiscreteGFF, SchrammSheffieldContinuumGFF, 
MillerSheffieldIG1}, the level lines and flow lines of the $\GFF$ were studied. 
The level lines are $\SLE_\kappa$ curves for $\kappa = 4$, and the flow lines $\SLE_{\kappa}$ curves 
for general $\kappa \geq 0$.
In this article, we 
study connection probabilities of the level lines (i.e., the case $\kappa=4$). 
In Theorems~\ref{thm::multiple_sle_4} and~\ref{thm: sle4 pure pffs},
we relate these 
probabilities to the pure partition functions of multiple $\SLE_4$ and find explicit formulas for them.

\smallbreak

Fix 
$\lambda:=\pi/2$. Let $x_1 < \cdots < x_{2N}$,
and let $\gff$ be the $\GFF$ on $\HH$ with alternating boundary data: 
\begin{align} \label{eq:GFFalt}
\lambda \text{ on }(x_{2j-1}, x_{2j}), \text{ for } j \in \{ 1, \ldots, N \} ,
\qquad \text{and} \qquad 
-\lambda \text{ on }(x_{2j}, x_{2j+1}) , \text{ for }  j \in \{ 0, 1, \ldots, N \} ,
\end{align}
with the convention that $x_0=-\infty$ and $x_{2N+1}=\infty$. 
For $j\in \{1,\ldots,N\}$, let $\eta_j$ be the level line of $\gff$ starting from $x_{2j-1}$, 
considered as an oriented curve. If $x_k$ is the other endpoint of $\eta_j$, 
we say that the level line $\eta_j$ terminates at $x_k$. 
The endpoints of the level lines $(\eta_1, \ldots, \eta_N)$ give rise to a planar pair partition,
which we encode in a link pattern $\LA = \LA(\eta_1, \ldots, \eta_N) \in \LP_N$.

\begin{restatable}{theorem}{multipleslefour}
\label{thm::multiple_sle_4}
Consider multiple level lines of the $\GFF$ on $\HH$ with alternating boundary data~\eqref{eq:GFFalt}. 
For any $\alpha \in \LP_N$, the probability $P_{\alpha}:=\PP[\LA=\alpha]$
is strictly positive. Conditioned on the event $\{\LA=\alpha\}$, the collection 
$(\eta_1, \ldots, \eta_N)\in X_0^{\alpha}(\HH; x_1, \ldots, x_{2N})$ is
the global $N$-$\SLE_4$ associated to $\alpha$ constructed in
Theorem~\ref{thm::global_existence}.
The connection probabilities are explicitly given by
\begin{align}\label{eq::crossing_probabilities_for_kappa4}
P_{\alpha} = \frac{\PartF_{\alpha}  (x_1, \ldots, x_{2N}) }
{\PartF^{(N)}_{\GFF}  (x_1, \ldots, x_{2N} ) }, 
\quad \text{ for all } 
\alpha\in \LP_N, \qquad\text{where } \quad \PartF^{(N)}_{\GFF} := \sum_{\alpha\in\LP_N} \PartF_{\alpha} ,
\end{align}
and $\PartF_{\alpha}$ are the functions of Theorem~\ref{thm::purepartition_existence} with $\kappa = 4$.
Finally, for $a,b\in\{1,\ldots,2N\}$, where $a$ is odd and $b$ is even, the probability that the level line 
of the $\GFF$ starting from $x_a$ terminates at $x_b$ is given by 
\begin{align}\label{eqn::levellines_proba_lk}
P^{(a,b)} (x_1, \ldots, x_{2N})
= \prod_{\substack{1\le j\le 2N, \\ j\neq a,b}} 
\bigg| \frac{x_j-x_a}{x_j-x_b} \bigg|^{(-1)^j} .
\end{align}
\end{restatable}
In order to prove Theorem~\ref{thm::multiple_sle_4}, we need good control of the asymptotics of the pure partition functions
$\PartF_\alpha$ of Theorem~\ref{thm::purepartition_existence} with $\kappa = 4$. Indeed, the strong 
bound~\eqref{eqn::partitionfunction_positive} enables us to control terminal values of certain martingales in Section~\ref{sec::levellines_gff}.
The property $\mathrm{(ASY)}$~\eqref{eq: multiple SLE asymptotics} is not sufficient for this purpose.

\smallbreak

An explicit, simple formula for the symmetric partition function 
$\PartF_{\GFF}$ is known~\cite{Dubedat:Euler_integrals_for_commuting_SLEs,
Kenyon-Wilson:Boundary_partitions_in_trees_and_dimers,
KytolaPeoltolaPurePartitionSLE},
see~\eqref{eq: gff symmetric pff} in Lemma~\ref{lem: gff symmetric pff}.
In fact, also the functions $\PartF_\alpha$ for $\kappa = 4$, and thus the connection probabilities 
$P_\alpha$ in~\eqref{eq::crossing_probabilities_for_kappa4}, have explicit algebraic formulas:

\begin{restatable}{theorem}{slepurepffs}
\label{thm: sle4 pure pffs}
Let $\kappa = 4$. Then, the functions $\{\PartF_\alpha \colon \alpha \in \LP\}$ 
of Theorem~\ref{thm::purepartition_existence} can be written as 
\begin{align} \label{eq: Pure partition function for kappa 4}
\PartF_\alpha(x_1, \ldots, x_{2N}) 
= \; & \sum_{\beta \in \LP_N} \Minv_{\alpha,\beta} \; \CobloF_\beta(x_1, \ldots, x_{2N}) ,
\end{align}
where $\CobloF_\beta$ are explicit functions defined in~\eqref{eq: CobloF definition} and
the coefficients $\Minv_{\alpha,\beta} \in \Z$ are given in Proposition~\ref{prop: matrix inversion}. 
\end{restatable}

In~\cite{Kenyon-Wilson:Boundary_partitions_in_trees_and_dimers,
Kenyon-Wilson:Double_dimer_pairings_and_skew_Young_diagrams},
R.~Kenyon and D.~Wilson derived formulas for connection probabilities in discrete models
(e.g., the double-dimer model) and related these to multichordal $\SLE$ connection 
probabilities for $\kappa = 2,4$, and $8$; see in particular
\cite[Theorem~5.1]{Kenyon-Wilson:Boundary_partitions_in_trees_and_dimers}.
The scaling limit of chordal interfaces in the double-dimer model is believed to be the multiple 
$\SLE_4$ (but this has turned out to be notoriously difficult to prove).
In~\cite[Theorem~5.1]{Kenyon-Wilson:Boundary_partitions_in_trees_and_dimers},
it was argued that the scaling limits of the double-dimer connection probabilities indeed agree 
with those of the $\GFF$, i.e., the connection probabilities given by 
$P_\alpha$ in Theorem~\ref{thm::multiple_sle_4}.
However, detailed analysis of the appropriate martingales was not carried~out.

The coefficients $\Minv_{\alpha,\beta}$ appearing in Theorem~\ref{thm: sle4 pure pffs} are enumerations of 
certain combinatorial objects known as ``cover-inclusive Dyck tilings'' (see Section~\ref{subsec: Combinatorics}). 
They were first introduced and studied in the 
articles~\cite{Kenyon-Wilson:Boundary_partitions_in_trees_and_dimers, 
Kenyon-Wilson:Double_dimer_pairings_and_skew_Young_diagrams,
Shigechi-Zinn:Path_representation_of_maximal_parabolic_Kazhdan-Lusztig_polynomials}. 
In this approach, one views the link patterns $\alpha \in \LP_N$ equivalently as walks known as Dyck paths of $2N$ steps, 
as illustrated in Figure~\ref{fig::bijection} and explained in Section~\ref{subsec: Combinatorics}.

\subsection{$\kappa = 3$:  Crossing Probabilities in Critical Ising Model}
\label{subsec::ising_conjecture}

In the article~\cite{PeltolaWuIsing}, we consider crossing probabilities in the critical planar Ising model.
The Ising model is a classical lattice model introduced and studied already in the 1920s by W. Lenz and E.~Ising.
It is arguably one of the most studied models of an order-disorder phase transition. 
Conformal invariance of the scaling limit of the 2D Ising model at criticality, 
in the sense of correlation functions, was postulated in the seminal 
article~\cite{BelavinPolyakovZamolodchikovConformalSymmetry} of A.~A. Belavin, A.~M. Polyakov, and A.~B. Zamolodchikov.
More recently, in his celebrated work~\cite{SmirnovConformalInvariance,SmirnovConformalInvarianceAnnals}, 
S.~Smirnov constructed discrete holomorphic observables,
which offered a way to rigorously establish conformal invariance for 
all correlation functions~\cite{ChelkakSmirnovIsing, ChelkakLzyurovSpinorIsing, HonglerSmirnovIsingEnergy,
ChelkakHonglerLzyurovConformalInvarianceCorrelationIsing}, as well as 
interfaces~\cite{HonglerKytolaIsingFree, 
CDCHKSConvergenceIsingSLE, BenoistHonglerIsingCLE, IzyurovIsingMultiplyConnectedDomains, BeffaraPeltolaWuUniquenessGloableMultipleSLEs}.

In this section, we briefly discuss the problem of determining crossing probabilities in the Ising model with alternating boundary conditions.
Suppose that discrete domains $(\Omega^{\delta}; x_1^{\delta}, \ldots, x_{2N}^{\delta})$ 
approximate a polygon $(\Omega; x_1, \ldots, x_{2N})$ as $\delta\to 0$ in some natural way (e.g., as specified in the aforementioned literature).
Consider the critical Ising model on $\Omega^{\delta}$ with alternating boundary conditions (see also Figure~\ref{fig::Ising}): 
\begin{align*}
\oplus\text{ on }(x_{2j-1}^{\delta} , x_{2j}^{\delta}),\quad\text{for }j\in\{1, \ldots, N\} ,
\qquad \text{and} \qquad 
\ominus\text{ on }(x_{2j}^{\delta} ,  x_{2j+1}^{\delta}),\quad \text{for }j\in\{0,1,\ldots, N\},
\end{align*}
with the convention that $x_{2N}^{\delta} = x_{0}^{\delta}$ and $x_{2N+1}^{\delta} = x_1^{\delta}$. 
Then, macroscopic interfaces $(\eta_1^{\delta}, \ldots, \eta_N^{\delta})$ connect the boundary points $x_1^{\delta}, \ldots, x_{2N}^{\delta}$,
forming a planar connectivity encoded in
a link pattern $\LA^{\delta} \in \LP_N$. 
Conditioned on $\{\LA^{\delta}=\alpha\}$, this collection  
of interfaces converges in the scaling limit to the global $N$-$\SLE_3$ associated to 
$\alpha$~\cite[Proposition~1.3]{BeffaraPeltolaWuUniquenessGloableMultipleSLEs}.

We are interested on the scaling limit of the crossing 
probability $\PP[\LA^{\delta}=\alpha]$ for $\alpha\in\LP_N$. For $N=2$, 
this limit was derived in~\cite[Equation~(4.4)]{IzyurovObservableFree}. 
In general, we expect the following:

\begin{conjecture}\label{conj::ising_crossing_proba}
We have
\begin{align*}
\lim_{\delta\to 0}\PP[\LA^{\delta}=\alpha] 
= \frac{\PartF_{\alpha}(\Omega;x_{1},\ldots,x_{2N})}{\PartF^{(N)}_{\mathrm{Ising}}(\Omega;x_{1},\ldots,x_{2N})} , 
\qquad\text{where } \quad \PartF^{(N)}_{\mathrm{Ising}} := \sum_{\alpha\in\LP_N}\PartF_{\alpha} , 
\end{align*}
and $\PartF_{\alpha}$ are the functions defined by~\eqref{eq: multiple SLE conformal covariance} 
and Theorem~\ref{thm::purepartition_existence} with $\kappa = 3$.
\end{conjecture}

We prove this conjecture for square lattice approximations in~\cite[Theorem~1.1]{PeltolaWuIsing}.
In light of the universality results in~\cite{ChelkakSmirnovIsing},
more general approximations should also work nicely.

The symmetric partition function 
$\PartF_{\mathrm{Ising}}$ has an explicit Pfaffian formula
\cite{KytolaPeoltolaPurePartitionSLE, IzyurovIsingMultiplyConnectedDomains}, 
see~\eqref{eq: ising symmetric pff} in Lemma~\ref{lem::ztotal_ising}.
However, explicit 
formulas for $\PartF_\alpha$ for $\kappa = 3$ are only known in the cases $N=1,2$, and
in contrast to the case of $\kappa = 4$, for $\kappa = 3$ the formulas are in general not algebraic. 


\medbreak
\noindent\textbf{Outline.} 
Section~\ref{sec::pre} contains preliminary material: 
the definition and properties of the $\SLE_\kappa$ processes, discussion about the multiple $\SLE$ partition functions 
and solutions of $\mathrm{(PDE)}$~\eqref{eq: multiple SLE PDEs}, 
as well as combinatorics needed in Section~\ref{sec::pure_pf_for_sle4}.
We also state a crucial result from~\cite{FloresKlebanSolutionSpacePDE2} 
(Theorem~\ref{thm::purepartition_unique} and Corollary~\ref{cor::purepartition_unique})
concerning uniqueness of solutions to $\mathrm{(PDE)}$~\eqref{eq: multiple SLE PDEs}.
Moreover, we recall H\"ormander's condition for hypoellipticity of linear partial differential operators,
crucial for proving the smoothness of $\SLE$ partition functions (Theorem~\ref{thm::HormanderHypoellipticTheorem}
and Proposition~\ref{prop::smoothness}).

The topic of Section~\ref{sec::characterization} is the construction of global multiple $\SLE$s,
in order to prove parts of Theorem~\ref{thm::global_existence}.
We construct global $N$-$\SLE_\kappa$ probability measures for all link patterns $\alpha$ and for all $N$ 
in Section~\ref{subsec::multiplesle_existence} (Proposition~\ref{prop::global_existence}). 
In the next Section~\ref{subsec::further_properties_of_Zalpha}, we give
the boundary perturbation property  
(Proposition~\ref{prop::multiplesle_boundary_perturbation}) and the characterization of the marginal law
(Proposition~\ref{prop::GeneralCascadeProp}) for these random curves.

In Section~\ref{sec::purepartition_existence}, we consider the pure partition functions $\PartF_\alpha$.
Theorem~\ref{thm::purepartition_existence} concerning
the existence and uniqueness of 
$\PartF_\alpha$ is proved in Section~\ref{subsec::purepfexistence_proof}. 
We complete the proof of Theorem~\ref{thm::global_existence} 
with Lemma~\ref{lem::global_is_local} in Section~\ref{subsec::global_vs_local},
by comparing the two definitions for multiple $\SLE$s --- the global and the local. 
In Section~\ref{subsec::global_vs_local}, we also prove Corollary~\ref{cor::localmultiplesle}.
Then, in Section~\ref{subsec::purepartition_thm}, we prove Proposition~\ref{prop::loewnerchain_purepartition}, 
which says that Loewner chains driven by the pure partition functions are generated by continuous curves up to and including the continuation threshold.
Finally, in Section~\ref{subsec::total_pf}, we discuss so-called symmetric partition functions and list explicit formulas 
for them for $\kappa = 2,3,4$.

The last Sections~\ref{sec::levellines_gff} and~\ref{sec::pure_pf_for_sle4} focus on the case of $\kappa = 4$
and the problem of connection probabilities of the level lines of the Gaussian free field.
We introduce the GFF and its level lines in Section~\ref{sub:GFF_pre}.
In Sections~\ref{subsec::GFF_pair_of_level_lines}--\ref{subsec:GFF_connection_proba}, 
we find the connection probabilities of the level lines.
Theorem~\ref{thm::multiple_sle_4} is proved in Section~\ref{subsec::gff_levellines_Plk}.
Then, in Section~\ref{sec::pure_pf_for_sle4}, we investigate the pure partition functions in the case $\kappa = 4$.
First, in Section~\ref{subsec: decay_properties}, we record decay properties of these functions and relate them to the $\SLE_4$ boundary arm-exponents.
In Sections~\ref{subsec: PDE_COV}--\ref{subsec: Conformal blocks}, 
we derive the explicit 
formulas of Theorem~\ref{thm: sle4 pure pffs} for these 
functions, 
using combinatorics and results from~\cite{Kenyon-Wilson:Boundary_partitions_in_trees_and_dimers,
Kenyon-Wilson:Double_dimer_pairings_and_skew_Young_diagrams,
KKP:Correlations_in_planar_LERW_and_UST_and_combinatorics_of_conformal_blocks}.
We find these formulas by constructing functions known as conformal blocks for the $\GFF$.
We also discuss in Section~\ref{subsec: Conformal blocks for GFF} 
how the conformal blocks generate multiple $\SLE_4$ processes that can be naturally coupled with the $\GFF$ with appropriate boundary data
(Proposition~\ref{prop::levellines_conformalblocks}).

The appendices contain some technical results needed in this article that we have found not instructive to include in the main text.

\medbreak
\noindent\textbf{Acknowledgments.}
We thank V.~Beffara, G.~Lawler, and W.~Qian for helpful discussions on multiple $\SLE$s. 
We thank M.~Russkikh for useful discussions on (double-)dimer models and 
B.~Duplantier, S.~Flores, A.~Karrila, K.~Kyt\"ol\"a, and A.~Sepulveda for interesting, 
useful, and stimulating discussions. 
Part of this work was completed during H.W.'s visit at the IHES, which we cordially thank for hospitality. 
Finally, we are grateful to the referee for careful comments on the manuscript.

\section{Preliminaries}
\label{sec::pre}
This section contains definitions and results from the literature that are needed to understand and prove 
the main results of this article. 
In Sections~\ref{subsec::pre_sle} and~\ref{subsec::boundary_perturbation}, we define the chordal $\SLE_\kappa$ and 
give a boundary perturbation property 
for it, using a conformally invariant measure 
known as the Brownian loop measure. 
Then, in Section~\ref{subsec::pre_pure} we discuss the solution space of the system~$\mathrm{(PDE)}$~\eqref{eq: multiple SLE PDEs} 
of second order partial differential equations. We give examples of solutions: multiple $\SLE$ partition functions.
In Theorem~\ref{thm::purepartition_unique}, we state a 
result of S.~Flores and P.~Kleban~\cite{FloresKlebanSolutionSpacePDE2}
concerning the asymptotics of solutions, which we use in Section~\ref{sec::purepartition_existence} to prove the uniqueness of 
the pure partition functions of Theorem~\ref{thm::purepartition_existence}.
In Proposition~\ref{prop::smoothness}, we prove that all solutions of~$\mathrm{(PDE)}$~\eqref{eq: multiple SLE PDEs} are smooth,
by showing that this PDE system is hypoelliptic 
--- to this end, we follow the idea 
of~\cite{Kontsevich:CFT_SLE_and_phase_boundaries, Friedrich-Kalkkinen:On_CFT_and_SLE, Dubedat:SLE_and_Virasoro_representations_localization},
using the powerful theory of H\"ormander~\cite{HormanderHypoelliptic}. 
Finally, in Section~\ref{subsec: Combinatorics} we introduce combinatorial notions and results needed in Section~\ref{sec::pure_pf_for_sle4}.

\subsection{Schramm-Loewner Evolutions}
\label{subsec::pre_sle}
We call a compact subset $K$ of $\overline{\HH}$ an \textit{$\HH$-hull} if $\HH\setminus K$ 
is simply connected. Riemann's mapping theorem asserts that there exists a unique conformal map 
$g_K$ from $\HH\setminus K$ onto $\HH$ with the property that $\lim_{z\to\infty}|g_K(z)-z|=0$.
We say that $g_K$ is \textit{normalized at} $\infty$.

In this article, we consider the following collections of $\HH$-hulls. 
They are associated with families of conformal maps $(g_{t}, t\ge 0)$ 
obtained by solving the Loewner equation: for each $z\in\mathbb{H}$,
\begin{align*}
\partial_{t}{g}_{t}(z)=\frac{2}{g_{t}(z)-W_{t}}, \qquad \qquad  g_{0}(z)=z,
\end{align*}
where $(W_t, t\ge 0)$ is a real-valued continuous function, which we call the driving function. 
Let $T_z$ be the \textit{swallowing time} of $z$ defined as $\sup\{t\ge 0 \colon \inf_{s\in[0,t]}|g_{s}(z)-W_{s}|>0\}$.
Denote $K_{t}:=\overline{\{z\in\mathbb{H}: T_{z}\le t\}}$.
Then, $g_{t}$ is the unique conformal map from $H_{t}:=\mathbb{H}\setminus K_{t}$ onto $\mathbb{H}$ normalized at $\infty$. 
The collection of $\HH$-hulls $(K_{t}, t\ge 0)$ associated with such maps is called a \textit{Loewner chain}.

Let $\kappa \geq 0$. The (chordal) \textit{Schramm-Loewner Evolution} $\SLE_{\kappa}$ in $\HH$ from $0$ to $\infty$ 
is the random Loewner chain $(K_{t}, t\ge 0)$ driven by $W_t=\sqrt{\kappa}B_t$, where $(B_t, t\ge 0)$ is 
the standard 
Brownian motion. S.~Rohde and O.~Schramm proved
in~\cite{RohdeSchrammSLEBasicProperty} 
that $(K_{t}, t\ge 0)$ is almost surely generated by a continuous transient curve, i.e., there almost 
surely exists a continuous curve $\eta$ such that for each $t\ge 0$, $H_{t}$ is the unbounded 
component of $\HH \setminus \eta[0,t]$ and $\lim_{t\to\infty}|\eta(t)|=\infty$. 
This random curve is 
the $\SLE_\kappa$ trace in $\HH$ from $0$ to $\infty$.
It exhibits phase transitions at $\kappa=4$ and $8$: 
the $\SLE_{\kappa}$ curves are simple when $\kappa\in [0,4]$ and they have self-touchings when $\kappa > 4$,
being space-filling when $\kappa\ge 8$. 
In this article, we focus on the range $\kappa\in (0,4]$ when the curve is simple.
Its law  is a probability measure $\PP(\HH; 0,\infty)$ on the set $X_0(\HH; 0,\infty)$.

The $\SLE_\kappa$ is \textit{conformally invariant}: it can be defined in any 
simply connected domain $\Omega$ with two boundary points $x, y \in \partial \Omega$ 
(around which the boundary is locally connected) by pushforward of a conformal map as follows.
Given any conformal map $\varphi \colon \HH \to \Omega$ such that $\varphi(0)=x$ and 
$\varphi(\infty)=y$, we have $\varphi(\eta) \sim \PP(\Omega; x,y)$ if $\eta \sim \PP(\HH; 0,\infty)$,
where $\PP(\Omega; x,y)$ denotes the law of the $\SLE_{\kappa}$ in $\Omega$ from $x$~to~$y$. 

Schramm's classification~\cite{SchrammScalinglimitsLERWUST} shows that $\PP(\Omega; x,y)$ 
is the unique probability measure on curves $\eta \in X_0(\Omega; x,y)$
satisfying conformal invariance and the 
\textit{domain Markov property}: for a stopping time $\tau$,
given an initial segment $\eta[0,\tau]$ of
the $\SLE_\kappa$ curve  
$\eta \sim \PP(\Omega; x,y)$, the conditional law of the remaining piece 
$\eta[\tau,\infty)$ is the law $\PP(\Omega \setminus K_\tau; \eta(\tau),y)$ of the $\SLE_\kappa$ 
in the remaining domain $\Omega \setminus K_\tau$ from the tip $\eta(\tau)$ to $y$.

We will also use the following \textit{reversibility} of the $\SLE_\kappa$ (for $\kappa \leq 4$)~\cite{ZhanReversibility}: 
the time reversal of the $\SLE_\kappa$ curve $\eta \sim \PP(\Omega; x,y)$ in $\Omega$ from $x$ to $y$ has 
the same law $\PP(\Omega; y,x)$ as the $\SLE_{\kappa}$ in $\Omega$ from $y$ to $x$.

\smallbreak

Finally, the following change of target point of the $\SLE_{\kappa}$ will be used in Section~\ref{sec::purepartition_existence}.
 
\begin{lemma}\label{lem::slekapparho_mart}
\textnormal{\cite{SchrammWilsonSLECoordinatechanges}}
Let $\kappa>0$ and $y>0$. 
Up to the first swallowing time of~$y$,
the $\SLE_{\kappa}$ in $\HH$ from $0$ to $y$ 
has the same law as the $\SLE_{\kappa}$ in $\HH$ from $0$ to $\infty$ weighted by 
the local martingale $g_t'(y)^h(g_t(y)-W_t)^{-2h}$.
\end{lemma}

\subsection{Boundary Perturbation of SLE}
\label{subsec::boundary_perturbation}

Let $(\Omega;x,y)$ be a polygon, which in this case is also called a \emph{Dobrushin domain}.
Also, if $U \subset \Omega$ is a sub-polygon, we also call $U$ a \emph{Dobrushin subdomain}.
If, in addition, the boundary points $x$ and $y$ lie on 
sufficiently regular segments of $\partial \Omega$ (e.g., $C^{1+\eps}$ for some $\eps > 0$),
we call $(\Omega;x,y)$ a \emph{nice Dobrushin domain}.

In the next Lemma~\ref{lem::sle_boundary_perturbation},
we recall the boundary perturbation property of the chordal $\SLE_\kappa$.
It gives the Radon-Nikodym derivative between the laws of the chordal $\SLE_\kappa$ curve in $U$ and $\Omega$
in terms of the Brownian loop measure and the boundary Poisson kernel.

The \textit{Brownian loop measure} is 
a conformally invariant measure on unrooted Brownian loops 
in the plane. In the present article, we do not need the precise definition of this 
measure, so we content ourselves with referring to the literature for the definition: 
see, e.g.,~\cite[Sections~3~and~4]{LawlerWernerBrownianLoopsoup} or~\cite{FieldLawlerReversedRadialSLEBrownianLoop}.
Given a non-empty simply connected domain $\Omega\subsetneq\C$ and two disjoint subsets $V_1, V_2 \subset \Omega$, we
denote by $\mu(\Omega; V_1, V_2)$ the Brownian loop measure of loops in $\Omega$ that intersect both 
$V_1$ and $V_2$. This quantity is conformally invariant: 
$\mu(\varphi(\Omega); \varphi(V_1), \varphi(V_2)) = \mu(\Omega; V_1, V_2)$ for any conformal transformation 
$\varphi \colon \Omega \to \varphi(\Omega)$.

In general, the Brownian loop measure is an infinite measure. However, we have $0 \leq \mu(\Omega; V_1, V_2) < \infty$ 
when both of $V_1, V_2$ are closed, one of them is compact, and $\dist(V_1, V_2) > 0$.
More generally, for $n$ disjoint subsets $V_1, \ldots, V_n$ of $\Omega$, 
we denote by $\mu(\Omega; V_1, \ldots, V_n)$ the Brownian loop measure of loops in $\Omega$ that 
intersect all of $V_1, \ldots, V_n$. Provided that $V_j$ are all closed and at least one of them is compact,
the quantity $\mu(\Omega; V_1, \ldots, V_n)$ is finite.

For a nice Dobrushin domain $(\Omega;x,y)$, the \textit{boundary Poisson kernel $H_{\Omega}(x,y)$} is
uniquely characterized by the following
two properties~\eqref{eqn::poisson_cov}~and~\eqref{eqn::poisson_upperhalfplane}. 
First, it is conformally covariant: 
we have  
\begin{align}\label{eqn::poisson_cov}
|\varphi'(x)| |\varphi'(y)| H_{\varphi(\Omega)}(\varphi(x), \varphi(y)) = H_{\Omega}(x,y) ,
\end{align}
for any conformal map $\varphi \colon \Omega \to \varphi(\Omega)$
(since $\Omega$ is nice, the derivative of $\varphi$ extends continuously to neighborhoods of $x$ and $y$).
Second, for the upper-half plane 
with $x,y \in \R$, we have the explicit formula 
\begin{align}\label{eqn::poisson_upperhalfplane}
H_{\HH}(x,y) = |y-x|^{-2} 
\end{align}
(we do not include $\pi^{-1}$ here). In addition, 
if $U \subset \Omega$ is a Dobrushin subdomain, then we have
\begin{align}\label{eqn::poissonkernel_mono}
H_{U}(x,y)\le H_{\Omega}(x,y) .
\end{align}
When considering ratios of boundary Poisson kernels, we may drop the niceness assumption.

\begin{lemma} \label{lem::sle_boundary_perturbation}
Let $\kappa\in (0,4]$. 
Let $(\Omega;x,y)$ be a Dobrushin domain and $U \subset \Omega$ a Dobrushin subdomain. 
Then, the $\SLE_\kappa$ in $U$ connecting $x$ and $y$
is absolutely continuous with respect to the $\SLE_\kappa$ in $\Omega$ connecting $x$ and $y$, with
Radon-Nikodym derivative 
\begin{align*}
\frac{\ud \PP(U; x,y)}{\ud \PP(\Omega; x,y)} (\eta)
= \left(\frac{H_{\Omega}(x,y)}{H_{U}(x,y)}\right)^h 
\one_{\{\eta \subset U\}} \exp(c \mu(\Omega; \eta, \Omega \setminus U)). 
\end{align*}  
\end{lemma}
\begin{proof}
See~\cite[Section 5]{LawlerSchrammWernerConformalRestriction} and~\cite[Proposition~3.1]{KozdronLawlerMultipleSLEs}. 
\end{proof}

\subsection{Solutions to the Second Order PDE System (PDE)}
\label{subsec::pre_pure}
In this section, we present known facts about the solution space of the system~$\mathrm{(PDE)}$~\eqref{eq: multiple SLE PDEs} 
of second order partial differential equations. 
Particular examples of solutions are the multiple $\SLE$ partition functions, and
we give examples of known formulas for them. 
We also state a crucial result from~\cite{FloresKlebanSolutionSpacePDE2}
concerning the asymptotics of solutions. 
This result, Theorem~\ref{thm::purepartition_unique}, says that solutions to
$\mathrm{(PDE)}$~\eqref{eq: multiple SLE PDEs} and 
$\mathrm{(COV)}$~\eqref{eq: multiple SLE Mobius covariance} having certain asymptotic properties 
must vanish. We use this property 
in Section~\ref{sec::purepartition_existence} to prove the uniqueness of the pure partition functions.
Finally, we discuss regularity of the solutions to the system~$\mathrm{(PDE)}$~\eqref{eq: multiple SLE PDEs}:
in Proposition~\ref{prop::smoothness}, we prove that these PDEs are hypoelliptic, that is, 
all distributional solutions for them are in fact smooth functions. 
This result was proved in~\cite{Dubedat:SLE_and_Virasoro_representations_localization} 
using the powerful theory of H\"ormander~\cite{HormanderHypoelliptic}, which we also briefly recall. 
The hypoellipticity of the PDEs in~\eqref{eq: multiple SLE PDEs} was already pointed out earlier in the 
articles~\cite{Kontsevich:CFT_SLE_and_phase_boundaries, Friedrich-Kalkkinen:On_CFT_and_SLE}. 

\subsubsection{Examples of Partition Functions}
\label{subsubsec::partitionfunctions_examples}
For $\kappa\in (0,8)$,
the pure partition functions for $N=1$ and $N=2$ can be found by a calculation.
The case $N=1$ is almost trivial: 
then we have, for $x<y$ and  $\vcenter{\hbox{\includegraphics[scale=0.3]{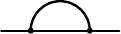}}} = \{ \link{1}{2} \}$,
\[ \PartF^{(1)}(x,y) 
= \PartF_{\vcenter{\hbox{\includegraphics[scale=0.2]{figures/link-0.pdf}}}}(x,y) 
= (y-x)^{-2h} . \]

When $N=2$, the system 
$\mathrm{(PDE)}$~\eqref{eq: multiple SLE PDEs} with M\"obius covariance
$\mathrm{(COV)}$~\eqref{eq: multiple SLE Mobius covariance} reduces to an ordinary differential 
equation (ODE), since we can fix three out of the four degrees of freedom. This ODE is a
hypergeometric equation, whose solutions are well-known. With boundary conditions
$\mathrm{(ASY)}$~\eqref{eq: multiple SLE asymptotics}, we obtain 
for $\vcenter{\hbox{\includegraphics[scale=0.3]{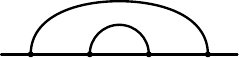}}} = \{\link{1}{4}, \link{2}{3} \}$ 
and $\vcenter{\hbox{\includegraphics[scale=0.3]{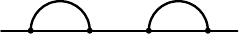}}} = \{\link{1}{2}, \link{3}{4} \}$, 
and for $x_1<x_2<x_3<x_4$,
\begin{align*}
\PartF_{\vcenter{\hbox{\includegraphics[scale=0.2]{figures/link-2.pdf}}}} (x_1,x_2,x_3,x_4)
= \; & (x_4-x_1)^{-2h}(x_3-x_2)^{-2h}z^{2/\kappa} \, \frac{F(z)}{F(1)} , \\ 
\PartF_{\vcenter{\hbox{\includegraphics[scale=0.2]{figures/link-1.pdf}}}} (x_1,x_2,x_3,x_4)
= \; &  (x_2-x_1)^{-2h}(x_4-x_3)^{-2h}(1-z)^{2/\kappa} \, \frac{F(1-z)}{F(1)},
\end{align*}
where $z$ is a cross-ratio and $F$ is a hypergeometric function:  
\begin{align*}
z=\frac{(x_2-x_1)(x_4-x_3)}{(x_4-x_2)(x_3-x_1)},\qquad \qquad
F(\cdot) := \hF\left(\frac{4}{\kappa}, 1-\frac{4}{\kappa}, \frac{8}{\kappa};\cdot\right).
\end{align*}
Note that $F$ is bounded on $[0,1]$ when $\kappa\in (0,8)$. 
For some parameter values, these formulas are algebraic:
\begin{align*}
&\text{For }\kappa=2,\quad &&
\PartF_{\vcenter{\hbox{\includegraphics[scale=0.2]{figures/link-2.pdf}}}} (x_1, x_2, x_3, x_4)=(x_4-x_1)^{-2}(x_3-x_2)^{-2}z(2-z).\\
&\text{For }\kappa=4,\quad &&
\PartF_{\vcenter{\hbox{\includegraphics[scale=0.2]{figures/link-2.pdf}}}} (x_1, x_2, x_3, x_4)=(x_4-x_1)^{-1/2}(x_3-x_2)^{-1/2}z^{1/2}. \\
&\text{For }\kappa=16/3, \quad &&
\PartF_{\vcenter{\hbox{\includegraphics[scale=0.2]{figures/link-2.pdf}}}} (x_1, x_2, x_3, x_4)=(x_4-x_1)^{-1/4}(x_3-x_2)^{-1/4}z^{3/8}(1+\sqrt{1-z})^{-1/2}.
\end{align*}
We note that when $\kappa=4$, we have
\begin{align*}
\frac{\PartF_{\vcenter{\hbox{\includegraphics[scale=0.2]{figures/link-2.pdf}}}}(x_1, x_2, x_3, x_4)}{\PartF_{\vcenter{\hbox{\includegraphics[scale=0.2]{figures/link-2.pdf}}}}(x_1, x_2, x_3, x_4)+\PartF_{\vcenter{\hbox{\includegraphics[scale=0.2]{figures/link-1.pdf}}}}(x_1, x_2, x_3, x_4)}=z.
\end{align*}
The right-hand side coincides with a connection probability of the level lines of the GFF, 
see Lemma~\ref{lem::twosle_proba}.

\subsubsection{Crucial Uniqueness Result}


The following theorem is a deep result due to S.~Flores and P.~Kleban. 
It is formulated as a lemma in the series
\cite{FloresKlebanSolutionSpacePDE1, FloresKlebanSolutionSpacePDE2, FloresKlebanSolutionSpacePDE3, 
FloresKlebanSolutionSpacePDE4} of articles, which concerns the dimension of the solution space 
of $\mathrm{(PDE)}$~\eqref{eq: multiple SLE PDEs} and
$\mathrm{(COV)}$~\eqref{eq: multiple SLE Mobius covariance} under a
condition~\eqref{eqn::powerlawbound} of power law growth given below. The proof of this lemma constitutes the whole 
article~\cite{FloresKlebanSolutionSpacePDE2}, relying on the theory of elliptic 
partial differential equations, Green function techniques, 
and careful estimates on the asymptotics of the solutions.

Uniqueness of solutions to hypoelliptic boundary value problems 
is not applicable in our situation, because the solutions that we consider 
cannot be continuously extended up to the boundary of $\chamber_{2N}$.

\begin{theorem}\label{thm::purepartition_unique}
\textnormal{\cite[Lemma~1]{FloresKlebanSolutionSpacePDE2}}
Let $\kappa\in (0,8)$. Let $F \colon \chamber_{2N} \to \C$ be a function satisfying
properties $\mathrm{(PDE)}$~\eqref{eq: multiple SLE PDEs} and
$\mathrm{(COV)}$~\eqref{eq: multiple SLE Mobius covariance}.
Suppose furthermore that there exist constants $C>0$ and $p>0$ such that for all 
$N \geq 1$ and $(x_1,\ldots, x_{2N}) \in \chamber_{2N}$, we have
\begin{align}\label{eqn::powerlawbound}
|F(x_1, \ldots, x_{2N})| \le C \prod_{1 \leq i<j \leq 2N}(x_j-x_i)^{\mu_{ij}(p)}, 
\qquad \text{where } \quad
\mu_{ij}(p) :=
\begin{cases}
p, \quad &\text{if } |x_j-x_i| > 1, \\
-p, \quad &\text{if } |x_j-x_i| < 1.
\end{cases}
\end{align}
If $F$ also has the asymptotics property
\begin{align*}
\lim_{x_j , x_{j+1} \to \xi} 
\frac{F(x_1 , \ldots , x_{2N})}{(x_{j+1} - x_j)^{-2h}} = 0 ,
\qquad \text{for all } j \in \{ 2, 3, \ldots , 2N-1 \}  \text{ and } \xi \in (x_{j-1}, x_{j+2}) 
\end{align*}
\textnormal{(}with the convention that $x_0 = -\infty$ and  $x_{2N+1} = +\infty$\textnormal{)},
then $F \equiv 0$. 
\end{theorem}

Motivated by Theorem~\ref{thm::purepartition_unique}, we define 
the following solution space of the system $\mathrm{(PDE)}$~\eqref{eq: multiple SLE PDEs}:
\begin{align}\label{eq: solution space}
\mathcal{S}_N := \{ F \colon \chamber_{2N} \to \C \colon F 
\text{ satisfies $\mathrm{(PDE)}$~\eqref{eq: multiple SLE PDEs}, $\mathrm{(COV)}$~\eqref{eq: multiple SLE Mobius covariance}, 
and~\eqref{eqn::powerlawbound}}\} .
\end{align}
We use this notation throughout.
The bound~\eqref{eqn::powerlawbound} is easy to verify for the solutions studied in the present article.
Hence, Theorem~\ref{thm::purepartition_unique} 
gives us the uniqueness of the pure partition functions for Theorem~\ref{thm::purepartition_existence}. 

%

\begin{corollary}\label{cor::purepartition_unique}
Let $\kappa\in (0,8)$. Let $\{F_\alpha \colon \alpha \in \LP\}$ be a collection of functions $F_\alpha \in \LS_N$, 
for $\alpha \in \LP_N$, satisfying $\mathrm{(ASY)}$~\eqref{eq: multiple SLE asymptotics} with normalization $F_\emptyset = 1$.
Then, the collection $\{F_\alpha \colon \alpha \in \LP\}$ is unique.
\end{corollary}
\begin{proof}
Let $\{F_\alpha \colon \alpha \in \LP\}$ and $\{\tilde{F}_\alpha \colon \alpha \in \LP\}$ 
be two collections satisfying the properties listed in the assertion. 
Then, for any $\alpha \in \LP_N$, the difference $F_\alpha - \tilde{F}_\alpha$ has the asymptotics property
\begin{align*}
\lim_{x_j , x_{j+1} \to \xi} 
\frac{(F_\alpha-\tilde{F}_{\alpha})(x_1 , \ldots , x_{2N})}{(x_{j+1} - x_j)^{-2h}} = 0 ,
\qquad \text{for all } j \in \{ 2, \ldots , 2N-1\} \text{ and } \xi \in (x_{j-1}, x_{j+2}) ,
\end{align*}
so Theorem~\ref{thm::purepartition_unique} shows that $F_\alpha - \tilde{F}_\alpha \equiv 0$.
The asserted uniqueness follows.
\end{proof}

\subsubsection{Hypoellipticity}

Following~\cite[Lemma~5]{Dubedat:SLE_and_Virasoro_representations_localization}, we prove next 
that any distributional solution 
to the system~$\mathrm{(PDE)}$~\eqref{eq: multiple SLE PDEs} is necessarily smooth. 
This holds by the fact that any PDE of type~\eqref{eq: multiple SLE PDEs}
is hypoelliptic, for it satisfies the H\"ormander bracket condition.  
For details concerning hypoelliptic PDEs, see, e.g.,~\cite[Chapter~7]{StroockPDEs},
and for general theory of distributions, e.g.,~\cite[Chapters~6--7]{RudinFA}, or~\cite{HormanderFA}.

\smallbreak

For an open set $O \subset \mathbb{R}^n$ and $\mathbb{F} \in \{ \mathbb{R} ,\mathbb{C}\}$, 
we denote by $C^\infty(O;\mathbb{F})$ the set of smooth functions from $O$ to $\mathbb{F}$.
We also denote by $\Test_space(O;\mathbb{F})$ the space of smooth compactly supported functions 
from $O$ to $\mathbb{F}$, and by $\Distr_space(O;\mathbb{F})$ the space of distributions, that is, 
the dual space of $\Test_space(O;\mathbb{F})$ consisting of continuous linear functionals $\Test_space(O;\mathbb{F}) \to \mathbb{F}$. 
We recall that any locally integrable (e.g., continuous)
function $f$ on $O$ defines a distribution, also denoted by $f \in \Distr_space(O;\mathbb{F})$, 
via the assignment
\begin{align}\label{eq::Fdistribution}
\langle f, \phi \rangle 
:= \; & \int_{O} f(\boldsymbol{x}) \phi (\boldsymbol{x}) \ud \boldsymbol{x} ,
\end{align}  
for all test functions $\phi \in \Test_space(O;\mathbb{F})$.
Furthermore, with this identification, the space $\Test_space(O;\mathbb{F}) \ni f$ of test functions is a dense subset 
in the space $\Distr_space(O;\mathbb{F})$ of distributions (see, e.g.,~\cite[Lemma 1.13.5]{TaoEpsilon}).
We also recall that any differential operator $\mathcal{D}$ defines a linear operator on the space of distributions
via its transpose (dual operator) $\mathcal{D}^*$: for a distribution $f \in \Distr_space(O;\mathbb{F})$,
we have $\mathcal{D} f \in \Distr_space(O;\mathbb{F})$ and
\begin{align} \label{eq::Diffdistribution}
\langle \mathcal{D} f, \phi \rangle := \int_{O} f(\boldsymbol{x}) \; (\mathcal{D}^{(i)})^* \phi (\boldsymbol{x}) \ud \boldsymbol{x} ,
\end{align}
for all test functions $\phi \in \Test_space(O;\mathbb{F})$.

Let $\mathcal{D}$ be a linear partial differential operator with real analytic coefficients defined on an open set 
$U \subset \mathbb{R}^n$. The operator $\mathcal{D}$ is said to be \textit{hypoelliptic} on $U$ if for every open set $O \subset U$, 
the following holds: if $F \in \Distr_space(O;\mathbb{C})$ satisfies $\mathcal{D} F \in C^\infty(O;\mathbb{C})$, 
then we have $F \in C^\infty(O;\mathbb{C})$.

Given a linear partial differential operator, how to prove that it is hypoelliptic?
For operators of certain form, L.~H\"ormander proved in~\cite{HormanderHypoelliptic}
a powerful characterization for hypoellipticity. Suppose $U \subset \mathbb{R}^n$ is an open set, 
denote $\boldsymbol{x} = (x_1, \ldots, x_n) \in \mathbb{R}^n$, and consider smooth vector fields
\begin{align} \label{eqn::vector fields}
X_j := \sum_{k=1}^n a_{jk}(\boldsymbol{x}) \partial_k , \quad \text{for } j \in \{0,1,\ldots, m\} ,
\end{align}
where $a_{jk} \in C^\infty(U;\mathbb{R})$ are smooth real-valued coefficients.
H\"ormander's theorem gives a characterization for hypoellipticity of partial differential operators of the form
\begin{align} \label{eqn::Hormander differential operator}
\mathcal{D} = \sum_{j=1}^m X_j^2 + X_0 + b(\boldsymbol{x}) ,
\end{align}
where $b \in C^\infty(U;\mathbb{R})$.
Denote by $\mathfrak{g}$ the real Lie algebra generated by the vector 
fields~\eqref{eqn::vector fields}, 
and for $\boldsymbol{x} \in U$, let
$\mathfrak{g}_{\boldsymbol{x}} \subset T_{\boldsymbol{x}} \mathbb{R}^n$ 
be the subspace of the tangent space of $\mathbb{R}^n$
obtained by evaluating the elements of $\mathfrak{g}$ at $\boldsymbol{x}$.
H\"ormander's theorem can be phrased as follows:

\begin{theorem} \label{thm::HormanderHypoellipticTheorem}
\textnormal{\cite[Theorem~1.1]{HormanderHypoelliptic}}
Let $U \subset \mathbb{R}^n$ be an open set and $X_0, \ldots, X_m$ vector fields as in~\eqref{eqn::vector fields}.
If for all $\boldsymbol{x} \in U$, the rank of $\mathfrak{g}_{\boldsymbol{x}}$ equals $n$,
then the operator $\mathcal{D}$ of the form~\eqref{eqn::Hormander differential operator} is hypoelliptic on $U$.
\end{theorem}

Consider now the partial differential operators appearing in the system~$\mathrm{(PDE)}$~\eqref{eq: multiple SLE PDEs}.
They are defined on the open set $\domainofdef_{2N} = \{ (x_1, \ldots, x_{2N}) \in \mathbb{R}^{2N} \colon x_i \neq x_j \text{ for all } i \neq j \}$.
The following result was proved in~\cite[Lemma~5]{Dubedat:SLE_and_Virasoro_representations_localization} in a very general setup. 
For clarity, we give the proof in our simple case.

\begin{proposition}\label{prop::smoothness}
Each partial differential operator 
$\mathcal{D}^{(i)} = \frac{\kappa}{2}\partial^2_i + \sum_{j\neq i}\left(\frac{2}{x_{j}-x_{i}}\partial_j - 
\frac{2h}{(x_{j}-x_{i})^{2}}\right)$, for fixed $i \in \{1,\ldots,2N\}$, is hypoelliptic.
In particular, any distributional solution $F$ to 
$\mathcal{D}^{(i)} F = 0$ is smooth: 
\begin{align*}
\begin{cases}
F \in \Distr_space(\domainofdef_{2N};\mathbb{C}) \\
\langle \mathcal{D}^{(i)} F, \phi \rangle = 0 , \qquad
\text{for all } \phi \in \Test_space(\domainofdef_{2N};\mathbb{C})
\end{cases}
\qquad \Longrightarrow \qquad 
F \in C^\infty(\domainofdef_{2N};\mathbb{C}) .
\end{align*}
\end{proposition}
\begin{proof}
Note that choosing $X_0 = \sum_{j\neq i} \frac{2}{x_{j}-x_{i}}\partial_j$, $X_1 = \sqrt{\frac{\kappa}{2}} \partial_i$,
and 
$b(\boldsymbol{x}) = \sum_{j\neq i} \frac{2h}{(x_{j}-x_{i})^{2}}$,
the operator $\mathcal{D}^{(i)}$ is of the form~\eqref{eqn::Hormander differential operator}.
Thus, by Theorem~\ref{thm::HormanderHypoellipticTheorem}, we only need to check that at any 
$\boldsymbol{x} = (x_1, \ldots, x_{2N}) \in \domainofdef_{2N}$, the vector fields $X_0$ and $X_1$ and their commutators 
at $\boldsymbol{x}$ generate a vector space of dimension $2N$. For this, without loss of generality, we let $i = 1$, and 
consider the $\ell$-fold commutators 
\begin{align*}
X_0^{[0]} := X_0 =  \sum_{j=2}^{2N} \frac{2}{x_{j}-x_{1}}\partial_j  \qquad \text{and} \qquad
X_0^{[\ell]} := \frac{1}{\ell!} \; \big[ \partial_1, X_0^{[\ell-1]} \big] =  \sum_{j=2}^{2N} \frac{2}{(x_{j}-x_1)^{\ell+1}} \partial_j , 
\quad \text{for $\ell \geq 1$.}
\end{align*}
Now, we can write $(X_0^{[0]}, \ldots, X_0^{[2N-2]})^t = 2 \,A \, (\partial_2, \ldots, \partial_{2N})^t$,
where $A = (A_{ij})$ with $A_{ij} = (x_{j}-x_1)^{-i}$ for $i,j, \in \{1,\ldots,2N-1 \}$
is a Vandermonde type matrix, whose determinant is non-zero. 
Thus, we have $\partial_1 = \sqrt{\frac{2}{\kappa}} X_1$
and we can solve for $\partial_2, \ldots, \partial_{2N}$ in terms of $X_0^{[0]}, \ldots, X_0^{[2N-2]}$.
This concludes the proof.
\end{proof}

\begin{remark}
The proof of Proposition~\ref{prop::smoothness} in fact shows that all partial differential operators of the form
\begin{align*}
\frac{\kappa}{2}\partial^2_i + \sum_{j\neq i}\left(\frac{2}{x_{j}-x_{i}}\partial_j - 
\frac{2\Delta_j}{(x_{j}-x_{i})^{2}}\right) ,
\end{align*}
where $i \in \{1,\ldots,2N\}$ and $\Delta_j \in \R$, for all $j \in \{1,\ldots,2N\}$, are hypoelliptic.
\end{remark}

\subsubsection{Dual Elements}

To finish this section, we consider certain linear functionals $\FKdual_\alpha \colon \mathcal{S}_N \to \C$ 
on the solution space $\mathcal{S}_N$ defined in~\eqref{eq: solution space}.
It was proved in the series~\cite{FloresKlebanSolutionSpacePDE1,FloresKlebanSolutionSpacePDE2,FloresKlebanSolutionSpacePDE3,FloresKlebanSolutionSpacePDE4} 
of articles that $\mathrm{dim}\mathcal{S}_N = \Catalan_N$.
The linear functionals $\FKdual_\alpha$ were defined in~\cite{FloresKlebanSolutionSpacePDE1}, 
where they were called ``allowable sequences of limits''
(see also~\cite{KytolaPeoltolaPurePartitionSLE}).
In fact, for each $N$, they form a dual basis for the multiple $\SLE$ pure partition functions 
$\{ \PartF_\alpha \colon \alpha \in \LP_N \}$ --- see Proposition~\ref{prop: FK dual elements}.
To define these linear functionals, we consider a link pattern 
\begin{align*}
\alpha = \{\link{a_1}{b_1},\ldots,\link{a_N}{b_N}\}\in\LP_N
\end{align*}
with its link ordered as $\link{a_1}{b_1},\ldots,\link{a_N}{b_N}$, where we take by convention $a_j < b_j$, for all $j \in \{ 1, \ldots, N\}$.
We consider successive removals of links of the form $\link{j}{j+1}$ from $\alpha$. Recall that the link pattern obtained from $\alpha$
by removing the link $\link{j}{j+1}$ is denoted by $\alpha \removeLink \link{j}{j+1}$, as illustrated in Figure~\ref{fig::link_removal}.
Note that after the removal, the indices of the remaining links have to be relabeled by $1,2,\ldots,2N-2$.
The ordering of links in $\alpha$ is said to be \textit{allowable} if all links of $\alpha$ can be removed in the order
$\link{a_1}{b_1},\ldots,\link{a_N}{b_N}$ in such a way that at each step, the link to be removed connects two consecutive indices,
as illustrated in Figure~\ref{fig::allowable_order}
(see, e.g.,~\cite[Section~3.5]{KytolaPeoltolaPurePartitionSLE} for a more formal definition). 

\begin{figure}[h]
\bigskip
\includegraphics[width=\textwidth]{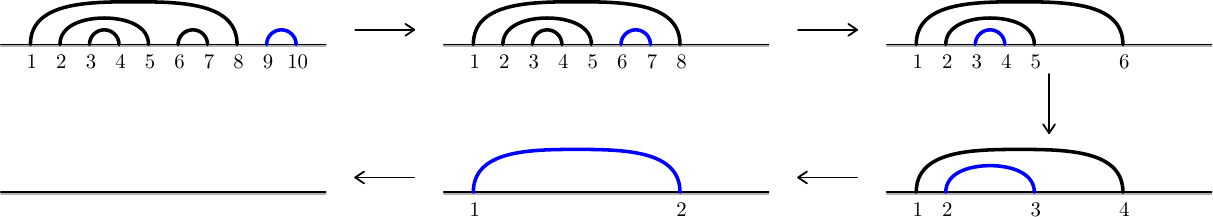} 
{\caption{ 
An allowable ordering of links in a link pattern $\alpha$ and the corresponding link removals.
}\label{fig::allowable_order}}
\end{figure}

Suppose the ordering $\link{a_1}{b_1},\ldots,\link{a_N}{b_N}$ of the links of $\alpha$ is allowable. 
Fix points $\xi_j \in (x_{a_j-1}, x_{b_j+1})$ for all $j \in \{ 1,\ldots,N \}$, with the convention that $x_0 = -\infty$ and  $x_{2N+1} = +\infty$.
It was proved in~\cite[Lemma~10]{FloresKlebanSolutionSpacePDE1} 
that the following sequence of limits exists and is finite for any solution $F \in \mathcal{S}_N$:
\begin{align}\label{eq: limit operation}
\FKdual_\alpha (F )
:= \lim_{x_{a_N},x_{b_{N}}\to\xi_{N}}
\cdots
\lim_{x_{a_{1}},x_{b_{1}}\to\xi_{1}}
(x_{b_{N}}-x_{a_{N}})^{2h}
\cdots
(x_{b_{1}}-x_{a_{1}})^{2h} \, F(x_1,\ldots,x_{2N}) .
\end{align}
Furthermore, by~\cite[Lemma~12]{FloresKlebanSolutionSpacePDE1}, 
any other allowable ordering of the links of $\alpha$ gives the same limit~\eqref{eq: limit operation}.
Therefore, for each $\alpha \in \LP_N$ with any choice of allowable ordering of links,~\eqref{eq: limit operation} defines a linear functional 
\[ \FKdual_\alpha \colon \mathcal{S}_N \to \C . \]
Finally, it was proved in~\cite[Theorem~8]{FloresKlebanSolutionSpacePDE3} that, for any $\kappa \in (0,8)$, the collection
$\{\FKdual_\alpha \colon \alpha \in \LP_N\}$ is a basis for the dual space $\mathcal{S}_N^*$ of the $\Catalan_N$-dimensional solution space $\mathcal{S}_N$.

\subsection{Combinatorics and Binary Relation ``$\;\KWleq\;$''}
\label{subsec: Combinatorics}
In this section, we 
introduce combinatorial objects closely related to the link patterns $\alpha \in \LP$, and
present properties of them which are needed to complete the proof of Theorem~\ref{thm: sle4 pure pffs} in Section~\ref{sec::pure_pf_for_sle4}.
Results of this flavor appear in~\cite{Kenyon-Wilson:Boundary_partitions_in_trees_and_dimers,
Kenyon-Wilson:Double_dimer_pairings_and_skew_Young_diagrams},
and in~\cite{KKP:Correlations_in_planar_LERW_and_UST_and_combinatorics_of_conformal_blocks} 
for the context of pure partition functions.
We follow the notations and conventions of the latter reference. 

\textit{Dyck paths} are walks on $\Znn$ with steps of length one, starting and ending at zero.
For $N \geq 1$, we denote the set of all Dyck paths of $2N$ steps by
\begin{align*}
\DP_N := \big\{ \alpha \colon \{ 0,1, \ldots,2N \} \rightarrow \Z_{\ge 0} \, \colon \, \alpha(0) = \alpha(2N) = 0, 
\text{ and } | \alpha(k) - \alpha(k-1) | = 1, \text{ for all } k \big\} .
\end{align*}
To each link pattern $\alpha \in \LP_N$, we associate a Dyck path, also denoted by 
$\alpha \in \DP_N$, as follows. We write $\alpha$ as an ordered collection
\begin{align} \label{eq: begin and end point convention sec 2}
\alpha = \{ \link{a_1}{b_1}, \ldots, \link{a_N}{b_N} \} ,
\qquad \text{where } a_1 < a_2 < \cdots < a_N \text{ and }
a_j < b_j , \text{ for all } j \in \{ 1, \ldots, N \} .
\end{align}
Then, we set $\alpha(0) = 0$ and, for all $k \in \{1, \ldots, 2N\}$, we set
\begin{align} \label{eq::DP} 
\alpha(k) = \begin{cases}
\alpha(k-1) + 1 , & \text{if } k = a_r \text{ for some } r, \\
\alpha(k-1) - 1 , & \text{if } k = b_s \text{ for some } s .
\end{cases}
\end{align}
Indeed, this defines a Dyck path $\alpha \in \DP_N$. Conversely, for any Dyck path 
$\alpha \colon \{ 0,1, \ldots,2N \} \rightarrow \Z_{\ge 0}$, we associate a link pattern $\alpha$
by associating to each up-step (i.e., step away from zero) an index $a_r$, for $r = 1,2,\ldots,N$, and 
to each down-step (i.e., step towards zero) an index $b_s$, for $s = 1,2,\ldots,N$, and setting 
$\alpha := \{ \link{a_1}{b_1}, \ldots, \link{a_N}{b_N} \}$. 
These two mappings $\LP_N \to \DP_N$ and $\DP_N \to \LP_N$ define a bijection between the sets of 
link patterns and Dyck paths, illustrated in Figure~\ref{fig::bijection}.
We thus identify the elements $\alpha$ of these two sets and use the indistinguishable notation
$\alpha \in \LP_N$ and $\alpha \in \DP_N$ for both.

\begin{figure}[h!]
\bigskip
\floatbox[{\capbeside\thisfloatsetup{capbesideposition={right,center},capbesidewidth=0.4\textwidth}}]{figure}[\FBwidth]
{\caption{Illustration of the bijection $\LP_N \leftrightarrow \DP_N$, 
identifying link patterns and Dyck paths
for $\alpha = \{\link{1}{10}, \link{2}{5}, \link{3}{4}, \link{6}{7}, \link{8}{9} \}$. 
Both the link pattern (top) and the Dyck path (bottom) are 
denoted~by~$\alpha$.}\label{fig::bijection}}
{\includegraphics[width=0.4\textwidth]{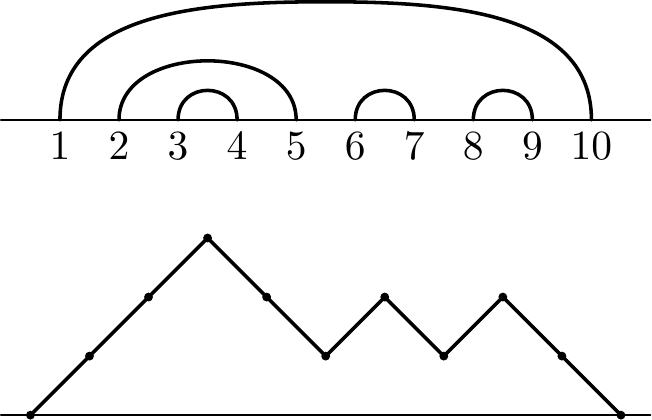} \qquad \quad}
\end{figure}

\begin{figure}[h!]
\bigskip
\begin{center}
\includegraphics[width=.3\textwidth]{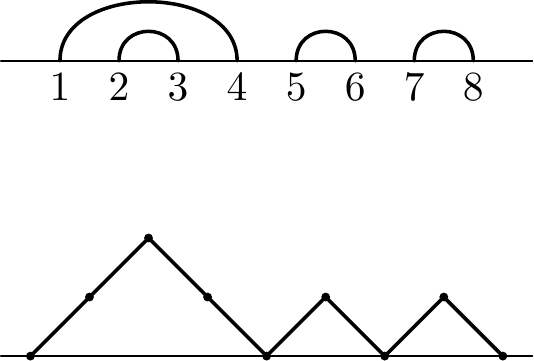} \qquad \qquad \qquad
\includegraphics[width=.3\textwidth]{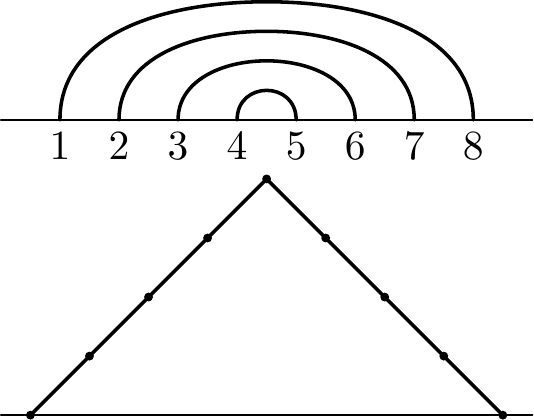}
\end{center}
\caption{These two link patterns are comparable in the partial order $\DPleq$, 
but incomparable in the binary relation $\KWleq$:
the left link pattern is $\alpha = \{ \link{1}{4}, \link{2}{3}, \link{5}{6}, \link{7}{8} \}$ and
the right link pattern is $\beta = \{ \link{1}{8}, \link{2}{7}, \link{3}{6}, \link{4}{5} \}$.}
\label{fig::not_comparable} 
\end{figure}


These sets have a natural partial order $\DPleq$ measuring how nested their elements are: we define
\begin{align}\label{eq: partial order}
\alpha \DPleq \beta  \qquad \text{ if and only if } \qquad
 \alpha(k) \le \beta(k) , \text{ for all } k \in \{0, 1, \ldots, N\} .
\end{align}
For instance, the rainbow link pattern $\nested_N$ is maximally nested --- it is the largest element 
in this partial order. 
In fact, the partial order $\DPleq$ is the transitive closure of a binary relation 
which was introduced by R.~Kenyon and D.~Wilson 
in~\cite{Kenyon-Wilson:Boundary_partitions_in_trees_and_dimers,
Kenyon-Wilson:Double_dimer_pairings_and_skew_Young_diagrams}
and K.~Shigechi and P.~Zinn-Justin 
in~\cite{Shigechi-Zinn:Path_representation_of_maximal_parabolic_Kazhdan-Lusztig_polynomials}. 
We give a definition for this binary relation $\KWleq$ 
that we have found the most suitable to the purposes of the present article. 
We refer to 
\cite[Section~2]{KKP:Correlations_in_planar_LERW_and_UST_and_combinatorics_of_conformal_blocks}
for a detailed survey of this binary relation and many equivalent definitions of it;
see also Figure~\ref{fig::not_comparable} for an example.
We define $\KWleq$ as follows:

\begin{definition}\label{def: KW relation}
\textnormal{\cite[Lemma~2.5]{KKP:Correlations_in_planar_LERW_and_UST_and_combinatorics_of_conformal_blocks}}
Let $\alpha = \{ \link{a_1}{b_1}, \ldots, \link{a_N}{b_N}  \} \in \LP_N$ be ordered as 
in~\eqref{eq: begin and end point convention sec 2}. Let $\beta \in \LP_N$.
Then, $\alpha \KWleq \beta$ if and only if there exists a permutation 
$\sigma \in \mathfrak{S}_N$ such that 
\[ \beta = \{ \link{a_1}{b_{\sigma(1)}}, \ldots, \link{a_N}{b_{\sigma(N)}} \} . \]
For each $N \geq 1$, the incidence matrix $\Mmat$ of this relation on the set 
$\LP_N \leftrightarrow \DP_N$ is the $\Catalan_N \times \Catalan_N$ matrix 
$\Mmat = (\Mmat_{\alpha,\beta})$ whose matrix elements are
\begin{align*}
\Mmat_{\alpha,\beta} = \one\{\alpha \KWleq \beta \} =
\begin{cases}
    1, & \text{ if } \alpha \KWleq \beta, \\
    0, & \text{ otherwise} .
\end{cases}
\end{align*}
\end{definition}

In order to state and prove Theorem~\ref{thm: sle4 pure pffs} in Section~\ref{sec::pure_pf_for_sle4}, we 
have to invert the matrix $\Mmat$.
For this purpose, we need some more combinatorics, related to skew-Young diagrams
and their tilings. Let $\alpha \DPleq \beta$. 
When the two Dyck paths $\alpha, \beta \in \DP_N$ are drawn on the same coordinate system, 
their difference forms a (rotated) skew Young diagram, denoted by $\alpha / \beta$,
which can be thought of as a union of atomic squares --- see Figure~\ref{fig::atomic_squares}.
We denote by $|\alpha / \beta|$ the number of atomic square tiles in the skew Young diagram 
$\alpha / \beta$.

\begin{figure}[h!]
\bigskip
\begin{center}
\includegraphics[width=.3\textwidth]{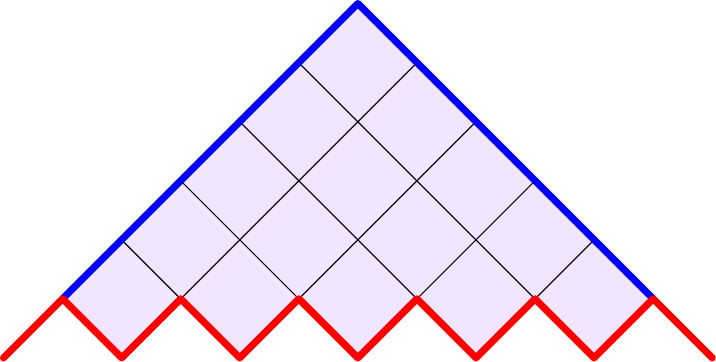} \quad
\includegraphics[width=.3\textwidth]{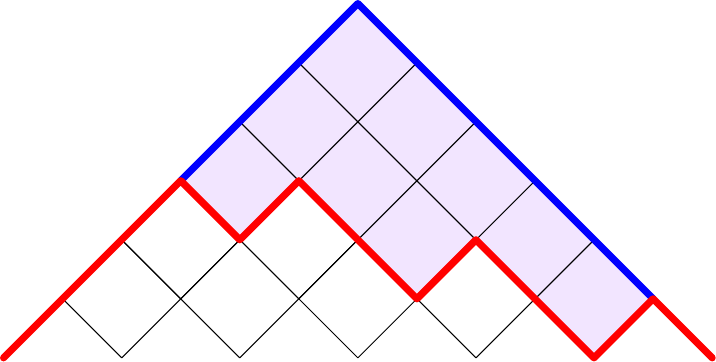} \quad
\includegraphics[width=.3\textwidth]{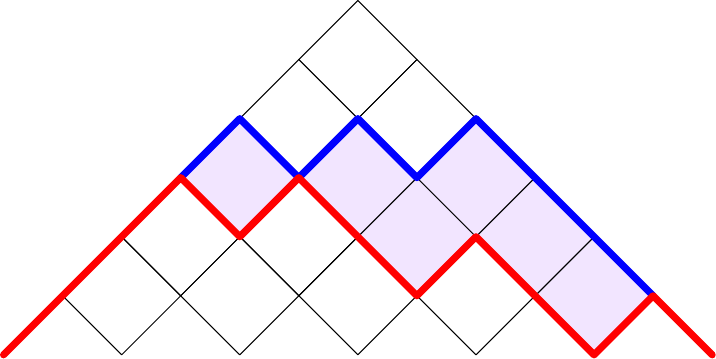} 

\bigskip
\bigskip

\includegraphics[width=.3\textwidth]{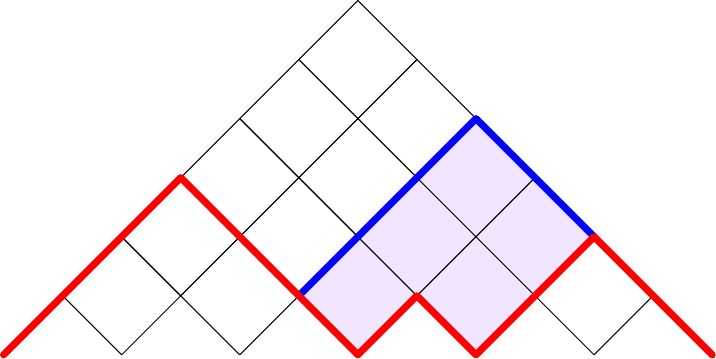} \quad
\includegraphics[width=.3\textwidth]{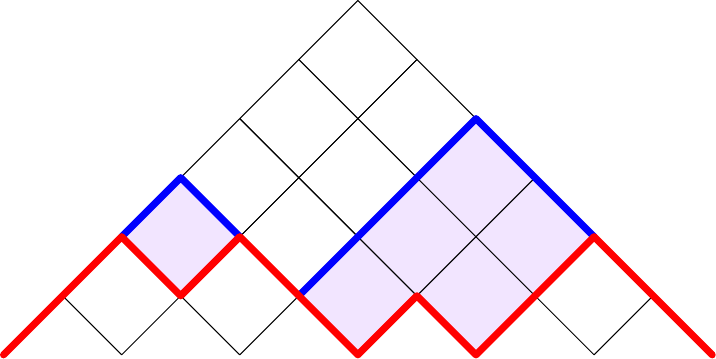} 
\end{center}
\caption{Skew Young diagrams $\alpha / \beta$. 
The smaller Dyck path $\alpha$ (resp. larger~$\beta$) is red (resp. blue).}
\label{fig::atomic_squares} 
\end{figure}

Consider then tilings of the skew Young diagram $\alpha / \beta$. The atomic square tiles form
one possible tiling of $\alpha / \beta$, a rather trivial one. In this article, following the 
terminology of~\cite{Kenyon-Wilson:Boundary_partitions_in_trees_and_dimers,
Kenyon-Wilson:Double_dimer_pairings_and_skew_Young_diagrams, 
KKP:Correlations_in_planar_LERW_and_UST_and_combinatorics_of_conformal_blocks}, 
we consider tilings of $\alpha / \beta$ by Dyck tiles, called Dyck tilings.
A \textit{Dyck tile} is a non-empty union of atomic squares, where the midpoints of the squares 
form a shifted Dyck path, see Figure~\ref{fig::Dyck_tiles}.
Note that also an atomic square is a Dyck tile.
A \textit{Dyck tiling} $T$ of a skew Young diagram $\alpha / \beta$ is a
collection of non-overlapping Dyck tiles whose union is $\bigcup T = \alpha / \beta$. 
Dyck tilings are also illustrated in Figure~\ref{fig::Dyck_tiles}.

\begin{figure}[h!]
\bigskip
\begin{center}
\includegraphics[width=.3\textwidth]{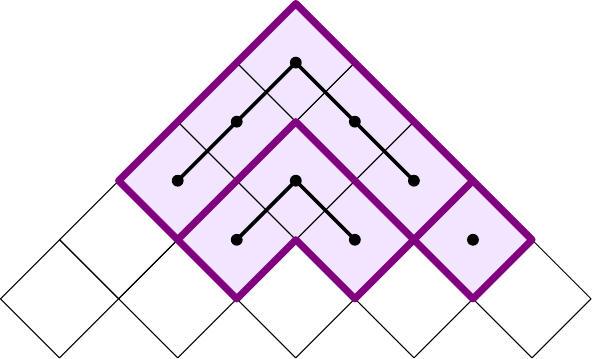} \quad
\includegraphics[width=.3\textwidth]{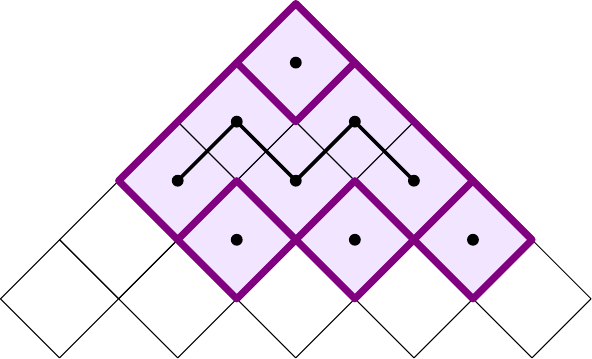} \quad
\includegraphics[width=.3\textwidth]{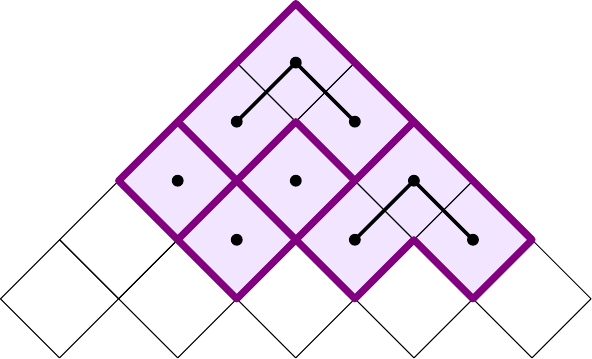} 

\bigskip
\bigskip

\includegraphics[width=.3\textwidth]{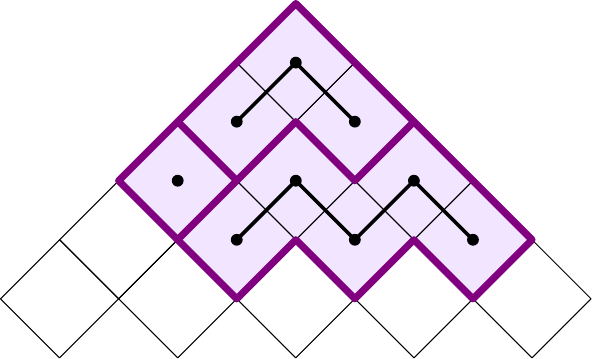} \quad
\includegraphics[width=.3\textwidth]{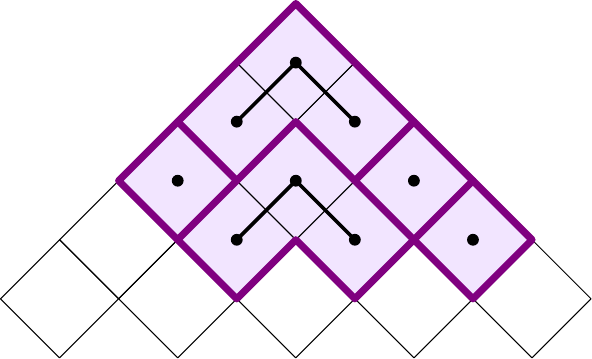} 
\end{center}
\caption{Examples of Dyck tilings, that is, tilings of a skew Young diagrams $\alpha / \beta$ 
by Dyck tiles.}
\label{fig::Dyck_tiles} 
\end{figure}


The \textit{placement} of a Dyck tile $t$ is given by the integer coordinates $(x_t,h_t)$ of 
the bottom left position of $t$, that is, the midpoint of the bottom left atomic square of $t$. 
If $(x'_t,h_t)$ is the bottom right position of $t$, 
we call the closed interval $[x_t , x'_t] \subset \R$ the \textit{horizontal extent} of $t$
--- see Figure~\ref{fig::Dyck_tiles_with_horizontal_extents} for an illustration.

A Dyck tile $t_1$ is said to \textit{cover} a Dyck tile $t_2$
if $t_1$ contains an atomic square which is an upward vertical
translation of some atomic square of $t_2$.
A Dyck tiling $T$ of $\alpha / \beta$ is said to be \textit{cover-inclusive} if
for any two distinct tiles of $T$, either the horizontal extents are disjoint, or the tile 
that covers the other has horizontal extent contained in the horizontal extent of the other.
See Figures~\ref{fig::Dyck_tiles}~and~\ref{fig::Dyck_tiles_with_horizontal_extents} 
for illustrations.

\begin{figure}[h!]
\bigskip
\begin{center}
\includegraphics[width=.3\textwidth]{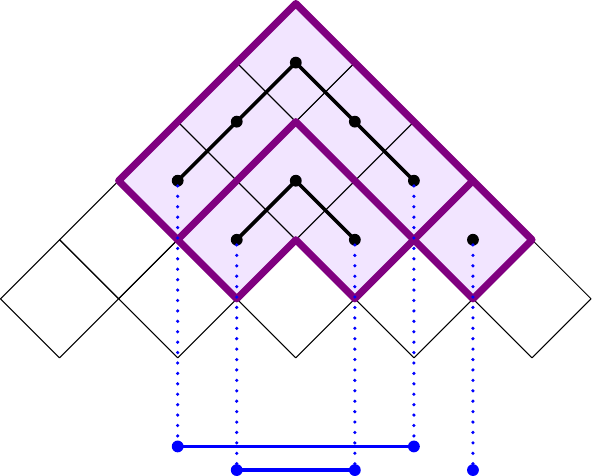} \quad
\includegraphics[width=.3\textwidth]{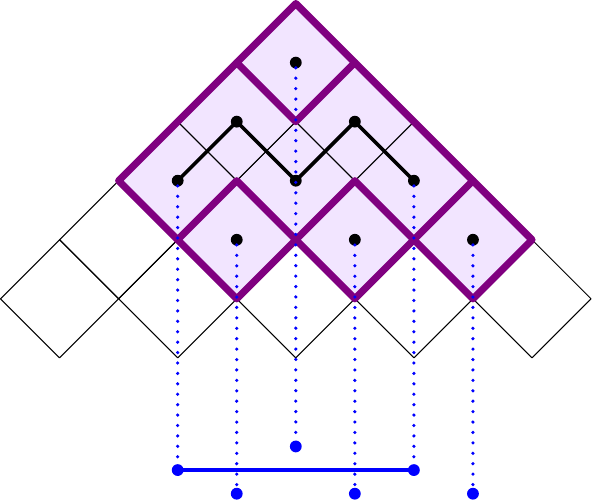} \quad
\includegraphics[width=.3\textwidth]{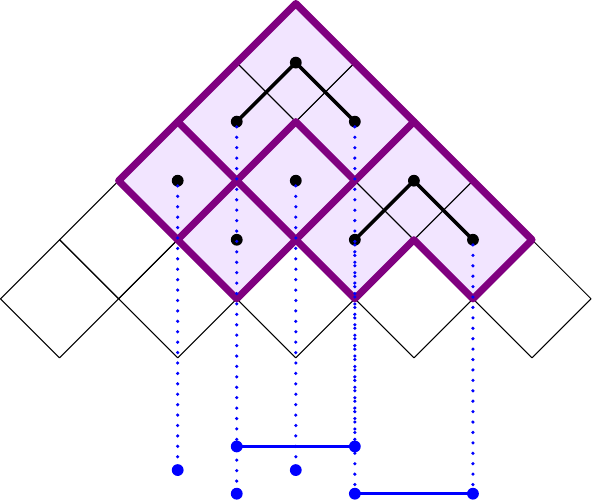} 

\bigskip
\bigskip

\includegraphics[width=.3\textwidth]{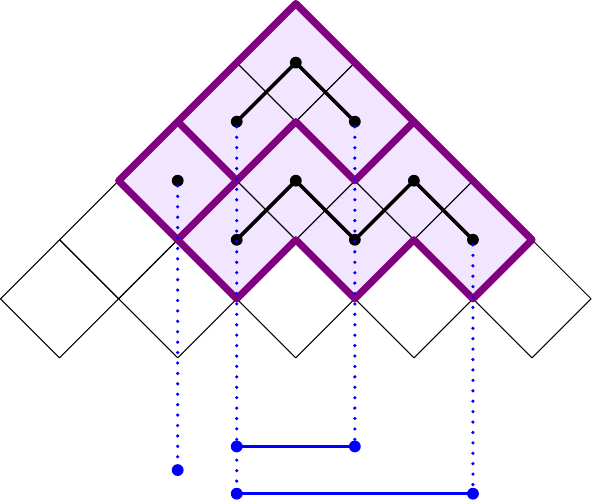} \quad
\includegraphics[width=.3\textwidth]{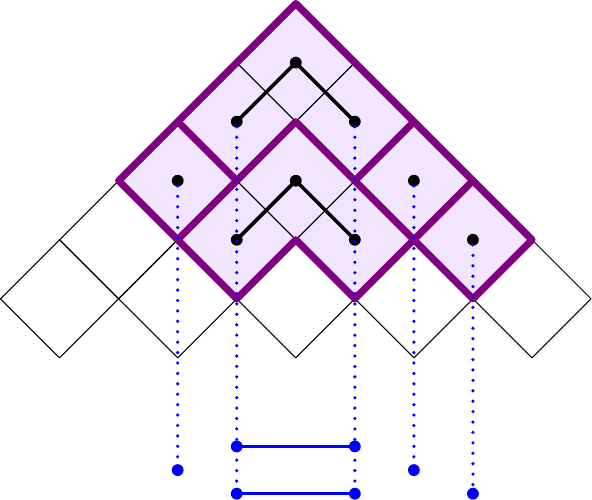} 
\end{center}
\caption{Examples of Dyck tilings with their horizontal extents illustrated. 
The two on the second row are cover-inclusive, but the three on the first row are not.}
\label{fig::Dyck_tiles_with_horizontal_extents} 
\end{figure}

\smallbreak

After these preparations, we are now ready to recall from 
\cite{Kenyon-Wilson:Double_dimer_pairings_and_skew_Young_diagrams,
KKP:Correlations_in_planar_LERW_and_UST_and_combinatorics_of_conformal_blocks} 
the following result, which enables us to write an explicit formula for the pure partition functions
for $\kappa = 4$ in Theorem~\ref{thm: sle4 pure pffs}.
\begin{proposition}\label{prop: matrix inversion}
The matrix $\Mmat$ is invertible with inverse given by 
\begin{align*}
\Minv_{\alpha,\beta} = \begin{cases}
    (-1)^{|\alpha / \beta|} \# \CItilingsof (\alpha / \beta) , & \text{ if } \alpha \DPleq \beta , \\
    0 , & \text{ otherwise} ,
\end{cases}
\end{align*}
where $|\alpha / \beta|$ is the number of atomic square tiles in the skew Young diagram $\alpha / \beta$,
and $\# \CItilingsof (\alpha / \beta)$ denotes the number of cover-inclusive Dyck tilings of 
$\alpha / \beta$, with the convention that $\# \CItilingsof (\alpha / \alpha) = 1$.
\end{proposition}
\begin{proof}
This follows immediately 
from~\cite[Theorem~2.9]{KKP:Correlations_in_planar_LERW_and_UST_and_combinatorics_of_conformal_blocks} 
with tile weight $-1$. Originally, the proof appears 
in~\cite[Theorems~1.5~and~1.6]{Kenyon-Wilson:Double_dimer_pairings_and_skew_Young_diagrams}.
\end{proof}

The entries $\Minv_{\alpha,\beta}$ are always integers, 
and the diagonal entries are all equal to one: $\Minv_{\alpha,\alpha} = 1$ for all~$\alpha$.
Thus, the formula~\eqref{eq: Pure partition function for kappa 4} in Theorem~\ref{thm: sle4 pure pffs} is lower-triangular
in the partial order $\DPgeq$. For instance, we have $\PartF_{\nested_N} = \CobloF_{\nested_N}$ for 
the rainbow link pattern. In Tables~\ref{tab::Matrix_ex1} and~\ref{tab::Matrix_ex2}, we give
examples of the matrix $\Mmat$ and its inverse $\Minv$.

\begin{table}[h!]
\begin{displaymath}
\begin{tabular}{c | cc}
$\LP_N$ with $N=2$
& $\vcenter{\hbox{\includegraphics[scale=0.2]{figures/link-2.pdf}}}$ 
& $\vcenter{\hbox{\includegraphics[scale=0.2]{figures/link-1.pdf}}}$ \\
\hline
$\vcenter{\hbox{\includegraphics[scale=0.2]{figures/link-2.pdf}}}$ & 1 & 0  \\
$\vcenter{\hbox{\includegraphics[scale=0.2]{figures/link-1.pdf}}}$ & 1 & 1  \\
\end{tabular}
\qquad \qquad
\begin{tabular}{c | cc}
$\LP_N$ with $N=2$
& $\vcenter{\hbox{\includegraphics[scale=0.2]{figures/link-2.pdf}}}$ 
& $\vcenter{\hbox{\includegraphics[scale=0.2]{figures/link-1.pdf}}}$ \\
\hline
$\vcenter{\hbox{\includegraphics[scale=0.2]{figures/link-2.pdf}}}$ & 1 & 0  \\
$\vcenter{\hbox{\includegraphics[scale=0.2]{figures/link-1.pdf}}}$ & -1 & 1  \\
\end{tabular}
\end{displaymath}
\caption{\label{tab::Matrix_ex1}
The matrix elements of $\Mmat$ (left) and $\Minv$ (right) for $N = 2$.}
\end{table}


\begin{table}[h!]
\begin{displaymath}
\begin{tabular}{c | ccccc}
$\LP_N$ with $N=3$
& $\vcenter{\hbox{\includegraphics[scale=0.2]{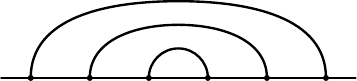}}}$ 
& $\vcenter{\hbox{\includegraphics[scale=0.2]{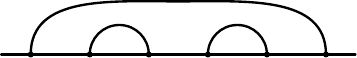}}}$ 
& $\vcenter{\hbox{\includegraphics[scale=0.2]{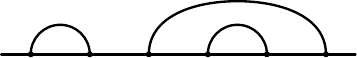}}}$ 
& $\vcenter{\hbox{\includegraphics[scale=0.2]{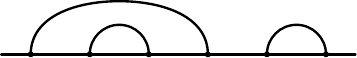}}}$ 
& $\vcenter{\hbox{\includegraphics[scale=0.2]{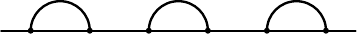}}}$ \\
\hline
$\vcenter{\hbox{\includegraphics[scale=0.2]{figures/link-7.pdf}}}$ & 1 & 0 & 0 & 0 & 0  \\
$\vcenter{\hbox{\includegraphics[scale=0.2]{figures/link-6.pdf}}}$ & 1 & 1 & 0 & 0 & 0  \\
$\vcenter{\hbox{\includegraphics[scale=0.2]{figures/link-5.pdf}}}$ & 0 & 1 & 1 & 0 & 0  \\
$\vcenter{\hbox{\includegraphics[scale=0.2]{figures/link-4.pdf}}}$ & 0 & 1 & 0 & 1 & 0  \\
$\vcenter{\hbox{\includegraphics[scale=0.2]{figures/link-3.pdf}}}$ & 1 & 1 & 1 & 1 & 1  \\
\end{tabular}
\end{displaymath}
\bigskip
\begin{displaymath}
\begin{tabular}{c | ccccc}
$\LP_N$ with $N=3$
& $\vcenter{\hbox{\includegraphics[scale=0.2]{figures/link-7.pdf}}}$ 
& $\vcenter{\hbox{\includegraphics[scale=0.2]{figures/link-6.pdf}}}$ 
& $\vcenter{\hbox{\includegraphics[scale=0.2]{figures/link-5.pdf}}}$ 
& $\vcenter{\hbox{\includegraphics[scale=0.2]{figures/link-4.pdf}}}$ 
& $\vcenter{\hbox{\includegraphics[scale=0.2]{figures/link-3.pdf}}}$ \\
\hline
$\vcenter{\hbox{\includegraphics[scale=0.2]{figures/link-7.pdf}}}$ & 1 & 0 & 0 & 0 & 0  \\
$\vcenter{\hbox{\includegraphics[scale=0.2]{figures/link-6.pdf}}}$ & -1 & 1 & 0 & 0 & 0  \\
$\vcenter{\hbox{\includegraphics[scale=0.2]{figures/link-5.pdf}}}$ & 1 & -1 & 1 & 0 & 0  \\
$\vcenter{\hbox{\includegraphics[scale=0.2]{figures/link-4.pdf}}}$ & 1 & -1 & 0 & 1 & 0 \\
$\vcenter{\hbox{\includegraphics[scale=0.2]{figures/link-3.pdf}}}$ & -2 & 1 & -1 & -1 & 1  \\
\end{tabular}
\end{displaymath}
\caption{\label{tab::Matrix_ex2}
The matrix elements of $\Mmat$ (top) and $\Minv$ (bottom) for $N = 3$.}
\end{table}

\smallbreak

To finish this preliminary section, we introduce notation for certain combinatorial operations on
Dyck paths and summarize results about them that are needed to complete the proof of
Theorem~\ref{thm: sle4 pure pffs} in Section~\ref{sec::pure_pf_for_sle4}.
In the bijection $\LP_N \leftrightarrow \DP_N$ illustrated in Figure~\ref{fig::bijection},
a link between $j$ and $j+1$ in $\alpha \in \LP_N$ corresponds with
an up-step followed by a down-step in the Dyck path $\alpha$, so $\link{j}{j+1} \in \alpha$
is equivalent to $j$ being a local maximum of the Dyck path $\alpha \in \DP_N$.
In this situation, we denote $\upwedgeat{j} \in \alpha$ and we say that $\alpha$ has 
an \textit{up-wedge} at $j$. 
\textit{Down-wedges} $\downwedgeat{j}$ are defined analogously, and an unspecified 
local extremum is called a \textit{wedge} $\wedgeat{j}$. Otherwise, we say that $\alpha$ 
has a \textit{slope} at $j$, denoted by $\slopeat{j} \in \alpha$.
When $\alpha$ has a down-wedge, $\downwedgeat{j} \in \alpha$, we define 
the \textit{wedge-lifting operation}  $\alpha \mapsto \alpha \wedgelift{j}$ by letting 
$\alpha \wedgelift{j}$ be the Dyck path obtained by converting the down-wedge $\downwedgeat{j}$ 
in $\alpha$ into an up-wedge $\upwedgeat{j}$. 

We recall that, if a link pattern $\alpha \in \LP_N$ has a link $\link{j}{j+1} \in \alpha$,
then we denote by $\alpha \removeLink \link{j}{j+1} \in \LP_{N-1}$ the link pattern obtained 
from $\alpha$ by removing the link $\link{j}{j+1}$ and relabeling 
the remaining indices by $1, 2, \ldots, 2N - 2$ (see Figure~\ref{fig::link_removal}).
In terms of the Dyck path, this operation is called an up-wedge removal and denoted by 
$\alpha \removeupwedge{j} \in \DP_{N-1}$. For Dyck paths, we can define a completely analogous 
down-wedge removal $\alpha \mapsto \alpha \removedownwedge{j}$.
Occasionally, it is not important to specify the type of wedge that is removed, so
whenever $\alpha$ has either type of local extremum at $j$ 
(that is, $\wedgeat{j} \in \alpha$), we denote by $\alpha \removewedge{j} \in \DP_{N-1}$ the two steps
shorter Dyck path obtained by removing the two steps around $\wedgeat{j}$, 
see Figure~\ref{fig::wedge_removal}.


\begin{figure}[h!]
\bigskip
\begin{center}
\includegraphics[width=\textwidth]{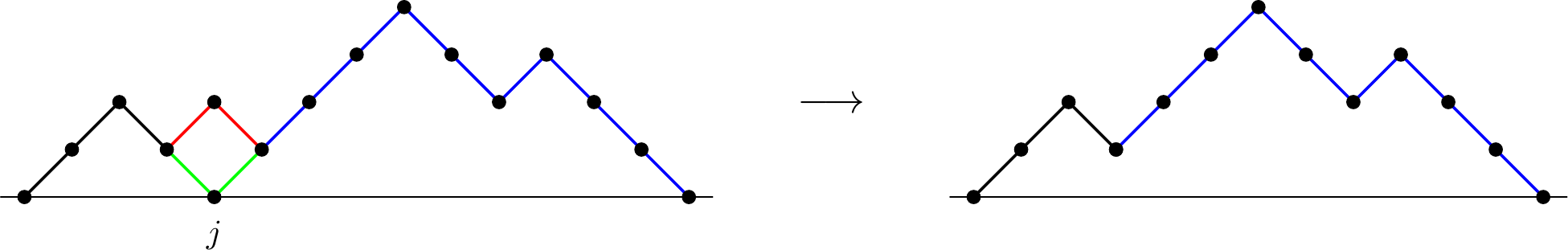}
\end{center}
\caption{The removal of a wedge from a Dyck path.
The left figure is the Dyck path $\alpha \in \DP_N$ and the right figure the shorter Dyck path
$\alpha \removewedge{j} \in \DP_{N-1}$, with $j=4$ and $N=7$.}
\label{fig::wedge_removal} 
\end{figure}


\begin{lemma} \label{lem: combinatorics} 
The following statements hold for Dyck paths $\alpha, \beta \in \DP_N$.

\begin{itemize}
\item[(a):] Suppose $\upwedgeat{j} \not \in \alpha$ and $\downwedgeat{j} \in \beta$. 
Then, we have $\alpha \DPleq \beta$ if and only if $\alpha \DPleq \beta \wedgelift{j}$.

\item[(b):] Suppose $\upwedgeat{j} \not \in \alpha$. 
Then the Dyck paths $\beta \in \DP_N$ such that $\beta \DPgeq \alpha$ and $\wedgeat{j} \in \beta$ 
come in pairs, one containing an up-wedge and the other a down-wedge at $j$:
\begin{align*}
\{ \beta \in \DP_N \colon \beta \DPgeq \alpha \} 
\; = \; \{ \beta \colon \beta \DPgeq \alpha, \downwedgeat{j} \in \beta \}  \cup 
\{ \beta \wedgelift{j} \colon \beta \DPgeq \alpha, \downwedgeat{j} \in \beta \}  \cup 
\{ \beta \colon \beta \DPgeq \alpha, \slopeat{j} \in \beta \} .
\end{align*}

\item[(c):] Suppose $\upwedgeat{j} \in \beta$. 
Then, we have $\alpha \KWleq \beta$ if and only if $\wedgeat{j} \in \alpha$ and 
$\alpha \removewedge{j} \KWleq \beta \removeupwedge{j}$.

\item[(d):] Suppose 
$\upwedgeat{j} \not \in \alpha$, $\downwedgeat{j} \in \beta$, and $\alpha \DPleq \beta$. 
Then we have $\Minv_{\alpha,\beta} = - \Minv_{\alpha,\beta \wedgelift{j}}$.

\end{itemize}
\end{lemma}
\begin{proof}
Parts (a) and (b) were proved, e.g., 
in~\cite[Lemma~2.11]{KKP:Correlations_in_planar_LERW_and_UST_and_combinatorics_of_conformal_blocks} 
(see also the remark below that lemma). Part (c) was proved, e.g., 
in~\cite[Lemma~2.12]{KKP:Correlations_in_planar_LERW_and_UST_and_combinatorics_of_conformal_blocks}. 
For completeness, we give a short proof for Part (d).
First,~\cite[Lemma~2.15]{KKP:Correlations_in_planar_LERW_and_UST_and_combinatorics_of_conformal_blocks} 
says that 
if $\upwedgeat{j} \not \in \alpha$, $\downwedgeat{j} \in \beta$, and $\alpha \DPleq \beta$,
then we have 
$\# \CItilingsof (\alpha / \beta) = 
\# \CItilingsof \big(\alpha / (\beta \wedgelift{j}) \big)$.
On the other hand, Proposition~\ref{prop: matrix inversion} shows that 
$\Minv_{\alpha,\beta} = (-1)^{|\alpha / \beta|} \# \CItilingsof (\alpha / \beta)$. 
The claim follows from this and the observation that the number of Dyck tiles in a cover-inclusive 
Dyck tiling of $\alpha / (\beta \wedgelift{j})$
is one more than the number of Dyck tiles in a cover-inclusive Dyck tiling of 
$\alpha / \beta$, 
by~\cite[proof of Lemma~2.15]{KKP:Correlations_in_planar_LERW_and_UST_and_combinatorics_of_conformal_blocks}.
\end{proof}

\section{Global Multiple SLEs}
\label{sec::characterization}
Throughout this section, we fix the value of $\kappa\in (0,4]$ and we let $(\Omega; x_1, \ldots, x_{2N})$ be a polygon.
For each link pattern $\alpha = \{\link{a_1}{b_1}, \ldots, \link{a_N}{b_N}\} \in \LP_N$,  
we construct an $N$-$\SLE_{\kappa}$ probability measure $\QQ_{\alpha}^{\#}$ 
on the set $X_0^{\alpha}(\Omega; x_1, \ldots, x_{2N})$ of pairwise disjoint, continuous simple curves 
$(\eta_1, \ldots, \eta_N)$ in $\Omega$ 
such that, for each $j\in\{1,\ldots,N\}$, the curve $\eta_j$ connects $x_{a_j}$ to $x_{b_j}$ according to $\alpha$
(see Proposition~\ref{prop::global_existence}).

In~\cite{KozdronLawlerMultipleSLEs}
M.~Kozdron and G.~Lawler constructed such a probability measure in the special case when the curves 
form the rainbow connectivity, illustrated in Figure~\ref{fig::rainbow}, encoded in the link pattern
$\nested_N = \{ \link{1}{2N}, \link{2}{2N-1}, \ldots, \link{N}{N+1} \}$
(see also \cite[Section~3.4]{Dubedat:Euler_integrals_for_commuting_SLEs}).
The generalization of this construction to the case of any possible topological 
connectivity of the curves, encoded in a general link pattern $\alpha \in \LP_N$,
was stated in Lawler's works~\cite{LawlerPartitionFunctionsSLE, LawlerSLENotes}, but without proof.

In the present article, we give a combinatorial construction,
which appears to agree with~\cite[Section~2.7]{LawlerPartitionFunctionsSLE}.
In contrast to the previous works, we formulate the result focusing on the conceptual definition of the global multiple $\SLE$s, 
instead of just defining them as weighted $\SLE$s. These 
$N$-$\SLE_\kappa$ measures have the defining property that, for each $j\in\{1,\ldots,N\}$,
the conditional law of $\eta_j$ given $\{\eta_1, \ldots, \eta_{j-1}, \eta_{j+1}, \ldots, \eta_N\}$ is 
the $\SLE_{\kappa}$ connecting $x_{a_j}$ and $x_{b_j}$ in the component $\hat{\Omega}_j$ of 
the domain $\Omega \setminus \{\eta_1, \ldots, \eta_{j-1}, \eta_{j+1}, \ldots, \eta_N\}$ having 
$x_{a_j}$ and $x_{b_j}$ on its boundary. 
In subsequent work~\cite{BeffaraPeltolaWuUniquenessGloableMultipleSLEs}, we prove that this property uniquely determines the global multiple $\SLE$ measures.

\begin{figure}[h]
\begin{center}
\includegraphics[scale=.6]{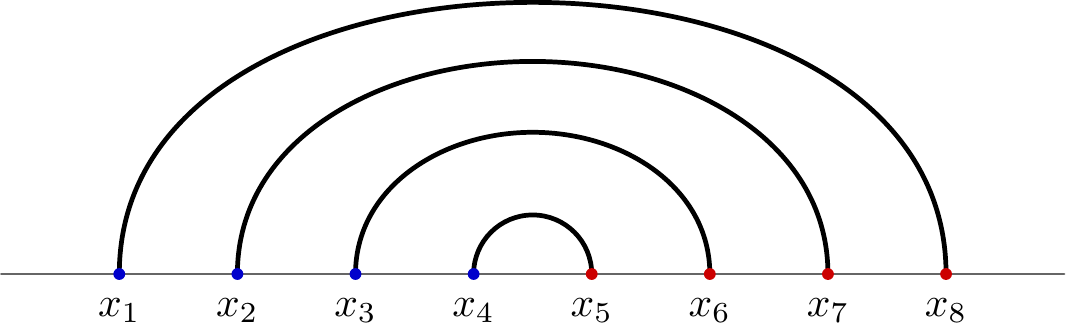}
\quad \quad
\includegraphics[scale=.6]{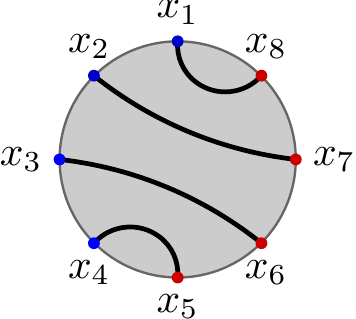}
\end{center}
\caption{The rainbow link pattern with four links, denoted by $\nested_4$.}
\label{fig::rainbow} 
\end{figure}

\subsection{Construction of Global Multiple SLEs}
\label{subsec::multiplesle_existence}
The general idea to construct global multiple $\SLE$s
is that one defines the measure by
its Radon-Nikodym derivative with respect to the product measure
of independent chordal $\SLE$s. This Radon-Nikodym derivative can be written
in terms of the Brownian loop measure. The same idea 
can also be used to construct multiple $\SLE$s in finitely connected domains, 
see~\cite{LawlerPartitionFunctionsSLE,LawlerSLENotes,LawlerSLEMultiplyConnected}.
\smallbreak

Fix $\alpha = \{ \link{a_1}{b_1}, \ldots, \link{a_N}{b_N} \}\in \LP_N$. 
To construct the global $N$-$\SLE_{\kappa}$ associated to $\alpha$,
we introduce a combinatorial expression of Brownian loop measures, denoted by $m_{\alpha}$.
For each configuration $(\eta_1, \ldots, \eta_N) \in X_0^{\alpha}(\Omega; x_1, \ldots, x_{2N})$, 
we note that $\Omega \setminus \{\eta_1, \ldots, \eta_N\}$ has $N+1$ connected components (c.c).
The boundary of each c.c. 
$\LC$ contains some of the curves $\{\eta_1, \ldots, \eta_N\}$. We denote by 
\[ \LB(\LC) := \{j\in \{1,\ldots,N\} \colon \eta_j \subset \partial\LC\} \] 
the set of indices $j$ specified by the curves $\eta_j \subset \partial \LC$. If $\LB(\LC) = \{j_1, \ldots, j_p\}$,
we define
\begin{align*}
m(\LC) 
:= & \; \sum_{\substack{i_1,i_2 \in \LB(\LC), \\ i_1\neq i_2}} 
\mu(\Omega; \eta_{i_1}, \eta_{i_2}) 
- \sum_{\substack{i_1, i_2, i_3 \in \LB(\LC), \\ i_1 \neq i_2 \neq i_3 \neq i_1}} 
\mu(\Omega; \eta_{i_1}, \eta_{i_2}, \eta_{i_3}) 
+ \cdots + (-1)^p \mu(\Omega; \eta_{j_1}, \ldots, \eta_{j_p}).
\end{align*}
For $(\eta_1, \ldots, \eta_N) \in X_0^{\alpha}(\Omega; x_1, \ldots, x_{2N})$, 
we define 
\begin{align}\label{eqn::def_malpha}
m_{\alpha}(\Omega; \eta_1, \ldots, \eta_N) := 
\sum_{\text{c.c. } \LC \text{ of } \Omega \setminus \{\eta_1, \ldots, \eta_N\} } m(\LC). 
\end{align}
If $\alpha$ is the rainbow pattern $\nested_N$, then
the quantity $m_{\alpha}$ has a simple expression:
\[ m_{\nested_N}(\Omega; \eta_1, \ldots, \eta_N) 
= \sum_{j=1}^{N-1} \mu(\Omega; \eta_j, \eta_{j+1}), 
\qquad \text{for }
\nested_N = \{ \link{1}{2N}, \link{2}{2N-1}, \ldots, \link{N}{N+1} \} . \]
More generally, $m_{\alpha}$ is given by an inclusion-exclusion procedure that depends on $\alpha$.
It has the following cascade property, which will be crucial in the sequel.

\begin{lemma} \label{lem::malpha_cascade}
Let $\alpha \in \LP_N$ and $j \in \{1,\ldots, N\}$, and denote\footnote{
We recall that the link pattern obtained from $\alpha$ by removing the link $\link{a}{b}$ 
is denoted by $\alpha \removeLink \link{a}{b}$, and, importantly,
after the removal, the indices of the remaining links relabeled by $1,2,\ldots,2N-2$ (see also Figure~\ref{fig::link_removal}).}
$\hat{\alpha} = \alpha \removeLink \link{a_j}{b_j} \in \LP_{N-1}$. 
Then we have
\[m_{\alpha}(\Omega; \eta_1, \ldots, \eta_N) 
= m_{\hat{\alpha}}(\Omega; \eta_1,\ldots, \eta_{j-1}, \eta_{j+1},\ldots, \eta_N) 
+ \mu(\Omega; \eta_j, \Omega\setminus\hat{\Omega}_j),\]
where $\hat{\Omega}_j$ is the connected component of $\Omega\setminus\{\eta_1,\ldots, \eta_{j-1}, \eta_{j+1}, \ldots, \eta_N\}$ 
having $x_{a_j}$ and $x_{b_j}$ on its boundary.
\end{lemma}
\begin{proof} As illustrated in Figure~\ref{fig::malpha_cascade}, the domain
$\Omega \setminus \{\eta_1, \ldots, \eta_N\}$
has $N+1$ connected components, two of which have the curve
$\eta_j$ on their
boundary. We denote them by $\LC^L_j$ and $\LC^R_j$. 
We split the summation in $m_{\alpha}$ into two parts, depending on whether or not $\eta_j$ is a part of 
the boundary of the c.c.~$\LC$:
\[ m_{\alpha}(\Omega; \eta_1, \ldots, \eta_N) = S_1 + S_2 , 
\qquad \text{where } \quad S_1 = m(\LC^L_j) + m(\LC^R_j) \quad \text{ and } 
\quad S_2 = \sum_{\LC \; : \; j \notin \LB(\LC)} m(\LC).\] 
The quantity $S_1$ is a sum
of terms of the form $\mu(\Omega; \eta_{i_1}, \ldots, \eta_{i_k})$. 
We split the terms in $S_1 = S_{1, 1} + S_{1, 2}$
into two parts: $S_{1, 1}$ is the sum 
of the terms in $S_1$ including $\eta_j$ and $S_{1,2}$ is the sum 
of the terms in $S_1$ excluding $\eta_j$.  
Now we have $m_{\alpha}(\Omega; \eta_1, \ldots, \eta_N)=S_{1,1}+S_{1,2}+S_2$. 

On the other hand, by definition~\eqref{eqn::def_malpha}, 
the quantity $m_{\hat{\alpha}}$ can be written in the form
\[m_{\hat{\alpha}}(\Omega; \eta_1, \ldots, \eta_{j-1}, \eta_{j+1}, \ldots, \eta_N)=S_2+S_{1,2}+S_3,\]
where $S_3$ contains the contribution of terms of type
$\mu(\Omega; \eta_{i_1}, \ldots, \eta_{i_k})$ for curves 
$\eta_{i_1}, \ldots,\eta_{i_k}$ such that $i_1, \ldots, i_k \in \LB(\hat{\Omega}_j)$ and 
at least two of these curves belong to different $\partial\LC^L_j$ and $ \partial\LC^R_j$. 
For such curves, any
Brownian loop intersecting all of them must also intersect $\eta_j$. Thus, we have 
$\mu(\Omega; \eta_j, \Omega\setminus\hat{\Omega}_j)=S_{1,1}-S_3$. The asserted identity follows:
\begin{align*}
m_{\alpha}(\Omega; \eta_1, \ldots, \eta_N) 
&= S_{1,1}-S_3+S_3+S_{1,2}+S_2 \\
&= \mu(\Omega; \eta_j, \Omega\setminus\hat{\Omega}_j) 
+ m_{\hat{\alpha}}(\Omega; \eta_1, \ldots, \eta_{j-1}, \eta_{j+1}, \ldots, \eta_N) .
\end{align*}
\end{proof}

\begin{figure} 
\floatbox[{\capbeside\thisfloatsetup{capbesideposition={right,center},capbesidewidth=0.5\textwidth}}]{figure}[\FBwidth]
{\caption{
Illustration of notations used throughout. \\ 
The red curve is $\eta_j$. 
The domain $\Omega \setminus \{\eta_1, \ldots, \eta_{N}\}$ has $N+1$ connected components. 
Two of them have $\eta_j$ on their boundary, denoted by
$\LC^L_j$ and $\LC^R_j$. The grey domain is $\hat{\Omega}_j$, that is, the connected component of 
$\Omega \setminus \{\eta_1, \ldots, \eta_{j-1}, \eta_{j+1}, \ldots, \eta_N\}$ having 
the endpoints $x_{a_j}, x_{b_j} \in \partial \hat{\Omega}_j$
of the curve $\eta_j$ on its boundary. \\
On the other hand, in Proposition~\ref{prop::GeneralCascadeProp} 
we denote by $\eta = \eta_j$ and $D_{\eta}^L$ and $D_{\eta}^R$ the two connected components
of $\Omega \setminus \eta$ on the left and right of the curve, respectively.
The sub-link patterns of $\alpha$ associated to these two components are denoted by $\alpha^L$
and $\alpha^R$, and illustrated in blue and green in the figure.
}\label{fig::malpha_cascade}}
{\includegraphics[width=0.4\textwidth]{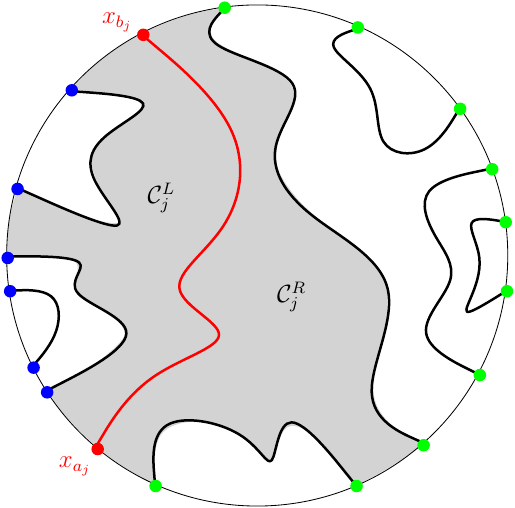} \qquad \quad}
\end{figure}

Next, we record a boundary perturbation property for the quantity $m_\alpha$,
also needed later.
\begin{lemma}\label{lem::malpha_boundaryperturbation}
Suppose $K$ is a relatively compact subset of $\Omega$ 
such that $\Omega\setminus K$ is simply connected, and assume that 
the distance between $K$ and $\{\eta_1, \ldots, \eta_N\}$ is strictly positive. 
Then we have
\begin{align}\label{eq::malpha_boundaryperturbation}
m_{\alpha}(\Omega; \eta_1, \ldots, \eta_N) 
= m_{\alpha}(\Omega\setminus K; \eta_1, \ldots, \eta_N) + 
\sum_{j=1}^N \mu(\Omega; K, \eta_j) - \mu\big(\Omega; K, \bigcup_{j=1}^N \eta_j\big) .
\end{align}
\end{lemma}
\begin{proof}
We prove the asserted identity by induction on $N \ge 1$. For $N=1$, we have 
$m_{\vcenter{\hbox{\includegraphics[scale=0.2]{figures/link-0.pdf}}}}(\Omega;\eta) = 0$, so the claim is clear. 
Assume that~\eqref{eq::malpha_boundaryperturbation}
holds for all link patterns in $\LP_{N-1}$, denote 
$\hat{\alpha} = \alpha \removeLink \link{x_{a_1}}{x_{b_1}} \in \LP_{N-1}$, and let $\eta_1$
be the curve from $x_{a_1}$ to $x_{b_1}$. Finally, 
let $\hat{\Omega}_1$ be the connected component of $\Omega \setminus \{\eta_2,\ldots,\eta_N\}$ 
having the endpoints of $\eta_1$ on its boundary (as in Figure~\ref{fig::malpha_cascade}). 
Using Lemma~\ref{lem::malpha_cascade} and the obvious fact that
$\mu(\Omega; \eta_1, \Omega\setminus\hat{\Omega}_1) 
= \mu(\Omega\setminus K; \eta_1, \Omega\setminus\hat{\Omega}_1)
+ \mu(\Omega; K, \eta_1, \Omega\setminus\hat{\Omega}_1)$, we can write $m_{\alpha}$ in the form
\begin{align*}
m_{\alpha}(\Omega; \eta_1, \ldots ,\eta_N) 
= & \; m_{\hat{\alpha}}(\Omega; \eta_2, \ldots, \eta_N) 
+ \mu(\Omega; \eta_1, \Omega\setminus\hat{\Omega}_1) \\
= & \; m_{\hat{\alpha}}(\Omega; \eta_2, \ldots, \eta_N)
+ \mu(\Omega\setminus K; \eta_1, \Omega\setminus\hat{\Omega}_1)
+ \mu(\Omega; K, \eta_1, \Omega\setminus\hat{\Omega}_1) .
\end{align*}
By the induction hypothesis, for $\hat{\alpha} \in \LP_{N-1}$,
we have
\begin{align*}
m_{\hat{\alpha}}(\Omega; \eta_2, \ldots, \eta_N) 
= m_{\hat{\alpha}}(\Omega\setminus K; \eta_2, \ldots, \eta_N)
+ \sum_{j=2}^N \mu(\Omega; K, \eta_j) - \mu\big(\Omega; K, \bigcup_{j=2}^N \eta_j\big) .
\end{align*}
Combining these two relations with Lemma~\ref{lem::malpha_cascade}, we obtain
\begin{align*}
m_{\alpha}(\Omega; \eta_1, \ldots ,\eta_N)
= \; & m_{\alpha}(\Omega\setminus K; \eta_1,\ldots, \eta_N) 
+ \sum_{j=2}^N \mu(\Omega; K, \eta_j) 
- \mu\big(\Omega; K, \bigcup_{j=2}^N\eta_j\big) + \mu(\Omega; K, \eta_1, \Omega\setminus\hat{\Omega}_1) .
\end{align*}
Note now that 
$\mu(\Omega; K, \eta_1, \Omega\setminus\hat{\Omega}_1) 
= \mu\big(\Omega; K, \eta_1, \bigcup_{j=2}^N\eta_j \big)$, so
\begin{align*}
\mu\big(\Omega; K, \bigcup_{j=1}^N \eta_j \big) 
= & \; \mu(\Omega; K, \eta_1) + \mu \big(\Omega; K, \bigcup_{j=2}^N \eta_j\big) 
- \mu\big(\Omega; K, \eta_1, \bigcup_{j=2}^N\eta_j \big) \\
= & \; \mu(\Omega; K, \eta_1) + \mu\big(\Omega; K, \bigcup_{j=2}^N \eta_j\big) 
- \mu(\Omega; K, \eta_1, \Omega\setminus\hat{\Omega}_1) .
\end{align*}
Combining the above two equations, we get the asserted identity~\eqref{eq::malpha_boundaryperturbation}: 
\[ m_{\alpha}(\Omega; \eta_1, \ldots ,\eta_N) 
= m_{\alpha}(\Omega\setminus K; \eta_1,\ldots, \eta_N) 
+ \sum_{j=2}^N \mu(\Omega; K, \eta_j) - \mu(\Omega; K, \bigcup_{j=1}^N \eta_j) + \mu(\Omega; K, \eta_1) .\] 
\end{proof}

Now, we are ready to construct the probability measure of Theorem~\ref{thm::global_existence}.
\begin{proposition} \label{prop::global_existence}
Let $\kappa\in (0,4]$ and 
let $(\Omega; x_1, \ldots, x_{2N})$ be a polygon. 
For any $\alpha \in \LP_N$, there exists a global $N$-$\SLE_{\kappa}$ associated to~$\alpha$. 
\end{proposition}
\begin{proof}
For $\alpha = \{ \link{a_1}{b_1}, \ldots, \link{a_N}{b_N} \} \in \LP_N$, let $\PP_{\alpha}$ 
denote the product measure
\begin{align*}
\PP_{\alpha}
:= \bigotimes_{j = 1}^{N} \PP(\Omega; x_{a_j}, x_{b_j}) 
\end{align*}
of $N$ independent chordal $\SLE_\kappa$ curves connecting the boundary points 
$x_{a_j}$ and $x_{b_j}$ for $j \in \{ 1,2,\ldots,N\}$ according to the connectivity $\alpha$.
Denote by $\E_{\alpha}$ the expectation with respect to $\PP_{\alpha}$. 
Define $\QQ_{\alpha}$ to be the measure which is absolutely continuous with respect to $\PP_{\alpha}$ 
with Radon-Nikodym derivative 
\begin{align}\label{eqn::def_radon_alpha}
\frac{\ud \QQ_{\alpha}}{\ud \PP_{\alpha}} (\eta_1, \ldots, \eta_N) =
R_{\alpha}(\Omega; \eta_1, \ldots, \eta_N) 
:= \one_{\{\eta_j\cap\eta_k=\emptyset \; \forall \; j\neq k\}} 
\exp(c m_{\alpha}(\Omega; \eta_1, \ldots, \eta_N)).
\end{align}

First, we prove that the total mass $|\QQ_{\alpha}| = \E_{\alpha} [R_{\alpha}(\Omega; \eta_1,\ldots, \eta_N)]$
of $\QQ_{\alpha}$ 
is positive and finite. 
Positivity is clear from the definition~\eqref{eqn::def_radon_alpha}.
We prove the finiteness by induction on $N \ge 1$, using the cascade property of Lemma~\ref{lem::malpha_cascade}.
The initial case $N = 1$ is obvious: $R_{\vcenter{\hbox{\includegraphics[scale=0.2]{figures/link-0.pdf}}}} = 1$.  
Let $N \geq 2$ and assume that
$|\QQ_{\hat{\alpha}}|$ is finite for all $\hat{\alpha} \in \LP_{N-1}$.
Using Lemma~\ref{lem::malpha_cascade}, 
we write the Radon-Nikodym derivative~\eqref{eqn::def_radon_alpha} in the form
\begin{align}\label{eqn::radon_cascade}
R_{\alpha}(\Omega; \eta_1,\ldots, \eta_N) 
= R_{\hat{\alpha}}(\Omega; \eta_1,\ldots, \eta_{j-1}, \eta_{j+1}, \ldots, \eta_N)\times \one_{\{\eta_j\subset \hat{\Omega}_j\}}\exp(c\mu(\Omega; \eta_j, \Omega\setminus \hat{\Omega}_j)) ,
\end{align}
for a fixed $j\in \{1,\ldots,N\}$, where 
$\hat{\alpha} = \alpha \removeLink \link{a_j}{b_j}$. 
Thus, we have
\begin{align*}
\E_{\alpha} [R_{\alpha}(\Omega; \eta_1,\ldots, \eta_N)] 
= \; & 
\E_{\alpha} \big[ \E_{\alpha} \left[R_{\alpha}(\Omega; \eta_1,\ldots, \eta_N) 
\cond \eta_1, \ldots, \eta_{j-1}, \eta_{j+1}, \ldots, \eta_N\right] \big] \\
= \; & 
\E_{\hat{\alpha}} \bigg[R_{\hat{\alpha}}(\Omega; \eta_1,\ldots, \eta_{j-1}, \eta_{j+1}, \ldots, \eta_N) 
\left(\frac{H_{\hat{\Omega}_j}(x_{a_j}, x_{b_j})}{H_{\Omega}(x_{a_j}, x_{b_j})}\right)^h \bigg] 
&&
\text{[by Lemma \ref{lem::sle_boundary_perturbation}]}  \\
\le \; &  \E_{\hat{\alpha}}\left[R_{\hat{\alpha}}(\Omega; \eta_1,\ldots, \eta_{j-1}, \eta_{j+1}, \ldots, \eta_N)\right] 
&&
\text{[by \eqref{eqn::poissonkernel_mono}]}  \\
\le \; &  1 . 
&&
\text{[by ind. hypothesis]}
\end{align*}

Noting that the Radon-Nikodym derivative~\eqref{eqn::def_radon_alpha} also depends on 
the fixed boundary points $x_1, \ldots, x_{2N}$, we define the function
$f_{\alpha}$ of $2N$ variables $x_1, \ldots, x_{2N} \in \partial \Omega$ by
\begin{align}\label{eqn::def_mixingpart}
f_{\alpha}(\Omega; x_1, \ldots, x_{2N}) 
:= \E_{\alpha}[R_{\alpha}(\Omega; \eta_1, \ldots, \eta_N)]= |\QQ_{\alpha}| .
\end{align}
Note that $f_{\alpha}$ is conformally invariant. From the above analysis, we see that it is also bounded:
\begin{align}\label{eq: bound for f alpha}
0 < f_{\alpha} \le 1 .
\end{align}

Second, we show that, for each $j\in\{1,\ldots, N\}$, 
under the probability measure $\QQ_{\alpha}^{\#} := \QQ_{\alpha}/|\QQ_{\alpha}|$,
the conditional law of $\eta_j$ given $\{\eta_1, \ldots, \eta_{j-1}, \eta_{j+1}, \ldots, \eta_N\}$ is 
the $\SLE_{\kappa}$ connecting $x_{a_j}$ and $x_{b_j}$ in the domain $\hat{\Omega}_j$. 
By the cascade property~\eqref{eqn::radon_cascade},
given $\{\eta_1, \ldots, \eta_{j-1}, \eta_{j+1}, \ldots, \eta_N\}$, the conditional law of $\eta_j$ is the same as $\PP(\Omega; x_{a_j}, x_{b_j})$ weighted by $\one_{\{\eta_j\subset \hat{\Omega}_j\}}\exp(c\mu(\Omega; \eta_j, \Omega\setminus \hat{\Omega}_j))$.
Now, by Lemma \ref{lem::sle_boundary_perturbation}, this
is the same as the law of the $\SLE_{\kappa}$ in $\hat{\Omega}_j$ connecting $x_{a_j}$ and $x_{b_j}$. 
This completes the proof.
\end{proof}

\subsection{Properties of Global Multiple SLEs}
\label{subsec::further_properties_of_Zalpha}
Next, we prove useful properties of global multiple $\SLE$s: first, we establish a boundary perturbation property,
and then a cascade property describing the marginal law of one curve in a global multiple $\SLE$. 

To begin, we set $\LB_{\emptyset}:=1$ and $\PartF_\emptyset := 1$, and define, for all integers $N \ge 1$ and link patterns 
$\alpha \in \LP_N$, the bound function $\LB_{\alpha}$ and the pure partition function $\PartF_{\alpha}$ as
\begin{align}\label{eqn::purepartition_alpha_def}
\begin{split}
\LB_{\alpha} \colon \chamber_{2N} \to \Rpos, 
\qquad \qquad
&\LB_{\alpha}(x_1, \ldots, x_{2N}) := 
\prod_{\link{a}{b} \in \alpha} |x_{b}-x_{a}|^{-1},\\
\PartF_{\alpha} \colon \chamber_{2N} \to \Rpos, 
\qquad \qquad
&\PartF_{\alpha}(x_1, \ldots, x_{2N}) := 
f_{\alpha}(\HH; x_1, \ldots , x_{2N}) \, \LB_{\alpha}(x_1, \ldots, x_{2N})^{2h},
\end{split}
\end{align}
where $f_{\alpha} = |\QQ_{\alpha}|$ is the function defined in~\eqref{eqn::def_mixingpart}.

If the points $x_1, \ldots, x_{2N}$ of the polygon $(\Omega; x_1, \ldots, x_{2N})$ lie on
sufficiently regular boundary segments (e.g., $C^{1+\eps}$ for some $\eps > 0$),
we call $(\Omega; x_1, \ldots, x_{2N})$ a \emph{nice polygon}. For a nice polygon $(\Omega; x_1, \ldots, x_{2N})$,
we define
\begin{align}\label{eqn::purepartition_malpha_def_general}
\begin{split}
&\LB_{\alpha}(\Omega; x_1,\ldots, x_{2N}) := 
\prod_{\link{a}{b} \in \alpha} H_{\Omega}(x_{a}, x_{b})^{1/2},\\
&\PartF_{\alpha}(\Omega; x_1,\ldots, x_{2N}) := f_{\alpha}(\Omega; x_1, \ldots, x_{2N}) \, \LB_{\alpha}(\Omega; x_1,\ldots, x_{2N})^{2h}.
\end{split}
\end{align}
This definition agrees with~\eqref{eq: multiple SLE conformal covariance}, by the conformal covariance property
of the boundary Poisson kernel $H_{\Omega}$ and the conformal invariance property of $f_{\alpha}$. 
We also note that the bounds~\eqref{eq: bound for f alpha} show that
\begin{align}\label{eqn::partf_alpha_upper}
\PartF_{\alpha}(\Omega; x_1,\ldots, x_{2N}) \le \LB_{\alpha}(\Omega; x_1,\ldots, x_{2N})^{2h}.
\end{align}

\subsubsection{Boundary Perturbation Property}

Multiple $\SLE$s have a boundary perturbation property analogous to Lemma~\ref{lem::sle_boundary_perturbation}.
To state it, we use the specific notation $\QQ_{\alpha}^{\#}(\Omega; x_1, \ldots, x_{2N})$ for the global $N$-$\SLE_\kappa$
probability measure associated to the link pattern $\alpha = \{\link{a_1}{b_1}, \ldots , \link{a_N}{b_N}\} \in \LP_N$ 
in the polygon $(\Omega; x_1, \ldots, x_{2N})$. 

\begin{proposition}\label{prop::multiplesle_boundary_perturbation}
Let $\kappa \in (0,4]$. Let $(\Omega; x_1, \ldots, x_{2N})$ be a polygon and $U \subset \Omega$ a sub-polygon. 
Then, the probability measure $\QQ_{\alpha}^{\#}(U; x_1, \ldots, x_{2N})$ is absolutely continuous with respect to 
$\QQ_{\alpha}^{\#}(\Omega; x_1, \ldots, x_{2N})$, with Radon-Nikodym derivative 
\begin{align*}
 \frac{\ud \QQ_{\alpha}^{\#}(U; x_1, \ldots, x_{2N})}{\ud \QQ_{\alpha}^{\#}(\Omega; x_1, \ldots, x_{2N})} (\eta_1, \ldots, \eta_N)
\, = \,
\frac{\PartF_{\alpha}(\Omega; x_1, \ldots, x_{2N})}{\PartF_{\alpha}(U; x_1, \ldots, x_{2N})} 
\; \one_{\{\eta_j\subset U \; \forall \; j\}} \exp \bigg(c\mu \Big(\Omega; \Omega\setminus U, \bigcup_{j=1}^N \eta_j \Big) \bigg).
\end{align*} 
Moreover, if $\kappa\le 8/3$ and $(\Omega; x_1, \ldots, x_{2N})$ is a nice polygon, then we have
\begin{align} \label{eqn::partitionfunction_mono}
\PartF_{\alpha}(\Omega; x_1, \ldots, x_{2N}) \geq \PartF_{\alpha}(U; x_1, \ldots, x_{2N}).
\end{align}
\end{proposition}
\begin{proof}
From the formula~\eqref{eqn::def_radon_alpha} and Lemma~\ref{lem::malpha_boundaryperturbation}, we see that 
\begin{align*}
 & \one_{\{\eta_j\subset U \; \forall \; j\}} \; \ud \QQ_{\alpha}(\Omega; x_1, \ldots, x_{2N}) \\
= & \; \one_{\{\eta_j\subset U \; \forall \; j\}}\one_{\{\eta_j\cap\eta_k=\emptyset \; \forall \; j\neq k\}} 
\; \exp(c m_{\alpha}(\Omega; \eta_1, \ldots, \eta_N)) \; \ud \PP_{\alpha}\\
= & \; \one_{\{\eta_j\subset U \; \forall \; j\}}\one_{\{\eta_j\cap\eta_k=\emptyset \; \forall \; j\neq k\}} 
\; \exp(c m_{\alpha}(U; \eta_1,\ldots, \eta_N))
\; \exp \bigg( \hspace{-1mm} -c \mu \Big(\Omega; \Omega\setminus U, \bigcup_{j=1}^N \eta_j \Big) \bigg)\\
& \times \prod_{j=1}^N\exp\left(c\mu(\Omega; \Omega\setminus U, \eta_j)\right) \; \ud \PP(\Omega; x_{a_j}, x_{b_j})
\end{align*}
By Lemma \ref{lem::sle_boundary_perturbation}, we have 
\begin{align*}
& \one_{\{\eta_j\subset U \; \forall \; j\}} \; \ud \QQ_{\alpha}(\Omega; x_1, \ldots, x_{2N}) \\
= & \; \one_{\{\eta_j\cap\eta_k=\emptyset \; \forall \; j\neq k\}} \; \exp(c m_{\alpha}(U; \eta_1,\ldots, \eta_N))
\; \exp \bigg(\hspace{-1mm} -c \mu \Big(\Omega; \Omega\setminus U, \bigcup_{j=1}^N \eta_j \Big) \bigg)\\
& \times \prod_{j=1}^N\left(\frac{H_{U}(x_{a_j}, x_{b_j})}{H_{\Omega}(x_{a_j}, x_{b_j})}\right)^h \; \ud \PP(U; x_{a_j}, x_{b_j})  \\
= & \; \exp \bigg(\hspace{-1mm} -c \mu \Big(\Omega; \Omega\setminus U, \bigcup_{j=1}^N \eta_j \Big) \bigg) 
\prod_{j=1}^N\left(\frac{H_{U}(x_{a_j}, x_{b_j})}{H_{\Omega}(x_{a_j}, x_{b_j})}\right)^h \; \ud \QQ_{\alpha}(U; x_1, \ldots, x_{2N}).
\end{align*}
Combining this with the definition~\eqref{eqn::purepartition_malpha_def_general},
we obtain the asserted Radon-Nikodym derivative.
The monotonicity property~\eqref{eqn::partitionfunction_mono} follows from the fact that
when $\kappa\le 8/3$, we have $c\le 0$ and thus, 
\begin{align*}
1 \ge\PP[\eta_j\subset U \; \text{ for all } j] \ge 
\frac{\PartF_{\alpha}(U; x_1, \ldots, x_{2N})}{\PartF_{\alpha}(\Omega; x_1, \ldots, x_{2N})}.
\end{align*}
This concludes the proof.
\end{proof}

\subsubsection{Marginal Law}
\label{subsubsec::globalmultiple_marginal}

Next we prove a cascade property for the measure $\QQ_{\alpha}^{\#}$.
Given any link $\link{a}{b} \in \alpha$, let $\eta$ be the curve connecting $x_{a}$ and $x_{b}$ in 
the global $N$-$\SLE_{\kappa}$ with law $\QQ_{\alpha}^{\#}$, as in Theorem~\ref{thm::global_existence}.
Assume that $a < b$ for notational simplicity.
Then, the link $\link{a}{b}$ divides the link pattern $\alpha$ into two sub-link patterns, connecting respectively the points
$\{a+1, \ldots, b-1\}$ and $\{b+1, \ldots, a-1\}$. After relabeling of the indices, we denote these two link patterns by $\alpha^R$ 
and $\alpha^L$. 
Also, the domain $\Omega\setminus\eta$ has two connected components, which we denote by $D_{\eta}^L$ and $D_{\eta}^R$. 
The notations are illustrated in Figure~\ref{fig::malpha_cascade}.

\begin{proposition} \label{prop::GeneralCascadeProp}
The marginal law of $\eta$ under $\QQ_{\alpha}^{\#}$ is absolutely continuous with respect to 
the law $\PP(\Omega; x_{a}, x_{b})$ of the $\SLE_\kappa$ connecting $x_{a}$ and $x_{b}$, 
with Radon-Nikodym derivative 
\begin{align*}
\frac{H_{\Omega}(x_{a}, x_{b})^h}{\PartF_{\alpha}(\Omega; x_1, \ldots, x_{2N})} \; 
\PartF_{\alpha^L}(D_{\eta}^L; x_{b+1}, x_{b+2}, \ldots, x_{a-1}) \; \PartF_{\alpha^R}(D_{\eta}^R; x_{a+1}, x_{a+2}, \ldots, x_{b-1}).
\end{align*}
\end{proposition}

\begin{proof}
Note that the points $x_{b+1}, \ldots, x_{a-1}$ (resp.~$x_{a+1}, \ldots, x_{b-1}$)
lie along the boundary of $D_{\eta}^L$ (resp.~$D_{\eta}^R$) in counterclockwise order.
Denote by $(\eta_1, \ldots, \eta_N)\in X_0^{\alpha}(\Omega; x_1, \ldots, x_{2N})$ the global  $N$-$\SLE_{\kappa}$
with law $\QQ_{\alpha}^{\#}$. Amongst the curves other than $\eta$, we denote by $\eta_1^L, \ldots, \eta_l^L$ the 
ones contained in $D_{\eta}^L$ and by $\eta^R_1, \ldots, \eta^R_r$ the ones contained in $D_{\eta}^R$ (so $l+r=N-1$). 

First, we prove by induction on $N \ge 1$ that
\begin{align}\label{eqn::malpha_decomposition}
m_{\alpha}(\Omega; \eta_1, \ldots, \eta_N) = m_{\alpha^L}(D_{\eta}^L; \eta_1^L, \ldots, \eta_l^L) 
+ m_{\alpha^R}(D_{\eta}^R; \eta_1^R, \ldots, \eta^R_r) + \sum_{\eta' \neq \eta}\mu(\Omega; \eta, \eta'). 
\end{align}
Equation~\eqref{eqn::malpha_decomposition} trivially holds for $N=1$, since 
$m_{\emptyset} = 0 = m_{\vcenter{\hbox{\includegraphics[scale=0.2]{figures/link-0.pdf}}}}$. 
By symmetry, we may assume that $\link{a}{b} \neq \link{2N-1}{2N} \in \alpha \, \cap \, \alpha^L$. 
Then, we let $\eta_1 = \eta_1^L \subset D_{\eta}^L$ be the curve connecting $x_{2N-1}$ and $x_{2N}$,
denote $\hat{\alpha} = \alpha\removeLink \link{2N-1}{2N}$, and define $\hat{\alpha}^L$ and $\hat{\alpha}^R$
similarly as above --- so $\alpha^L = \hat{\alpha}^L \cup \{\link{2N-1}{2N}\}$ and $\hat{\alpha}^R = \alpha^R$. 
Applying Lemma~\ref{lem::malpha_cascade} and the induction hypothesis, we~get
\begin{align*}
m_{\alpha}(\Omega; \eta_1, \ldots, \eta_N) 
& = m_{\hat{\alpha}}(\Omega; \eta_2, \ldots, \eta_N) + \mu(\Omega; \eta_1, \Omega\setminus\hat{\Omega}_1) 
\\
& = m_{\hat{\alpha}^L}(D_{\eta}^L; \eta^L_2, \ldots, \eta^L_l) + m_{\alpha^R}(D_{\eta}^R; \eta_1^R, \ldots, \eta_r^R)
 + \sum_{\eta' \neq \eta, \eta_1}\mu(\Omega; \eta, \eta') + \mu(\Omega; \eta_1, \Omega\setminus\hat{\Omega}_1).
\end{align*}
Combining this with the decomposition 
$\mu(\Omega; \eta_1, \Omega\setminus\hat{\Omega}_1) = \mu(D_{\eta}^L; \eta_1, D_{\eta}^L\setminus\hat{\Omega}_1)+\mu(\Omega; \eta_1, \eta)$,
we obtain
\begin{align*}
m_{\alpha}(\Omega; \eta_1, \ldots, \eta_N) 
& = m_{\hat{\alpha}^L}(D_{\eta}^L; \eta^L_2, \ldots, \eta^L_l)+\mu(D_{\eta}^L; \eta_1, D_{\eta}^L\setminus\hat{\Omega}_1) 
+ m_{\alpha^R}(D_{\eta}^R; \eta_1^R, \ldots, \eta_r^R) + \sum_{\eta' \neq \eta}\mu(\Omega; \eta, \eta') \\
& = m_{\alpha^L}(D_{\eta}^L; \eta_1^L, \ldots, \eta^L_l) 
+ m_{\alpha^R}(D_{\eta}^R; \eta_1^R, \ldots, \eta_r^R)+\sum_{\eta' \neq \eta}\mu(\Omega; \eta, \eta') ,
\end{align*}
by Lemma~\ref{lem::malpha_cascade}. This completes the proof of the identity~\eqref{eqn::malpha_decomposition}. 

Next, we prove the proposition. From~\eqref{eqn::def_radon_alpha}, we see that
\begin{align*}
\ud \QQ_{\alpha} 
= & \;\one_{\{\eta_i\cap\eta_k=\emptyset \; \forall \; i\neq k\}} 
\; \exp( c m_{\alpha}(\Omega; \eta_1, \ldots, \eta_N)) \prod_{\link{c}{d} \in \alpha} \; \ud \PP(\Omega; x_{c}, x_{d})\\
= & \;\one_{\{\eta_i\cap\eta_k=\emptyset \; \forall \; i\neq k\}} 
\; \exp\left(c m_{\alpha^L}(D_{\eta}^L; \eta_1^L, \ldots, \eta^L_l)\right) 
\; \exp\left(c m_{\alpha^R}(D_{\eta}^R; \eta^R_1, \ldots, \eta_r^R)\right) \\
& \times \prod_{\eta' \neq \eta}\exp(c\mu(\Omega; \eta, \eta')) 
\prod_{\substack{\link{c}{d} \in \alpha , \\ \link{c}{d} \neq \link{a}{b}}} \; \ud \PP(\Omega; x_{c}, x_{d}) \times \ud \PP(\Omega; x_{a}, x_{b}) 
&& \text{[by \eqref{eqn::malpha_decomposition}]} \\
= & \;\one_{\{\eta_i\cap\eta_k=\emptyset \; \forall \; i\neq k\}} 
\;  \ud \PP(\Omega; x_{a}, x_{b})
&& \text{[by Lemma~\ref{lem::sle_boundary_perturbation}]} \\
& \times \exp\left(c m_{\alpha^L}(D_{\eta}^L; \eta_1^L, \ldots, \eta^L_l)\right) \; 
\prod_{\link{c}{d} \in \alpha^L} \left(\frac{H_{D_{\eta}^L}(x_{c}, x_{d})}{H_{\Omega}(x_{c}, x_{d})}\right)^h \;  \ud \PP(D_{\eta}^L; x_{c}, x_{d})\\
& \times \exp\left(c m_{\alpha^R}(D_{\eta}^R; \eta^R_1, \ldots, \eta_r^R)\right)
\; \prod_{\link{c}{d} \in \alpha^R}\left(\frac{H_{D_{\eta}^R}(x_{c}, x_{d})}{H_{\Omega}(x_{c}, x_{d})}\right)^h\;  \ud \PP(D_{\eta}^R; x_{c}, x_{d}).
\end{align*}
By definitions~\eqref{eqn::def_mixingpart},~\eqref{eqn::purepartition_alpha_def}, and~\eqref{eqn::purepartition_malpha_def_general},
this implies that the law of $\eta$ under $\QQ_{\alpha}^{\#} = \QQ_{\alpha} / f_{\alpha}$
is absolutely continuous with respect to $\PP(\Omega; x_{a}, x_{b})$, and the Radon-Nikodym derivative has the asserted form. 
\end{proof}

\begin{corollary}\label{cor::globalsle_weighted}
Let $\alpha \in \LP_N$ and $j\in\{1, \ldots, 2N-1\}$ such that $\link{j}{j+1} \in \alpha$,
and denote by $\hat{\alpha}=\alpha \removeLink \link{j}{j+1} \in \LP_{N-1}$. 
Let $\eta_j$ be the curve connecting $x_j$ and $x_{j+1}$ in the global $N$-$\SLE_{\kappa}$ with law $\QQ_{\alpha}^{\#}$.
Denote by $D_j$ the connected component of $\Omega\setminus \eta_j$ having $x_1, \ldots, x_{j-1}, x_{j+2}, \ldots, x_{2N}$ on its boundary. 
Then, the marginal law of $\eta_j$ under $\QQ_{\alpha}^{\#}$
is absolutely continuous with respect to
the law $\PP(\Omega; x_j, x_{j+1})$ of the $\SLE_\kappa$ connecting $x_{j}$ and $x_{j+1}$, 
with Radon-Nikodym derivative 
\begin{align*}
\frac{H_{\Omega}(x_j, x_{j+1})^h}{\PartF_{\alpha}(\Omega; x_1, \ldots, x_{2N})} \; \PartF_{\hat{\alpha}}(D_j; x_1, \ldots, x_{j-1}, x_{j+2}, \ldots, x_{2N}).
\end{align*}
\end{corollary}

\section{Pure Partition Functions for Multiple SLEs}
\label{sec::purepartition_existence}
In this section, we prove Theorem~\ref{thm::purepartition_existence}, which says that 
the pure partition functions of multiple $\SLE$s are smooth, positive, and (essentially) unique. 
Corollary~\ref{cor::localmultiplesle} in Section~\ref{subsec::global_vs_local} relates them to certain extremal 
multiple $\SLE$ measures, thus verifying a conjecture from~\cite{KytolaMultipleSLE, KytolaPeoltolaPurePartitionSLE}. 
In Section~\ref{subsec::global_vs_local}, we also complete the proof of Theorem~\ref{thm::global_existence}, 
by proving in Lemma~\ref{lem::global_is_local}
that the local and global $\SLE_\kappa$ associated to $\alpha$ agree.

\subsection{Pure Partition Functions: Proof of Theorem~\ref{thm::purepartition_existence}}
\label{subsec::purepfexistence_proof}
We prove Theorem~\ref{thm::purepartition_existence} by a succession of lemmas 
establishing the 
asserted properties of the pure partition functions $\PartF_{\alpha}$ defined in~\eqref{eqn::purepartition_alpha_def}.
From the Brownian loop measure construction, it is difficult to show directly that the partition function $\PartF_{\alpha}$
is a solution to the system $\mathrm{(PDE)}$~\eqref{eq: multiple SLE PDEs}, because it is not clear 
why $\PartF_{\alpha}$ should be twice continuously differentiable. 
To this end, we use the hypoellipticity of the PDEs~\eqref{eq: multiple SLE PDEs} from 
Proposition~\ref{prop::smoothness}. With hypoellipticity, it suffices to
prove that $\PartF_{\alpha}$ is a distributional solution to $\mathrm{(PDE)}$~\eqref{eq: multiple SLE PDEs},
which we establish in Lemma~\ref{lem::partitionfunction_pde} by constructing a martingale from the conditional expectation 
of the Radon-Nikodym derivative~\eqref{eqn::def_radon_alpha}.  

\begin{lemma}\label{lem::partitionfunction_positive}
The function $\PartF_{\alpha}$ defined in~\eqref{eqn::purepartition_alpha_def} 
satisfies the bound~\eqref{eqn::partitionfunction_positive}.
\end{lemma}
\begin{proof}
This follows from~\eqref{eqn::partf_alpha_upper}, which in turn follows from~\eqref{eq: bound for f alpha}.
\end{proof}

\begin{lemma}\label{lem::partitionfunction_cov}
The function $\PartF_{\alpha}$ defined in~\eqref{eqn::purepartition_alpha_def} 
satisfies the M\"obius covariance $\mathrm{(COV)}$~\eqref{eq: multiple SLE Mobius covariance}. 
\end{lemma}
\begin{proof}
The function $f_{\alpha}(\HH; x_1, \ldots, x_{2N})$ is M\"obius invariant by~\eqref{eqn::def_radon_alpha}.
Combining with the conformal covariance~\eqref{eqn::poisson_cov} of the boundary Poisson kernel, 
we see that $\PartF_{\alpha}$ satisfies the M\"obius covariance 
$\mathrm{(COV)}$~\eqref{eq: multiple SLE Mobius covariance}. 
\end{proof}

\begin{lemma}\label{lem::partitionfunction_asy}
The function $\PartF_{\alpha}$ defined in~\eqref{eqn::purepartition_alpha_def} 
satisfies the following 
asymptotics: for all $\alpha \in \LP_N$ and for all $j \in \{1, \ldots, 2N-1 \}$ and 
$x_1 < \cdots < x_{j-1} < \xi < x_{j+2} < \cdots < x_{2N}$, we have 
\begin{align}\label{eqn::partf_alpha_asy_refined}
\lim_{\substack{\tilde{x}_j , \tilde{x}_{j+1} \to \xi, \\ \tilde{x}_i\to x_i \text{ for } i \neq j, j+1}} 
\frac{\PartF_\alpha(\tilde{x}_1 , \ldots , \tilde{x}_{2N})}{(\tilde{x}_{j+1} - \tilde{x}_j)^{-2h}} 
=\begin{cases}
0 , \quad &
    \text{if } \link{j}{j+1} \notin \alpha , \\
\PartF_{\hat{\alpha}}(x_{1},\ldots,x_{j-1},x_{j+2},\ldots,x_{2N}) , &
    \text{if } \link{j}{j+1} \in \alpha ,
\end{cases}
\end{align}
where $\hat{\alpha} = \alpha \removeLink \link{j}{j+1}$.  
In particular, $\PartF_{\alpha}$ satisfies $\mathrm{(ASY)}$~\eqref{eq: multiple SLE asymptotics}.
\end{lemma}

\begin{proof}
The case $\link{j}{j+1} \notin \alpha$ follows immediately from the bound~\eqref{eqn::partf_alpha_upper} 
with Lemma~\ref{lem::b_alpha_asy_refined} in Appendix~\ref{subsec:appendix-prop}. 
To prove the case $\link{j}{j+1} \in \alpha$, we assume without loss of generality that $j=1$ and $\{1,2\}\in\alpha$. 
Let $\tilde{\eta}$ be the $\SLE_{\kappa}$ in $\HH$ connecting $\tilde{x}_1$ and $\tilde{x}_2$,
let $\tilde{D}$ be the unbounded connected component of $\HH\setminus\tilde{\eta}$, and denote by
$\tilde{g}$ the conformal map from $\tilde{D}$ onto $\HH$ normalized at $\infty$. Then we have
\begin{align*}
\frac{\PartF_\alpha(\tilde{x}_1 , \ldots , \tilde{x}_{2N})}{(\tilde{x}_2 - \tilde{x}_1)^{-2h}} 
&=\E\left[\PartF_{\hat{\alpha}}(\tilde{D}; \tilde{x}_3, \ldots, \tilde{x}_{2N})\right]&&\text{[by Corollary~\ref{cor::globalsle_weighted}]}\\
&=\E\left[\prod_{i=3}^{2N}\tilde{g}'(\tilde{x}_i)^h\PartF_{\hat{\alpha}}(\tilde{g}(\tilde{x}_3), \ldots, \tilde{g}(\tilde{x}_{2N}))\right] . 
&& \text{[by~\eqref{eq: multiple SLE conformal covariance}]}
\end{align*}
Now, as $\tilde{x}_1, \tilde{x}_2 \to \xi$ and $\tilde{x}_i\to x_i$ for $i\neq 1,2$, we have $\tilde{g} \to \mathrm{id}_{\HH}$ almost surely. 
Moreover, by the bound~\eqref{eqn::partf_alpha_upper} and the monotonicity property~\eqref{eqn::b_alpha_mono} from 
Appendix~\ref{subsec:appendix-prop}, 
we have 
\begin{align*}
\PartF_{\hat{\alpha}}(\tilde{D}; \tilde{x}_3, \ldots, \tilde{x}_{2N})\le \LB_{\hat{\alpha}}(\tilde{D}; \tilde{x}_3, \ldots, \tilde{x}_{2N})^{2h} 
\le \LB_{\hat{\alpha}}(\tilde{x}_3, \ldots, \tilde{x}_{2N})^{2h}.
\end{align*}
Thus, by the bounded convergence theorem, as $\tilde{x}_1, \tilde{x}_2 \to \xi$, and $\tilde{x}_i\to x_i$ for $i\neq 1,2$,  we have
\begin{align*}
\frac{\PartF_\alpha(\tilde{x}_1 , \ldots , \tilde{x}_{2N})}{(\tilde{x}_2 - \tilde{x}_1)^{-2h}} 
=\E \left[\prod_{i=3}^{2N}\tilde{g}'(\tilde{x}_i)^h\PartF_{\hat{\alpha}}(\tilde{g}(\tilde{x}_3), \ldots, \tilde{g}(\tilde{x}_{2N})) \right] 
\; \longrightarrow \; \PartF_{\hat{\alpha}}(x_3, \ldots, x_{2N}) ,
\end{align*}
which proves~\eqref{eqn::partf_alpha_asy_refined}. 
The asymptotics property $\mathrm{(ASY)}$~\eqref{eq: multiple SLE asymptotics} is then immediate.
\end{proof}

\begin{lemma}\label{lem::partitionfunction_pde}
The function $\PartF_{\alpha}$ defined in~\eqref{eqn::purepartition_alpha_def} 
is smooth and it satisfies the system $\mathrm{(PDE)}$~\eqref{eq: multiple SLE PDEs} of $2N$ 
partial differential equations of second order.
\end{lemma}

\begin{proof}
We prove that $\PartF_{\alpha}$ satisfies the partial differential equation 
of~\eqref{eq: multiple SLE PDEs} for $i = 1$; the others follow by symmetry.
Denote the pair of $i = 1$ in $\alpha$ by $b$, 
and denote 
$\hat{\alpha} = \alpha \removeLink \link{1}{b}$.
Let $\eta_1$ be the curve connecting $x_1$ and $x_b$, and 
$\hat{\Omega}_1$ the connected component of $\HH \setminus \{\eta_2, \ldots, \eta_N\}$ 
that has $x_1$ and $x_b$ on its boundary.
Then, given $\{\eta_2, \ldots, \eta_N\}$, the conditional law of $\eta_1$ is that of the chordal 
$\SLE_{\kappa}$ in $\hat{\Omega}_1$ from $x_1$ to $x_b$.

Recall from~\eqref{eqn::purepartition_alpha_def} that the function $\PartF_\alpha$ is defined 
in terms of the expectation of $R_{\alpha}$. We calculate the conditional expectation 
$\E_{\alpha}[R_{\alpha}(\HH;\eta_1, \ldots, \eta_N) \cond \eta_1[0,t]]$ for small $t > 0$,
and construct a martingale involving the function $\PartF_{\alpha}$.
Diffusion theory then provides us with the desired partial differential equation~\eqref{eq: multiple SLE PDEs}
in distributional sense, and we may conclude by hypoellipticity (Proposition~\ref{prop::smoothness}). 

For $t \geq 0$, we denote $K_t := \eta_1[0,t]$, and $H_t:= \HH \setminus K_t$, and 
$\tilde{\eta}_1 := (\eta_1(s), s\ge t)$. 
Using the observation that the Brownian loop measure can be decomposed as
\begin{align}\label{eq: BLM trivial decomposition}
\mu(\HH; \eta_1, \HH \setminus \hat{\Omega}_1) 
= \mu(H_t; \tilde{\eta}_1, \HH\setminus\hat{\Omega}_1) + \mu(\HH; K_t, \HH\setminus\hat{\Omega}_1),
\end{align}
combined with Lemmas~\ref{lem::malpha_cascade} and~\ref{lem::malpha_boundaryperturbation}, 
we write the quantity $m_{\alpha}$ defined in~\eqref{eqn::def_malpha} in the following form: 
\begin{align*}
m_{\alpha}(\HH; \eta_1, \eta_2, \ldots, \eta_N) 
= \; & m_{\hat{\alpha}}(\HH;\eta_2,\ldots,\eta_N) + \mu(\HH; \eta_1, \HH\setminus\hat{\Omega}_1) 
&& \text{[by Lemma \ref{lem::malpha_cascade}]} \\
= \; & m_{\hat{\alpha}}(H_t; \eta_2,\ldots,\eta_N) + \sum_{j=2}^N \mu(\HH; K_t, \eta_j) - 
\mu \big(\HH; K_t, \bigcup_{j=2}^N \eta_j \big)
&& \text{[by Lemma \ref{lem::malpha_boundaryperturbation}]} \\
\; & + \mu(H_t; \tilde{\eta}_1, \HH\setminus\hat{\Omega}_1) + \mu(\HH; K_t, \HH\setminus\hat{\Omega}_1).
&& \text{[by \eqref{eq: BLM trivial decomposition}]} 
\end{align*}
Note that 
$\mu(\HH; K_t, \bigcup_{j=2}^N \eta_j) = \mu(\HH; K_t, \HH\setminus\hat{\Omega}_1)$,
so the last terms of the last two lines cancel. Combining the first terms of these two lines 
with the help of Lemma \ref{lem::malpha_cascade}, we obtain
\begin{align*}
m_{\alpha}(\HH; \eta_1, \eta_2, \ldots, \eta_N) 
= m_{\alpha}(H_t; \tilde{\eta}_1, \eta_2, \ldots, \eta_N) + \sum_{j=2}^N \mu(\HH; K_t, \eta_j) .
\end{align*}
Using this, we write the Radon-Nikodym derivative~\eqref{eqn::def_radon_alpha} in the form
\begin{align*}
\; & R_{\alpha}(\HH; \eta_1, \eta_2, \ldots, \eta_N) \\
= \; & \one_{\{\eta_j\cap\eta_k=\emptyset \; \forall \; j\neq k\}} 
\; \exp(c m_{\alpha}(H_t; \tilde{\eta}_1, \eta_2, \ldots, \eta_N)) 
\;  \prod_{j=2}^N \exp(c\mu(\HH; K_t, \eta_j)) \\
= \; & R_{\alpha}(H_t; \tilde{\eta}_1, \eta_2, \ldots, \eta_N) 
\;  \prod_{j=2}^N \one_{\{\eta_j \subset H_t\}} \exp(c\mu(\HH; K_t, \eta_j))
&& 
\text{[by~\eqref{eqn::def_radon_alpha}]}  \\
= \; & R_{\alpha}(H_t; \tilde{\eta}_1, \eta_2, \ldots, \eta_N)
\prod_{\substack{\link{c}{d} \in \alpha , \\ \link{c}{d} \neq \link{1}{b} }}
\left(\frac{H_{H_t}(x_{c},x_{d})}{H_{\HH}(x_{c},x_{d})}\right)^h 
\frac{\ud \PP(H_t; x_{c},x_{d})}{\ud \PP(\HH; x_{c},x_{d})} .
&& 
\text{[by Lemma~\ref{lem::sle_boundary_perturbation}]} 
\end{align*}
This implies that, given $K_t = \eta_1[0,t]$, the conditional expectation of $R_{\alpha}$ is
\begin{align*}
\E_{\alpha}[R_{\alpha}(\HH; \eta_1, \eta_2, \ldots, \eta_N) \cond K_t]
= \; & 
\E_{\alpha}[R_{\alpha}(H_t;\tilde{\eta}_1, \eta_2, \ldots, \eta_N)]
\prod_{\substack{\link{c}{d} \in \alpha , \\ \link{c}{d} \neq \link{1}{b} }}
\left(\frac{H_{H_t}(x_{c}, x_{d})}{H_{\HH}(x_{c}, x_{d})}\right)^h \\
= \; & 
f_{\alpha}(H_t; \eta_1(t), x_2, \ldots, x_{2N})
\prod_{\substack{\link{c}{d} \in \alpha , \\ \link{c}{d} \neq \link{1}{b} }}
\left(\frac{H_{H_t}(x_{c}, x_{d})}{H_{\HH}(x_{c}, x_{d})}\right)^h .
\end{align*}
Let $g_t$ be the Loewner map normalized at $\infty$ associated to $\eta_1$, and $W_t$ its driving process.
By the conformal invariance of $f_{\alpha}$, using~\eqref{eqn::purepartition_alpha_def}
and the formula $H_{\HH}(x,y) = (y-x)^{-2}$ for the Poisson kernel in $\HH$, we have
\begin{align*}
f_{\alpha}(H_t; \eta_1(t), x_2, \ldots, x_{2N})
= \; & f_{\alpha}(\HH; W_t, g_t(x_2), \ldots, g_t(x_{2N})) \\
= \; & (g_t(x_b)-W_t)^{2h}
\prod_{\substack{\link{c}{d} \in \alpha , \\ \link{c}{d} \neq \link{1}{b} }}
(g_t(x_{c})-g_t(x_{d}))^{2h} \times
\PartF_{\alpha}(W_t, g_t(x_2), \ldots, g_t(x_{2N})) .
\end{align*}
On the other hand, by~\eqref{eqn::poisson_cov}, we have
\begin{align*}
\prod_{\substack{\link{c}{d} \in \alpha , \\ \link{c}{d} \neq \link{1}{b} }}
H_{H_t}(x_{c}, x_{d})^h = 
\prod_{\substack{\link{c}{d} \in \alpha , \\ \link{c}{d} \neq \link{1}{b} }}
g_t'(x_{c})^hg_t'(x_{d})^h (g_t(x_{c})-g_t(x_{d}))^{-2h} .
\end{align*}
Combining the above observations, we get
$\E_{\alpha}[R_{\alpha}(\HH; \eta_1, \eta_2, \ldots, \eta_N)\cond K_t] 
= \prod_{\link{c}{d} \in \alpha} (x_{d}-x_{c})^{2h} \times M_t$,
where
\begin{align*}
M_t := \prod_{i \neq 1, b} g_t'(x_i)^h \times \PartF_{\alpha}(W_t, g_t(x_2), \ldots, g_t(x_{2N})) \times (g_t(x_b)-W_t)^{2h}. 
\end{align*}
Thus, $M_t$ is a martingale for $\eta_1$. Now, we write $M_t= F(X_t)$, where 
\begin{align*} 
F (\boldsymbol{x}, \boldsymbol{y})
=  \prod_{j\neq 1, b} y_j^h \times \PartF_{\alpha}(x_1, x_2, \ldots, x_{2N}) \times (x_b-x_1)^{2h} 
\end{align*}
is a continuous function of  $(\boldsymbol{x}, \boldsymbol{y}) := (x_1, \ldots, x_{2N}, y_2, \ldots, y_{2N}) \in \chamber_{2N} \times \R^{2N-1}$
(independent of $y_b$), and $X_t = (W_t, g_t(x_2), \ldots, g_t(x_{2N}), g_t'(x_2), \ldots, g_t'(x_{2N}))$ is an It\^o process, 
whose infinitesimal generator, when acting on twice continuously differentiable functions, can be written as the differential operator
\begin{align} \label{eq::generator}
A = \frac{\kappa}{2} \partial^2_1 +\frac{\kappa-6}{x_1-x_b}\partial_1
+ \sum_{j = 2}^{2N} \left( \frac{2}{x_{j}-x_{1}}\partial_j - \frac{2y_j}{(x_{j}-x_{1})^{2}} \partial_{2N-1+j} \right) 
\end{align} 
--- see, e.g.,~\cite[Chapter VII]{RevuzYorMartBM} for background on diffusions.

We next consider the generator $A$ in the distributional sense.
Let $(P_t)_{t>0}$ be the transition semigroup of $(X_t)_{t>0}$. 
By definition, $A$ is a linear operator on the space 
$C_0 = C_0(\chamber_{2N} \times (1/2,3/2)^{2N-1} ;\C)$ of continuous functions that vanish at 
infinity\footnote{We remark that the space $C_0 = C_0(\chamber_{2N} \times (1/2,3/2)^{2N-1} ;\C)$ 
is the Banach completion of the test function space
$\Test_space = \Test_space(\chamber_{2N} \times (1/2,3/2)^{2N-1} ;\C)$ of smooth compactly supported functions,
with respect to the sup norm.}.
The domain of $A$ in $C_0$ consists of those functions $f$ for which the limit
\begin{align} \label{eq::Adefn}
Af := \lim_{t \searrow 0} \frac{1}{t}(P_t  - \mathrm{id})f
\end{align}
exists in $C_0$. When restricted to twice continuously differentiable functions in $C_0$, 
$A$ equals the differential operator~\eqref{eq::generator}. More generally, we will argue that 
$A$ can be defined as $A := \lim_{t \searrow 0} \frac{1}{t}(P_t  - \mathrm{id})$ in the distributional sense,
and this extended definition of $A$ agrees with~\eqref{eq::generator} acting on distributions via~\eqref{eq::Diffdistribution}.

For each $t \geq 0$, the operator $P_t$ is bounded (a contraction), so the image $P_t f$ 
of any $f \in C_0$ is locally integrable. Therefore, $P_t f$ defines a distribution 
$P_t f \in \Distr_space = \Distr_space(\chamber_{2N} \times (1/2,3/2)^{2N-1} ;\C)$ via~\eqref{eq::Fdistribution}. 
More generally, because $P_t$ is a continuous operator (with respect to the sup norm), 
and the space $\Test_space  = \Test_space(\chamber_{2N} \times (1/2,3/2)^{2N-1} ;\C)$ of test functions is dense in 
$\Distr_space$~\cite[Lemma 1.13.5]{TaoEpsilon}, 
the operator $P_t$ defines, for any distribution $f \in \Distr_space$, a distribution $P_t f \in \Distr_space$ via
\begin{align*}
\langle P_t f, \psi \rangle 
:= \; & \lim_{n \to \infty} \int_{\chamber_{2N} \times (1/2,3/2)^{2N-1}}
P_t f_n(\boldsymbol{x}, \boldsymbol{y}) \times \psi (\boldsymbol{x},\boldsymbol{y}) \ud \boldsymbol{x} \ud \boldsymbol{y} ,
\end{align*} 
for any test function $\psi \in \Test_space$,
where $f_n \in \Test_space$ is a sequence converging to $f$ in $\Distr_space$. 
In conclusion, the transition semigroup of $(X_t)_{t>0}$ gives rise to a linear operator
$P_t \colon \Distr_space \to \Distr_space$. 

Now, we define $A$ on $\Distr_space$ via its values on the dense subspace $\Test_space \subset \Distr_space$.
In this subspace, $A$ is already defined by~\eqref{eq::Adefn}, and in general, 
for any $f \in \Distr_space$, we define $A f \in \Distr_space$ via
\begin{align} \label{eqn::infinitesimalgenerator_equivdef2}
\langle A f, \psi \rangle 
:= \; \lim_{t\searrow 0} \int_{\chamber_{2N} \times (1/2,3/2)^{2N-1}} \frac{1}{t} 
(P_t f - f)(\boldsymbol{x}, \boldsymbol{y}) \times \psi (\boldsymbol{x},\boldsymbol{y}) \ud \boldsymbol{x} \ud \boldsymbol{y} ,
\end{align} 
if the limit exists for all test functions $\psi \in \Test_space$. Note that if $f \in \Test_space$, then 
$\frac{1}{t} (P_t f - f)$ converges to $Af$ locally uniformly, so the definition~\eqref{eqn::infinitesimalgenerator_equivdef2} 
indeed coincides with~\eqref{eq::Adefn}, and hence with~\eqref{eq::generator}, for all $f \in \Test_space$.
In conclusion, the definition~\eqref{eqn::infinitesimalgenerator_equivdef2} of $A \colon \Distr_space \to \Distr_space$
is an extension of the definition~\eqref{eq::Adefn} of $A$ from $C_0$ to the space of distributions.
In particular, we have
\begin{align*}
\langle A f, \psi \rangle 
= & \; \lim_{t\searrow 0} \int_{\chamber_{2N} \times (1/2,3/2)^{2N-1}} \frac{1}{t} 
(P_t f - f)(\boldsymbol{x}, \boldsymbol{y}) \times \psi (\boldsymbol{x},\boldsymbol{y}) \ud \boldsymbol{x} \ud \boldsymbol{y} 
&& \text{[by~\eqref{eqn::infinitesimalgenerator_equivdef2}]} \\
= & \;  \int_{\chamber_{2N} \times (1/2,3/2)^{2N-1}} 
f(\boldsymbol{x}, \boldsymbol{y}) \times (A^* \psi) (\boldsymbol{x}, \boldsymbol{y}) \ud \boldsymbol{x} \ud \boldsymbol{y} ,
&& \text{[by~\eqref{eq::generator}]}
\end{align*}
where $A^*$ is the transpose (dual operator) of~\eqref{eq::generator}:
\begin{align*}
A^* = \frac{\kappa}{2} \partial^2_1 - \frac{\kappa-6}{x_1-x_b}\partial_1
- \sum_{j = 2}^{2N} \left( \frac{2}{x_{j}-x_{1}}\partial_j + \frac{2y_j}{(x_{j}-x_{1})^{2}} \partial_{2N-1+j} \right)  .
\end{align*}

Now, since $F(X_t)$ is a martingale, we have $\E[F(X_t)]=F(X_0)$, i.e. 
\[(P_t F - F) (\boldsymbol{x}, \boldsymbol{1})=0 , \qquad 
\text{for all } t \geq 0 \text{ and } \boldsymbol{x} \in \chamber_{2N} .\]
Therefore, the continuous function 
$f(\boldsymbol{x}, \boldsymbol{y}) := F (\boldsymbol{x}, \boldsymbol{1}) 
= \PartF_{\alpha}(x_1, x_2, \ldots, x_{2N}) \times (x_b-x_1)^{2h}$, 
independent of $\boldsymbol{y}$, defines a distribution $f \in \Distr_space$
such that $\langle Af, \psi \rangle = 0$ for all test functions $\psi \in \Test_space$.

Recall that our goal is to show that $\PartF_\alpha$ is a distributional solution to 
the hypoelliptic PDE~\eqref{eq: multiple SLE PDEs} for $i=1$, that is, for all test functions $\phi \in \Test_space(\chamber_{2N};\C)$, 
we have
\begin{align} \label{eq::goal}
\langle \mathcal{D}^{(1)} \PartF_{\alpha}, \phi \rangle
:= \; & \int_{\chamber_{2N}} \PartF_{\alpha}(\boldsymbol{x}) \times (\mathcal{D}^{(1)})^* \phi (\boldsymbol{x}) \ud \boldsymbol{x} = 0 ,
\end{align}
where 
\begin{align*}
\mathcal{D}^{(1)} = \frac{\kappa}{2}\partial^2_1 + \sum_{j\neq 1}\left(\frac{2}{x_{j}-x_{1}}\partial_j - 
\frac{2h}{(x_{j}-x_{1})^{2}}\right) , \qquad
(\mathcal{D}^{(1)})^* := \frac{\kappa}{2}\partial^2_1 - \sum_{j\neq 1}\left(\frac{2}{x_{j}-x_{1}}\partial_j +
\frac{2h}{(x_{j}-x_{1})^{2}}\right)
\end{align*}
are respectively the partial differential operator in~\eqref{eq: multiple SLE PDEs} for $i = 1$, and its transpose.
A calculation shows that the two differential operators $A$ and $\mathcal{D}^{(1)}$ are related via
\begin{align*}
A \prod_{j\neq 1, b} y_j^h \times (x_b-x_1)^{2h}\times\phi(\boldsymbol{x})
= \prod_{j\neq 1, b} y_j^h \times (x_b-x_1)^{2h} \times \mathcal{D}^{(1)} \phi(\boldsymbol{x}),
\end{align*}
for any test function $\phi \in \Test_space(\chamber_{2N};\C)$. Therefore, for any $\boldsymbol{y} \in \R^{2N-1}$,
we have
\begin{align} \label{eq::relation_of_two_operators}
\begin{split}
\langle \mathcal{D}^{(1)} \PartF_{\alpha}, \phi \rangle
= \; & \int_{\chamber_{2N}} \PartF_{\alpha}(\boldsymbol{x}) \times 
\Big( \prod_{j\neq 1, b} y_j^{h} \times (x_b-x_1)^{2h} \Big) A^* \Big( \prod_{j\neq 1, b} y_j^{-h} \times (x_b-x_1)^{-2h} \Big)
 \phi (\boldsymbol{x}) \ud \boldsymbol{x}  \\
 = \; & \int_{\chamber_{2N}} F(\boldsymbol{x},\boldsymbol{y}) \times 
(A^* \tilde{\phi} )(\boldsymbol{x}, \boldsymbol{y}) \ud \boldsymbol{x} ,
\end{split}
\end{align}
where we defined 
$\tilde{\phi} (\boldsymbol{x}, \boldsymbol{y}) := \Big( \prod_{j\neq 1, b} y_j^{-h} \times (x_b-x_1)^{-2h} \Big) \phi (\boldsymbol{x})$.
Now, recalling that $\langle A f, \psi \rangle = 0$ for $f(\boldsymbol{x}, \boldsymbol{y}) := F (\boldsymbol{x}, \boldsymbol{1})$,
and choosing $\psi (\boldsymbol{x}, \boldsymbol{y}) 
:= \tilde{\phi} (\boldsymbol{x}, \boldsymbol{1}) = (x_b-x_1)^{-2h} \phi (\boldsymbol{x}) \in \Test_space$, 
we finally obtain
\begin{align*}
0 = \langle A f, \psi \rangle 
= \; & \int_{\chamber_{2N} \times (1/2,3/2)^{2N-1}} 
F(\boldsymbol{x}, \boldsymbol{1}) \times (A^* \tilde{\phi}) (\boldsymbol{x}, \boldsymbol{1}) \ud \boldsymbol{x} \ud \boldsymbol{y} 
 \\
= \; & \int_{(1/2,3/2)^{2N-1}} \ud \boldsymbol{y}  \int_{\chamber_{2N} } 
F(\boldsymbol{x}, \boldsymbol{1}) \times (A^* \tilde{\phi}) (\boldsymbol{x}, \boldsymbol{1}) \ud \boldsymbol{x} 
 \\
= \; & \int_{\chamber_{2N} } 
F(\boldsymbol{x}, \boldsymbol{1}) \times (A^* \tilde{\phi}) (\boldsymbol{x}, \boldsymbol{1}) \ud \boldsymbol{x}  \\
= \; & \langle \mathcal{D}^{(1)} \PartF_{\alpha}, \phi \rangle .
&& \text{[by~\eqref{eq::relation_of_two_operators} with $\boldsymbol{y} = \boldsymbol{1}$]} 
\end{align*}
Because the test function $\phi \in \Test_space(\chamber_{2N};\C)$ can be chosen arbitrarily,
this implies our goal~\eqref{eq::goal}, i.e., that $\mathcal{D}^{(1)} \PartF_{\alpha} = 0$ in the distributional sense. Since
$\PartF_{\alpha}$ is a distributional solution to the hypoelliptic partial differential equation~\eqref{eq: multiple SLE PDEs} for $i = 1$,
Proposition~\ref{prop::smoothness} now implies that $\PartF_{\alpha}$ is in fact a smooth solution.
\end{proof}

We are now ready to conclude:

\purepartitionexistence*
\begin{proof} 
The functions $\{\PartF_\alpha \colon \alpha \in \LP\}$
defined in~\eqref{eqn::purepartition_alpha_def} satisfy all of
the asserted defining properties: 
the 
stronger bound~\eqref{eqn::partitionfunction_positive}, 
partial differential equations $\mathrm{(PDE)}$~\eqref{eq: multiple SLE PDEs}, 
covariance $\mathrm{(COV)}$~\eqref{eq: multiple SLE Mobius covariance}, and 
asymptotics $\mathrm{(ASY)}$~\eqref{eq: multiple SLE asymptotics}
respectively follow from Lemmas~\ref{lem::partitionfunction_positive},
\ref{lem::partitionfunction_pde}, \ref{lem::partitionfunction_cov}, and~\ref{lem::partitionfunction_asy}.
Uniqueness then follows from Corollary~\ref{cor::purepartition_unique}.
Finally, the linear independence is the content of the next Proposition~\ref{prop: FK dual elements}.
\end{proof}

\begin{proposition}\label{prop: FK dual elements}
Let $\{ \FKdual_\alpha \colon \alpha \in \LP \}$ be the collection of linear functionals defined in~\eqref{eq: limit operation} 
and let $\{ \PartF_\alpha \colon \alpha \in \LP \}$ be the collection of functions defined in~\eqref{eqn::purepartition_alpha_def}.
Then, we have $\PartF_\alpha \in \LS_N$ and
\begin{align}\label{eq: FK dual basis}
\FKdual_\alpha ( \PartF_\beta ) = \delta_{\alpha,\beta} 
= & \begin{cases}
1, & \text{if } \beta = \alpha , \\
0, & \text{if }\beta\neq\alpha ,
\end{cases} 
\end{align}
for all $\alpha, \beta \in \LP_N$. In particular, the set $\{ \PartF_\alpha \colon \alpha \in \LP_N \}$ is linearly independent and it thus forms a basis for
the $\Catalan_N$-dimensional solution space $\mathcal{S}_N$ with dual basis $\{ \FKdual_\alpha \colon \alpha \in \LP_N \}$.
\end{proposition}
\begin{proof}
By the above proof, we have $\PartF_\alpha \in \LS_N$.
Property~\eqref{eq: FK dual basis} follows from the asymptotics properties $\mathrm{(ASY)}$~\eqref{eq: multiple SLE asymptotics}
of the functions $\PartF_\alpha$ from Lemma~\ref{lem::partitionfunction_asy}, and the last assertion follows immediately from this.
\end{proof}

In~\cite[Theorem 4.1]{KytolaPeoltolaPurePartitionSLE}, 
K.~Kyt\"ol\"a and E.~Peltola constructed candidates for the pure partition functions 
$\PartF_\alpha$ with $\kappa\in (0,8)\setminus\QQ$ using Coulomb gas techniques and 
a hidden quantum group symmetry on the solution space of
$\mathrm{(PDE)}$~\eqref{eq: multiple SLE PDEs} 
and $\mathrm{(COV)}$~\eqref{eq: multiple SLE Mobius covariance}, inspired by conformal field theory. 
S.~Flores and P.~Kleban proved independently and simultaneously
in~\cite{FloresKlebanSolutionSpacePDE1, FloresKlebanSolutionSpacePDE2, FloresKlebanSolutionSpacePDE3, FloresKlebanSolutionSpacePDE4} 
the existence of such functions for $\kappa\in (0,8)$, 
and argued that they can be found by inverting a certain system of linear equations.
However,  the functions constructed in these works were not shown to be positive.
As a by-product of Theorem~\ref{thm::purepartition_existence}, we establish positivity for these functions when $\kappa\in (0,4)$, 
thus identifying them with our functions of Theorem~\ref{thm::purepartition_existence}.

\subsection{Global Multiple SLEs are Local Multiple SLEs}
\label{subsec::global_vs_local}
In this section, we show that the global $\SLE_\kappa$ probability measures 
$\QQ_\alpha^\#$ constructed in Section \ref{subsec::multiplesle_existence} agree with another natural 
definition of multiple $\SLE$s --- the local $N$-$\SLE_\kappa$. 
The latter measures are defined in terms of their 
Loewner chain description, which allows one to treat the random curves as growth processes.
We first recall the definition of a local multiple $\SLE_\kappa$ from~\cite{DubedatCommutationSLE}
and~\cite[Appendix~A]{KytolaPeoltolaPurePartitionSLE}. 
 
\smallbreak
 
Let $(\Omega; x_1, \ldots, x_{2N})$ be a polygon. The localization neighborhoods $U_1, \ldots, U_{2N}$ 
are assumed to be closed subsets of $\overline{\Omega}$ such that $\Omega\setminus U_j$ are simply 
connected and $U_j\cap U_k=\emptyset$ for $j\neq k$. A \textit{local} $N$-$\SLE_{\kappa}$ in $\Omega$, 
started from $(x_1, \ldots, x_{2N})$ and localized in $(U_1, \ldots, U_{2N})$, 
is a probability measure on $2N$-tuples of oriented unparameterized curves 
$(\gamma_1, \ldots, \gamma_{2N})$. 
For convenience, we choose a parameterization of the curves by $t \in [0,1]$, so that 
for each $j$, the curve $\gamma_j \colon  [0,1]\to U_j$ starts at $\gamma_j(0)=x_j$ and ends at 
$\gamma_j(1)\in \partial(\Omega\setminus U_j)$. The local $N$-$\SLE_{\kappa}$ 
is an indexed collection of probability measures on $(\gamma_1, \ldots, \gamma_{2N})$: 
\[ \localSLE = \left(\localSLE^{(\Omega; x_1, \ldots, x_{2N})}_{(U_1, \ldots, U_{2N})}\right)_{\Omega; x_1, \ldots, x_{2N}; U_1, \ldots, U_{2N}} . \]
This collection 
is required to satisfy conformal invariance $\mathrm{(CI)}$,
domain Markov property $\mathrm{(DMP)}$, and absolute continuity of marginals with 
respect to the chordal $\SLE_{\kappa}$ $\mathrm{(MARG)}$:
\begin{itemize}[align=left]
\item[$\mathrm{(CI)}$] If 
$(\gamma_1, \ldots, \gamma_{2N})\sim \localSLE^{(\Omega; x_1, \ldots, x_{2N})}_{(U_1, \ldots, U_{2N})}$ 
and $\varphi \colon \Omega\to \varphi(\Omega)$ is a conformal map, then 
\[(\varphi(\gamma_1), \ldots, \varphi(\gamma_{2N}))\sim \localSLE^{(\varphi(\Omega); \varphi(x_1), \ldots, \varphi(x_{2N}))}_{(\varphi(U_1), \ldots, \varphi(U_{2N}))}.\]

\item[$\mathrm{(DMP)}$] 
Let $\tau_j$ be stopping times for $\gamma_j$, for $j \in \{1, \ldots, N\}$. 
Given initial segments $(\gamma_1[0,\tau_1], \ldots, \gamma_{2N}[0,\tau_{2N}])$, the conditional law of the
remaining parts $(\gamma_1|_{[\tau_1, 1]}, \ldots, \gamma_{2N}|_{[\tau_{2N}, 1]})$ is 
$\localSLE^{(\tilde{\Omega}; \tilde{x}_1, \ldots, \tilde{x}_{2N})}_{(\tilde{U}_1, \ldots, \tilde{U}_{2N})}$, 
where $\tilde{\Omega}$ is the component of 
$\Omega \setminus \bigcup_j \gamma_j[0,\tau_j]$
containing all tips $\tilde{x}_j=\gamma_j(\tau_j)$ on 
its boundary, and $\tilde{U}_j=U_j\cap\tilde{\Omega}$. 

\item[$\mathrm{(MARG)}$] 
There exist smooth functions $F_j \colon \chamber_{2N} \to \R$, for 
$j \in \{1, \ldots, 2N\}$, such that for 
the domain $\Omega=\HH$, boundary points $x_1<\cdots<x_{2N}$, and their localization neighborhoods 
$U_1, \ldots, U_{2N}$, the marginal law of $\gamma_j$ under 
$\localSLE^{(\HH; x_1, \ldots, x_{2N})}_{(U_1,\ldots, U_{2N})}$ is the Loewner evolution driven by 
$W_t$ which solves
\begin{align}\label{eqn::localmultiple_marginal_sde}
\begin{split}
\ud W_t 
= \; & \sqrt{\kappa} \ud B_t 
+ F_j \left(V_t^1, \ldots, V_t^{j-1}, W_t, V_t^{j+1},\ldots, V_t^{2N}\right) \ud t, 
\qquad W_0 = x_j , \\
\ud V_t^i = \; & \frac{2 \ud t}{V_t^i-W_t}, \qquad V_0^i = x_i, \quad \text{for } i\neq j. 
\end{split}
\end{align}
\end{itemize}
\begin{remark}
It follows from the definition that the local $N$-$\SLE_{\kappa}$ is consistent
under restriction to smaller localization neighborhoods, 
see~\cite[Proposition~A.2]{KytolaPeoltolaPurePartitionSLE}.
\end{remark}

J.~Dub\'edat proved in~\cite{DubedatCommutationSLE} that the local $N$-$\SLE_\kappa$ processes are 
classified by partition functions $\PartF$ as described in the next proposition. 
The convex structure of the set of the local $N$-$\SLE_\kappa$ was further studied in~\cite[Appendix~A]{KytolaPeoltolaPurePartitionSLE}. 

\begin{proposition} \label{prop::local_multiple_SLE}
Let $\kappa>0$. 
\begin{enumerate}
\item Suppose $\localSLE$ is a local $N$-$\SLE_{\kappa}$. Then, there exists a function 
$\PartF \colon \chamber_{2N} \to \Rpos$ satisfying $\mathrm{(PDE)}$~\eqref{eq: multiple SLE PDEs} 
and $\mathrm{(COV)}$~\eqref{eq: multiple SLE Mobius covariance}, such that 
for all $j \in \{1,\ldots,2N\}$,
the drift functions in $\mathrm{(MARG)}$ take the form $F_j=\kappa\partial_j \log\PartF$. 
Such a function $\PartF$ is 
determined up to a multiplicative constant. 

\item Suppose $\PartF \colon \chamber_{2N} \to \Rpos$ satisfies 
$\mathrm{(PDE)}$~\eqref{eq: multiple SLE PDEs} and 
$\mathrm{(COV)}$~\eqref{eq: multiple SLE Mobius covariance}. Then, the random collection of curves 
obtained by the Loewner chain in $\mathrm{(MARG)}$ with $F_j=\kappa\partial_j \log\PartF$,
for all $j \in \{1,\ldots,2N\}$, is a local $N$-$\SLE_{\kappa}$. 
Two functions $\PartF$ and $\tilde{\PartF}$ give rise to the same local $N$-$\SLE_{\kappa}$ if and only if $\PartF=\text{const.}\times\tilde{\PartF}$. 
\end{enumerate}
\end{proposition}
\begin{proof}
This 
follows from results in~\cite{DubedatCommutationSLE, GrahamSLE, KytolaVirasoroSLE} and~\cite[Theorem A.4]{KytolaPeoltolaPurePartitionSLE}.
\end{proof}

For each (normalized) partition function $\PartF \colon \chamber_{2N} \to \Rpos$, 
that is, a solution to $\mathrm{(PDE)}$~\eqref{eq: multiple SLE PDEs} and 
$\mathrm{(COV)}$~\eqref{eq: multiple SLE Mobius covariance}, we call the collection $\localSLE$ 
of probability measures for which we have in $\mathrm{(MARG)}$ $F_j=\kappa \partial_j \log\PartF$, 
for all $j \in \{1,\ldots,2N\}$,  \textit{the local $N$-$\SLE_{\kappa}$ with partition function $\PartF$}. 
Next, we prove that our construction of the global $N$-$\SLE_\kappa$ measures in Section~\ref{sec::characterization}
is consistent with this local definition.

\begin{lemma}\label{lem::global_is_local}
Let $\kappa\in (0,4]$. Any global $N$-$\SLE_{\kappa}$ satisfying
$\mathrm{(MARG)}$ is a local $N$-$\SLE_{\kappa}$ when it is restricted to any  
localization neighborhoods. For any $\alpha \in \LP_N$, the restriction of the
global $N$-$\SLE_{\kappa}$ probability measure $\QQ_{\alpha}^{\#}$ 
associated to $\alpha$ 
(constructed in Proposition~\ref{prop::global_existence}) to any localization neighborhoods
coincides with the local $N$-$\SLE_{\kappa}$ with 
partition function $\PartF_\alpha$ given by~\eqref{eqn::purepartition_alpha_def}.
\end{lemma}
\begin{proof}
Fix $\Omega=\HH$, boundary points $x_1<\cdots<x_{2N}$, localization neighborhoods 
$(U_1, \ldots, U_{2N})$, and a link pattern $\alpha\in\LP_N$. 
Suppose that $(\eta_1, \ldots, \eta_N)$ is a global $N$-$\SLE_{\kappa}$ associated to~$\alpha$.
Given any link $\link{a}{b} \in \alpha$, let $\eta$ be the curve connecting $x_{a}$ to $x_{b}$,
and denote by $\overline{\eta}$ the time-reversal of $\eta$. Let $\tau$ be the first time when $\eta$ exits $U_{a}$, and define 
$\gamma_{a}$ to be the curve $(\eta(t): 0\le t\le \tau)$. Let $\overline{\tau}$ be the first time 
when $\overline{\eta}$ exits $U_{b}$, and define $\gamma_{b}$ to be the curve
$(\overline{\eta}(t): 0\le t\le \overline{\tau})$. By conformal invariance of the $\SLE_\kappa$, 
the law of the collection $(\gamma_1, \ldots, \gamma_{2N})$ satisfies $\mathrm{(CI)}$.
It also satisfies $\mathrm{(DMP)}$,
thanks to the domain Markov property and reversibility of the $\SLE_{\kappa}$.
Therefore, any global $N$-$\SLE_{\kappa}$ satisfying $\mathrm{(MARG)}$ is a local $N$-$\SLE_{\kappa}$ when 
restricted to any localization neighborhoods.

Suppose then that $(\eta_1,\ldots, \eta_N)\sim\QQ_{\alpha}^{\#}(\HH; x_1, \ldots, x_{2N})$ and define 
$(\gamma_1, \ldots, \gamma_{2N})$ 
as above. 
We only need to check 
the property $\mathrm{(MARG)}$. Without loss of generality, we 
prove it for $\gamma_1$. From the proof of Lemma~\ref{lem::partitionfunction_pde}, we see that 
the marginal law of $\gamma_1$ under $\QQ_{\alpha}^{\#}$ is 
absolutely continuous with respect to the $\SLE_{\kappa}$ in $\HH$ from $x_1$ to $\infty$, 
and the Radon-Nikodym derivative 
is given by the local martingale
\[\prod_{j=2}^{2N} g_t'(x_j)^h \times \PartF_{\alpha}(W_t, g_t(x_2), \ldots, g_t(x_{2N})).\]
This implies that the curve $\gamma_1$ has the same driving function as in $\mathrm{(MARG)}$ for $j=1$, 
with drift function 
$F_1=\kappa \partial_1\log \PartF_{\alpha}$.
Because, by Lemma~\ref{lem::partitionfunction_pde},
$\PartF_{\alpha}$ is smooth, $F_1$ is smooth. 
This completes the proof.   
\end{proof}

We finish this section with the proofs of Theorem~\ref{thm::global_existence} and Corollary~\ref{cor::localmultiplesle}.

\globalexistencethm*
\begin{proof}
A global $N$-$\SLE_{\kappa}$ 
was constructed in Proposition~\ref{prop::global_existence}. 
The two properties 
were proved respectively 
in Propositions~\ref{prop::multiplesle_boundary_perturbation} and~\ref{prop::GeneralCascadeProp}.
That the local and global $\SLE_\kappa$ 
agree follows from Lemma~\ref{lem::global_is_local}.
\end{proof}

Corollary~\ref{cor::localmultiplesle} describes the convex structure of the local multiple $\SLE$ probability measures.
If $\PartF_1$ and $\PartF_2$ are two partition functions, i.e., positive solutions to $\mathrm{(PDE)}$~\eqref{eq: multiple SLE PDEs} and 
$\mathrm{(COV)}$~\eqref{eq: multiple SLE Mobius covariance},
set $\PartF = \PartF_1 + \PartF_2$ and denote by $\localSLE, \localSLE_1, \localSLE_2$ the local multiple $\SLE$s associated to $\PartF, \PartF_1, \PartF_2$, respectively. 
Then, the probability measure $\localSLE$ can be 
written as
the following convex combination; see \cite[Theorem~A.4(c)]{KytolaPeoltolaPurePartitionSLE}: 
\begin{align*}
\localSLE^{(\Omega; x_1, \ldots, x_{2N})}_{(U_1, \ldots, U_{2N})} = 
\frac{\PartF_1(\Omega; x_1, \ldots, x_{2N})}{\PartF(\Omega; x_1, \ldots, x_{2N})} \, (\localSLE_1)^{(\Omega; x_1, \ldots, x_{2N})}_{(U_1, \ldots, U_{2N})} 
\, + \, \frac{\PartF_2(\Omega; x_1, \ldots, x_{2N})}{\PartF(\Omega; x_1, \ldots, x_{2N})} \, (\localSLE_2)^{(\Omega; x_1, \ldots, x_{2N})}_{(U_1, \ldots, U_{2N})}.
\end{align*}

\localmultiplesle*
\begin{proof}
This is a consequence of Theorem~\ref{thm::purepartition_existence} and Proposition~\ref{prop::local_multiple_SLE}. 
\end{proof}

\subsection{Loewner Chains Associated to Pure Partition Functions}
\label{subsec::purepartition_thm}
In this section, we show that the Loewner chain associated to $\PartF_\alpha$ 
is almost surely generated by a continuous curve up to and including the continuation threshold.
This is a consequence of the strong bound~\eqref{eqn::partitionfunction_positive} in Theorem~\ref{thm::purepartition_existence}.

\begin{proposition} \label{prop::loewnerchain_purepartition}
Let $\kappa\in (0,4]$ and $\alpha\in\LP_N$. Assume that $\link{a}{b}\in\alpha$. 
Let $W_t$ be the solution to the following SDEs:
\begin{align}\label{eqn::loewnerchain_purepartition}
\begin{split}
\ud W_t 
= \; & \sqrt{\kappa} \ud B_t 
+\kappa\partial_{a}\log\PartF_{\alpha} \left(V_t^1, \ldots, V_t^{a-1}, W_t, V_t^{a+1}, \ldots, V_t^{2N}\right) \ud t, 
\qquad W_0 = x_a , \\
\ud V_t^i = \; & \frac{2 \ud t}{V_t^i-W_t}, \qquad V_0^i = x_i, \quad \text{for } i\neq a. 
\end{split}
\end{align}
Then, the Loewner chain driven by $W_t$ is well-defined up to the swallowing time $T_b$ of $x_b$.
Moreover, it is almost surely generated by a continuous curve up to and including $T_b$. 
This curve has the same law as the one connecting $x_a$ and $x_b$ in the global multiple $\SLE_{\kappa}$ associated to $\alpha$ 
in the polygon $(\HH; x_1, \ldots, x_{2N})$. 
\end{proposition}

\begin{proof}
Without loss of generality, we assume that $a=1$. 
Consider the Loewner chain $K_t$ driven by $W_t$. Let $\gamma$ be the chordal $\SLE_{\kappa}$ in $\HH$ from $x_1$ to $x_b$. 
For each $i \in \{2, \ldots, 2N\}$, let $T_i$ be the swallowing time of the point $x_i$ and define $T$ to be the minimum of all $T_i$ for $i\neq 1$. 
It is clear that the Loewner chain is well-defined up to $T$. 
For $t<T$, the law of $K_t$ is that of the curve $\gamma[0,t]$ weighted by the martingale 
\begin{align*}
M_t := \prod_{i \neq 1, b} g_t'(x_i)^h \times \PartF_{\alpha}(W_t, g_t(x_2), \ldots, g_t(x_{2N})) \times (g_t(x_b)-W_t)^{2h}. 
\end{align*}
It follows from the bound~\eqref{eqn::partitionfunction_positive} that 
$M_t$ is in fact a bounded martingale: for any $t<T$, we have
\begin{align*}
M_t& \le \prod_{i\neq 1, b} g_t'(x_i)^h \times \prod_{\substack{\link{c}{d} \in \alpha , \\ \link{c}{d} \neq \link{1}{b} }}(g_t(x_{d})-g_t(x_{c}))^{-2h} 
&& 
\text{[by~\eqref{eqn::partitionfunction_positive}]}\\
&=  \prod_{\substack{\link{c}{d} \in \alpha , \\ \link{c}{d} \neq \link{1}{b} }} 
\left(\frac{g_t'(x_{c})g_t'(x_{d})}{(g_t(x_{d})-g_t(x_{c}))^2}\right)^h
\le \prod_{\substack{\link{c}{d} \in \alpha , \\ \link{c}{d} \neq \link{1}{b} }} (x_{d}-x_{c})^{-2h}. 
&& 
\text{[by~\eqref{eqn::poissonkernel_mono}]}
\end{align*}

Now, $\gamma$ is a continuous curve up to and including the swallowing time of $x_b$, and almost surely, it does not hit any other point in $\R$. 
Combining this with the fact that $(M_t, t\le T)$ is bounded, the same property is also true for the Loewner chain $(K_t, t\le T)$, and we have $T=T_b$. 
This shows that the Loewner chain driven by $W_t$ is almost surely generated by a continuous curve up to and including $T_b$. 

Finally, let $\eta$ be the curve connecting $x_1$ and $x_b$ in the global multiple $\SLE_{\kappa}$ associated to $\alpha$.
From the proof of Lemma~\ref{lem::global_is_local}, we know that the Loewner chain $K_t$ has the same law as $\eta[0,t]$ for any $t<T_b$. 
Since both $K$ and $\eta$ are continuous curves up to and including the swallowing time of $x_b$, 
this implies that $(K_t, t\le T_b)$ has the same law as $\eta$.  This completes the proof.
\end{proof}

\subsection{Symmetric Partition Functions}
\label{subsec::total_pf}
In this section, we collect some results concerning the symmetric partition functions
\begin{align}\label{eq: total SLE partition function}
\PartF^{(N)} := \sum_{\alpha\in \LP_N} \PartF_{\alpha},
\end{align}
where $\{\PartF_\alpha \colon \alpha \in \LP\}$ is the collection of functions of 
Theorem~\ref{thm::purepartition_existence}. 
In the range $\kappa \in (0,4]$, the functions $\PartF^{(N)}$ have explicit 
formulas for $\kappa=2,3$, and $4$, given respectively in 
Lemmas~\ref{lem::ztotal_lerw}, \ref{lem::ztotal_ising} and~\ref{lem: gff symmetric pff}. 

\begin{lemma}\label{lem::totalpartitionfunction}
The collection $\{\PartF^{(N)}\colon N\ge 0\}$ of functions $\PartF^{(N)} \colon \chamber_{2N} \to \Rpos$ satisfies 
$\PartF^{(N)} \in \LS_N$ and $\PartF^{(0)} = 1$, and the asymptotics property
\begin{align}\label{eqn::total_asy_refined}
\lim_{\substack{\tilde{x}_j , \tilde{x}_{j+1} \to \xi, \\ \tilde{x}_i\to x_i \text{ for } i \neq j, j+1}} 
\frac{\PartF^{(N)}(\tilde{x}_1 , \ldots , \tilde{x}_{2N})}{(\tilde{x}_{j+1} - \tilde{x}_j)^{-2h}} 
=\PartF^{(N-1)}(x_{1},\ldots,x_{j-1},x_{j+2},\ldots,x_{2N}) ,
\end{align} 
for all $j \in \{1, \ldots, 2N-1 \}$ and 
$x_1 < \cdots < x_{j-1} < \xi < x_{j+2} < \cdots < x_{2N}$.
In particular, we have
\begin{align}\label{eq: total SLE asymptotics}
\lim_{x_j ,x_{j+1} \to \xi} 
\frac{\PartF^{(N)}(x_1 , \ldots , x_{2N})}{(x_{j+1} -x_j)^{-2h}} 
=\PartF^{(N-1)}(x_{1},\ldots,x_{j-1},x_{j+2},\ldots,x_{2N}) .
\end{align}
\end{lemma}
\begin{proof}
The normalization $\PartF^{(0)} = 1$ is clear, and 
we have $\PartF^{(N)} \in \LS_N$ by Proposition~\ref{prop: FK dual elements}.
The asymptotics~\eqref{eqn::total_asy_refined} and~\eqref{eq: total SLE asymptotics} follow from the asymptotics of 
the pure partition functions $\PartF_{\alpha}$ from Lemma~\ref{lem::partitionfunction_asy}.
\end{proof}

\begin{corollary}\label{cor::totalpartitionfunction_existence}
Let $\{F^{(N)}\colon N\ge 0\}$ be a collection of functions $F^{(N)} \in \LS_N$ 
satisfying 
the asymptotics property~\eqref{eq: total SLE asymptotics} with the normalization $F^{(0)} = 1$.
Then we have
$F^{(N)} = \PartF^{(N)}$ for all $N\ge 0$. 
\end{corollary}
\begin{proof}
After replacing $\mathrm{(ASY)}$~\eqref{eq: multiple SLE asymptotics} 
by~\eqref{eq: total SLE asymptotics},
the proof of Corollary~\ref{cor::purepartition_unique} applies verbatim to show that the collection 
$\{F^{(N)}\colon N\ge 0\}$ is unique. 
Lemma~\ref{lem::totalpartitionfunction} then shows that we have
$F^{(N)} = \PartF^{(N)}$ for all $N\ge 0$.
\end{proof}

Next we give algebraic formulas for the symmetric partition functions for $\kappa=2,3$ and $4$. 
To state them for $\kappa = 2,3$, we use the following notation.
Let $\Pi_N$ be the set of all partitions
$\varpi = \{\{a_1,b_1\} , \ldots, \{a_N, b_N\}\}$ of 
$\{1,\ldots,2N\}$ into $N$ disjoint two-element subsets $\{a_j,b_j\} \subset \{1,\ldots,2N\}$, 
with the convention that 
$a_j < b_j$, for all $j \in \{ 1 , \ldots, N\}$, and $a_1 < a_2 < \cdots < a_N$.
Denote by $\mathrm{sgn}(\varpi)$ the sign of the partition $\varpi$ 
defined as the sign of the product $\prod (a-c) (a-d) (b-c) (b-d)$ over pairs of distinct elements 
$\{a,b\},\{c,d\} \in \varpi$.

\begin{lemma}\label{lem::ztotal_lerw}
Let $\kappa = 2$. For all $N \geq 1$, we have 
\begin{align}\label{eq: LERW symmetric pff}
\PartF^{(N)}_{\mathrm{LERW}}(x_1, \ldots, x_{2N}) 
= \sum_{\varpi \in \Pi_N} \mathrm{sgn}(\varpi) 
\; \det \left( \frac{1}{(x_{b_j} - x_{a_i})^2} \right)_{i,j=1}^N .
\end{align}
In particular, $\PartF^{(N)}_{\mathrm{LERW}}(x_1, \ldots, x_{2N}) > 0$.
\end{lemma}
\begin{proof}
Consider the function 
$\tilde{\PartF}^{(N)}_{\mathrm{LERW}} := \sum_\varpi \mathrm{sgn}(\varpi) \det \big((x_{b_j} - x_{a_i})^{-2}\big)$ 
on the right-hand side.
By \cite[Lemmas~4.4~and~4.5]{KKP:Correlations_in_planar_LERW_and_UST_and_combinatorics_of_conformal_blocks}
and linearity, this function 
satisfies $\mathrm{(PDE)}$~\eqref{eq: multiple SLE PDEs} and 
$\mathrm{(COV)}$~\eqref{eq: multiple SLE Mobius covariance} with $\kappa = 2$.
It also clearly satisfies the bound~\eqref{eqn::powerlawbound}.
Also, if $N = 0$, then we have $\tilde{\PartF}^{(0)}_{\mathrm{LERW}} = 1$.
Thus, by Corollary~\ref{cor::totalpartitionfunction_existence},
it suffices to show that $\tilde{\PartF}^{(N)}_{\mathrm{LERW}}$ also 
satisfies the asymptotics property~\eqref{eq: total SLE asymptotics} with $\kappa = 2$.
To prove this, fix $j \in \{1, \ldots, 2N-1 \}$ and $\xi \in (x_{j-1}, x_{j+2})$.
The limit in~\eqref{eq: total SLE asymptotics} with $\kappa = 2$ reads
\begin{align}
& \; \lim_{x_j , x_{j+1} \to \xi} 
(x_{j+1} - x_j)^{2} \; \tilde{\PartF}^{(N)}_{\mathrm{LERW}} (x_1, \ldots, x_{2N}) 
\nonumber \\
=\; &\lim_{x_j , x_{j+1} \to \xi} 
(x_{j+1} - x_j)^{2} \sum_{\varpi \in \Pi_N} \mathrm{sgn}(\varpi) 
\; \det \left( \frac{1}{(x_{b_k} - x_{a_l})^2} \right)_{l,k=1}^N 
\nonumber \\
=\; & \lim_{x_{j},x_{j+1}\to\xi}(x_{j+1}-x_{j})^{2} \sum_{\varpi \in \Pi_N} \mathrm{sgn}(\varpi) 
\sum_{\sigma \in \mathfrak{S}_N} \mathrm{sgn}(\sigma)
\prod_{k=1}^N \frac{1}{(x_{b_k} - x_{a_{\sigma(k)}})^2} 
\nonumber \\
=\; & \lim_{x_{j},x_{j+1}\to\xi}(x_{j+1}-x_{j})^{2} \sum_{\sigma \in \mathfrak{S}_N} \mathrm{sgn}(\sigma)
\sum_{\varpi \in \Pi_N} \mathrm{sgn}(\varpi) 
\prod_{k=1}^N \frac{1}{(x_{b_k} - x_{a_{\sigma(k)}})^2} ,
\label{eq: limit for LERW pf}
\end{align}
where $\mathfrak{S}_N$ denotes the group of permutations of $\{1,\ldots,N\}$. 
To evaluate this limit, for any pair of indices $k,l \in \{1,2,\ldots,N\}$, with 
$k \neq l$, we define the bijection
\begin{align*}
\varphi_{l,k} \; \colon \; & 
\{\varpi \in \Pi_N \colon j = b_k \text{ and } j + 1 = a_l \text{ in } \varpi \} 
\; \longrightarrow \;
\{ \varpi \in \Pi_N \colon j = a_l \text{ and } j + 1 = b_k  \text{ in } \varpi \} , \\
\varphi_{l,k}(\varpi) \; := \; & 
\Big( \varpi \setminus \{\{j',j\},\{j+1,(j+1)'\}\}\Big) \cup \{\{j',j+1\},\{j,(j+1)'\}\},
\end{align*}
where $j'$ and $(j+1)'$ denote the pairs of $j$ and $j+1$ in $\varpi$, respectively.
Note that $\mathrm{sgn}(\varphi_{l,k}(\varpi)) = -\mathrm{sgn}(\varpi)$.

Consider a term in~\eqref{eq: limit for LERW pf} with fixed $\sigma \in \mathfrak{S}_N$.
Only terms where in $\varpi = \{\{a_1,b_1\} , \ldots, \{a_N, b_N\}\}$ we have 
for some $k \in \{1,2,\ldots,N\}$ either $j = a_{\sigma(k)}$ and $j+1 = b_k$, 
or $j = b_k$ and $j+1 = a_{\sigma(k)}$, can have a non-zero limit. With the bijections
$\varphi_{\sigma(k),k}$, we may cancel all terms for which $\sigma(k) \neq k$.
Thus, we are left with the terms for which $\{j,j+1\} = \{a_k,b_k\} \in \varpi$ and
$\sigma(k) = k$, which allows us to reduce the sums over $\sigma \in \mathfrak{S}_N$
and $\varpi = \{\{a_1,b_1\} , \ldots, \{a_N, b_N\}\} \in \Pi_N$ into sums over 
$\tau \in \mathfrak{S}_{N-1}$ and 
$\hat{\varpi} = \{\{c_1,d_1\} , \ldots, \{c_{N-1}, d_{N-1}\}\} \in \Pi_{N-1}$, to obtain 
the asserted asymptotics property~\eqref{eq: total SLE asymptotics} with $\kappa = 2$:
\begin{align*}
& \; \lim_{x_j , x_{j+1} \to \xi} 
(x_{j+1} - x_j)^{2} \;  \tilde{\PartF}^{(N)}_{\mathrm{LERW}} (x_1, \ldots, x_{2N})\\
= \; & \sum_{\varpi \colon \{j,j+1\} \in \varpi} \mathrm{sgn}(\varpi) 
\sum_{\tau \in \mathfrak{S}_{N-1}} \mathrm{sgn}(\tau)
\lim_{x_{j},x_{j+1}\to\xi}(x_{j+1}-x_{j})^{2} \prod_{\substack{1 \leq k \leq N, \\ b_k \neq j+1}}
\frac{1}{(x_{b_k} - x_{a_{\tau(k)}})^2} \\
= \; & \sum_{\hat{\varpi} \in \Pi_{N-1}} \mathrm{sgn}(\hat{\varpi}) 
\sum_{\tau \in \mathfrak{S}_{N-1}} \mathrm{sgn}(\tau)
\prod_{k=1}^{N-1} \frac{1}{(x_{d_k} - x_{c_{\tau(k)}})^2} \\
= \; & \sum_{\hat{\varpi} \in \Pi_{N-1}} \mathrm{sgn}(\hat{\varpi}) 
\; \det \left( \frac{1}{(x_{c_k} - x_{d_l})^2} \right)_{k,l=1}^{N-1} \\
= \; & \tilde{\PartF}^{(N-1)}_{\mathrm{LERW}}(x_{1},\ldots,x_{j-1},x_{j+2},\ldots,x_{2N}) .
\end{align*}
This concludes the proof.
\end{proof}

\begin{lemma}\label{lem::ztotal_ising}
Let $\kappa = 3$. For all $N \geq 1$, we have 
\begin{align}\label{eq: ising symmetric pff}
\PartF^{(N)}_{\mathrm{Ising}}(x_1, \ldots, x_{2N}) 
= \mathrm{pf} \left(  \frac{1}{x_{j}-x_{i}} \right)_{i,j=1}^{2N}
= \sum_{\varpi \in \Pi_N} 
\mathrm{sgn}(\varpi) \prod_{\{a,b\} \in \varpi} \frac{1}{x_{b}-x_{a}}  .
\end{align}
In particular, $\PartF^{(N)}_{\mathrm{Ising}}(x_1, \ldots, x_{2N}) > 0$.
\end{lemma}
\begin{proof}
It was proved in~\cite[Proposition~4.6]{KytolaPeoltolaPurePartitionSLE}
that the function 
$\tilde{\PartF}^{(N)}_{\mathrm{Ising}} 
:= \sum_\varpi \mathrm{sgn}(\varpi) \Big(\prod \frac{1}{x_{b}-x_{a}} \Big)$
on the right-hand side satisfies $\mathrm{(PDE)}$~\eqref{eq: multiple SLE PDEs} and 
$\mathrm{(COV)}$~\eqref{eq: multiple SLE Mobius covariance} with $\kappa = 3$, and that it also has 
the asymptotics property~\eqref{eq: total SLE asymptotics} with $\kappa = 3$.
Moreover, this function obviously satisfies the bound~\eqref{eqn::powerlawbound}, and if $N = 0$, 
then we have $\tilde{\PartF}^{(0)}_{\mathrm{Ising}} = 1$.
The claim then follows from Corollary~\ref{cor::totalpartitionfunction_existence}.
\end{proof}

\begin{lemma} \label{lem: gff symmetric pff}
Let $\kappa = 4$. For all $N \geq 1$, we have 
\begin{align}\label{eq: gff symmetric pff}
\PartF^{(N)}_{\GFF}(x_1, \ldots, x_{2N}) 
= \prod_{1\leq k<l\leq 2N}(x_l-x_k)^{\frac{1}{2} (-1)^{l-k}} .
\end{align}
\end{lemma}
\begin{proof}
It was proved in~\cite[Proposition~4.8]{KytolaPeoltolaPurePartitionSLE}
that the function 
$\tilde{\PartF}^{(N)}_{\GFF} 
:= \prod_{k<l} (x_l-x_k)^{\frac{1}{2} (-1)^{l-k}}$
on the right-hand side satisfies $\mathrm{(PDE)}$~\eqref{eq: multiple SLE PDEs} and 
$\mathrm{(COV)}$~\eqref{eq: multiple SLE Mobius covariance} with $\kappa = 4$, and that it also has the asymptotics 
property~\eqref{eq: total SLE asymptotics} with $\kappa = 4$.
Moreover, this function obviously satisfies the bound~\eqref{eqn::powerlawbound}, and if $N = 0$, 
then we have $\tilde{\PartF}^{(0)}_{\GFF} = 1$.
The claim then follows from Corollary~\ref{cor::totalpartitionfunction_existence}.
\end{proof}

\section{Gaussian Free Field}
\label{sec::levellines_gff}
This section is devoted to the study of the level lines of the Gaussian free field (GFF) with
alternating boundary data, generalizing the Dobrushin boundary data $-\lambda, +\lambda$
on two complementary boundary segments to $-\lambda, +\lambda, \ldots, -\lambda, +\lambda$ on $2N$
boundary segments. Much of these level lines is already known: a level line starting from
a boundary point is an $\SLE_{4}(\underline{\rho})$ process, 
and the level lines can be coupled with the GFF in such a way that they are almost surely 
determined by the field~\cite{DubedatSLEFreefield, SchrammSheffieldContinuumGFF, MillerSheffieldIG1}. 

We are interested in the probabilities that the level
lines form a particular connectivity pattern, encoded in $\alpha \in \LP_N$. The main result of this 
section, Theorem~\ref{thm::multiple_sle_4}, states that this probability is given by 
the pure partition functions $\PartF_\alpha$ for multiple $\SLE_\kappa$ with $\kappa = 4$.
We prove Theorem~\ref{thm::multiple_sle_4} in Section~\ref{subsec::gff_levellines_Plk}. 
In Section~\ref{sec::pure_pf_for_sle4}, we find explicit formulas for these connection probabilities, 
see~\eqref{eq: Pure partition function for kappa 4} in Theorem~\ref{thm: sle4 pure pffs}.

\subsection{Level Lines of GFF} \label{sub:GFF_pre}

In this section, we introduce the Gaussian free field and its level lines and summarize some of their useful properties.
We refer to the literature~\cite{SheffieldGFFMath, SchrammSheffieldContinuumGFF, MillerSheffieldIG1, WangWuLevellinesGFFI} for details.

To begin, we discuss $\SLE$s with multiple force points (different from multiple $\SLE$s) --- the \textit{$\SLE_{\kappa}(\underline{\rho})$ processes}. They
are variants of the $\SLE_{\kappa}$ where one keeps track of additional points on the boundary. 
Let $\underline{y}^L=(y^{l, L}<\cdots<y^{1,L}\le 0)$ and $\underline{y}^R=(0\le y^{1,R}<\cdots<y^{r,R})$ and 
$\underline{\rho}^L=(\rho^{l,L},\cdots, \rho^{1,L})$ and $\underline{\rho}^R=(\rho^{1,R},\cdots, \rho^{r,R})$, 
where $\rho^{i,q}\in\R$, for $q\in\{L, R\}$ and $i \in \N$.
An $\SLE_{\kappa}(\underline{\rho}^L;\underline{\rho}^R)$ process with force points 
$(\underline{y}^L;\underline{y}^R)$ is the Loewner evolution driven by $W_t$ that solves
the following system of integrated SDEs: 
\begin{align}
\label{eq: SLE_kappa_rho}
\begin{split}
W_t = \; & \sqrt{\kappa} B_t \, + \, 
\sum_{i=1}^l 
\int_0^t\frac{\rho^{i,L} ds}{W_s-V_s^{i,L}} \, + \, 
\sum_{i=1}^r 
\int_0^t\frac{\rho^{i,R} ds}{W_s-V_s^{i,R}} , \\
V^{i,q}_t = \; & y^{i,q} \, + \,  \int_0^t\frac{2ds}{V^{i,q}_s-W_s} , 
\quad \text{for }q\in\{L, R\}  \text{ and } i \in \N,
\end{split}
\end{align}
where $B_t$ is the one-dimensional Brownian motion. Note that the process $V_t^{i,q}$ is 
the evolution of the point $y^{i,q}$, and we may write $g_t(y^{i,q})$ for $V_t^{i,q}$. 
We define the \textit{continuation threshold} 
of the $\SLE_{\kappa}(\underline{\rho}^L;\underline{\rho}^R)$ 
to be the infimum of the time $t$ for which 
\[ \text{either} \quad \sum_{i \; : \; V^{i,L}_t=W_t} \rho^{i,L}\le -2 , 
\quad \text{or} \quad \sum_{i \; : \; V^{i,R}_t=W_t} \rho^{i,R}\le -2 . \]
By~\cite{MillerSheffieldIG1}, 
the $\SLE_{\kappa}(\underline{\rho}^L;\underline{\rho}^R)$
process is well-defined up to the continuation threshold, and it is almost surely
generated by a continuous curve up to and including the continuation threshold.

\smallbreak

Let $D\subsetneq \C$ be a non-empty simply connected domain. 
For two functions $f,g\in L^2(D)$, we denote by $(f,g)$ their inner product in $L^2(D)$,
that is, $(f,g) := \int_D f(z)g(z)d^2z$, where $d^2z$ is the Lebesgue area measure. We denote by $H_s(D)$ 
the space of real-valued smooth functions which are compactly supported in $D$. This space has a 
\textit{Dirichlet inner product} defined by
\[ (f,g)_{\nabla} := \frac{1}{2\pi} \int_D \nabla f(z) \cdot \nabla g(z) \ud^2 z . \]
We denote by $H(D)$ the Hilbert space completion of $H_s(D)$
with respect to the Dirichlet inner product. 

The \textit{zero-boundary $\GFF$} on $D$ is a random sum of the form $\gff =\sum_{j=1}^{\infty} \zeta_j f_j$, 
where $ \zeta_j $ are i.i.d. standard normal random variables and
$(f_j)_{j\geq 0}$ an orthonormal basis for $H(D)$. This sum almost surely diverges within $H(D)$; 
however, it does converge almost surely in the space of distributions --- that is, as $n \to \infty$, the limit of 
$\sum_{j=1}^n  \zeta_j (f_j,g)$ exists almost surely for all $g \in H_s(D)$
and we may define $(\gff,g) := \sum_{j=1}^\infty  \zeta_j (f_j,g)$.
The limiting value as a function of $g$ is almost surely a continuous functional on $H_s(D)$.
In general, for any harmonic function $\gff_0$ on $D$, we define the \textit{$\GFF$ with boundary data $\gff_0$} by  
$\gff := \tilde{\gff} + \gff_0$ where $\tilde{\gff}$ is the zero-boundary $\GFF$ on $D$. 
 
We next introduce the level lines of the $\GFF$ and list some of 
their properties proved in~\cite{SchrammSheffieldContinuumGFF, MillerSheffieldIG1, WangWuLevellinesGFFI}. 
Let $K = (K_t, t\ge 0)$ be an $\SLE_{4}(\underline{\rho}^L;\underline{\rho}^R)$ process with force points 
$(\underline{y}^L;\underline{y}^R)$, with $W,V^{i,q}$ solving the SDE system~\eqref{eq: SLE_kappa_rho}. 
Let $(g_t, t\ge 0)$ be the corresponding family of conformal maps and set $f_t := g_t - W_t$. 
Let $\gff_t^0$ be the harmonic function on $\HH$ with boundary data 
\begin{align*}
\begin{cases}
-\lambda(1+\sum_{i=0}^j\rho^{i,L}) , \qquad & \text{if } x \in \big(  f_t(y^{j+1,L}), f_t(y^{j,L})\big) , \\
+\lambda(1+\sum_{i=0}^j \rho^{i,R}) , \qquad & \text{if } x \in \big( f_t(y^{j,R}), f_t(y^{j+1, R}) \big),
\end{cases}
\end{align*}
where $\lambda=\pi/2$ and 
$\rho^{0,L}=\rho^{0,R}=0$, $\; y^{0,L}=0_-$, $\; y^{l+1, L}=-\infty$,$\; y^{0,R}=0_+$, and $\; y^{r+1, R}=\infty$ by convention. 
Define $\gff_t(z) := \gff_t^0(f_t(z))$.
By~\cite{DubedatSLEFreefield,SchrammSheffieldContinuumGFF, MillerSheffieldIG1}, 
there exists a coupling $(\gff,K)$ where $\gff = \tilde{\gff} + \gff_0$, with $\tilde{\gff}$ the zero-boundary $\GFF$ on $\HH$, 
such that the following is true.  Let$\tau$ be any $K$-stopping time before 
the continuation threshold. Then, the conditional law of $\gff$ restricted to $\HH\setminus K_{\tau}$ given 
$K_{\tau}$ is the same as the law of $\gff_{\tau} + \tilde{\gff}\circ f_{\tau}$. 
Furthermore, in this coupling, the process 
$K$ is almost surely determined by $\gff$. We refer to the $\SLE_{4}(\underline{\rho}^L;\underline{\rho}^R)$ 
in this coupling as the \textit{level line} of the field $\gff$. In particular, if the boundary value of $\gff$ 
is $-\lambda$ on $\R_-$ and $\lambda$ on $\R_+$, then the level line of $\gff$ starting from $0$ has the law 
of the chordal $\SLE_4$ from $0$ to $\infty$. 
In this case, we say that the field has \textit{Dobrushin boundary data}. 
In general, for $u\in\R$, the level line of $\gff$ with height $u$ is the level line of $h-u$.

Let $\gff$ be the $\GFF$ on $\HH$ with piecewise constant boundary data and let $\eta$ be 
the level line of $\gff$ starting from $0$. 
For $0<x<y$, assume that the boundary value of $\gff$ is a constant $c$ on $(x,y)$. 
Consider the intersection of $\eta$ with the interval 
$[x,y]$. The following facts were proved in~\cite[Section 2.5]{WangWuLevellinesGFFI}.
First, if $|c|\ge \lambda$, then $\eta\cap (x,y)=\emptyset$ almost surely; 
second, if $c\ge \lambda$, then $\eta$ can never hit the point $x$; 
third, if $c\le -\lambda$, then $\eta$ can never hit the point $y$, but it may hit the point $x$, 
and when it hits $x$, it meets its continuation threshold and cannot continue. In this case, 
we say that \textit{$\eta$ terminates at $x$}. 

\subsection{Pair of Level Lines}
\label{subsec::GFF_pair_of_level_lines}

Fix four points $x_1<x_2<x_3<x_4$ on the real line and let $\gff$ be the $\GFF$ on $\HH$ with 
the following boundary data (see also Figure~\ref{fig::twosle}):
\begin{align*}
-\lambda\text{ on }(-\infty, x_1), \qquad +\lambda \text{ on }(x_1, x_2), \qquad 
-\lambda\text{ on }(x_2,x_3), \qquad +\lambda \text{ on }(x_3,x_4), \qquad  -\lambda\text{ on }(x_4,\infty) .
\end{align*}
Let $\eta_1$ (resp.~$\eta_2$) be the level line of $\gff$ starting from $x_1$ (resp~$x_3$).
The two curves $\eta_1$ and $\eta_2$ cannot hit each other, and there are two cases for 
the possible endpoints of $\eta_1$ and $\eta_2$,
illustrated in Figure~\ref{fig::twosle}: Case~$\,\vcenter{\hbox{\includegraphics[scale=0.3]{figures/link-2.pdf}}}\,$, where
$\eta_1$ terminates at $x_4$ and $\eta_2$ terminates at $x_2$;
and Case~$\,\vcenter{\hbox{\includegraphics[scale=0.3]{figures/link-1.pdf}}}\,$, where
$\eta_1$ terminates at $x_2$ and $\eta_2$ terminates at $x_4$.
Both 
cases have a positive chance. As a warm-up, we calculate the probabilities for these 
two cases in Lemma~\ref{lem::twosle_proba}. Note that, given $\eta_1$, the curve $\eta_2$ is 
the level line of the $\GFF$ on $\HH\setminus \eta_1$ with Dobrushin boundary data. 
Therefore, in either case, 
the conditional law of $\eta_2$ given $\eta_1$ is the chordal $\SLE_4$ and, similarly, 
the conditional law of $\eta_1$ given $\eta_2$ is the chordal $\SLE_4$. 

\begin{figure}[h]
\begin{subfigure}[b]{0.48\textwidth}
\begin{center}
\includegraphics[width =0.95\textwidth]{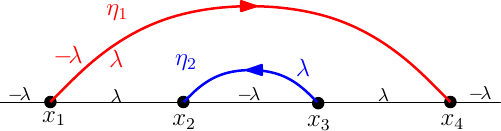}
\end{center}
\caption{Case $\vcenter{\hbox{\includegraphics[scale=0.3]{figures/link-2.pdf}}}=\{ \link{1}{4},\link{2}{3} \}$}
\end{subfigure}
\begin{subfigure}[b]{0.48\textwidth}
\begin{center}\includegraphics[width =0.95\textwidth]{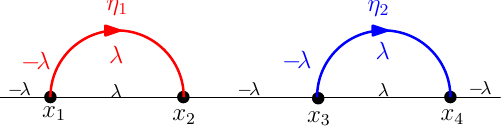}
\end{center}
\caption{Case $\vcenter{\hbox{\includegraphics[scale=0.3]{figures/link-1.pdf}}}=\{\link{1}{2},\link{3}{4}\}$.}
\end{subfigure}
\caption{Two level lines of the $\GFF$. The conditional law of $\eta_1$ given $\eta_2$ is the chordal 
$\SLE_4$ and the conditional law of $\eta_2$ given $\eta_1$ is the chordal $\SLE_4$.}
\label{fig::twosle} 
\end{figure}

\begin{remark}\label{rem::crossratio_trivialbound}
The following trivial fact will be important later: For $x_1<x_2<x_3<x_4$, we have
\begin{align*}\label{eqn::crossratio_trivialbound}
0\le \frac{(x_4-x_1)(x_3-x_2)}{(x_4-x_2)(x_3-x_1)}\le 1.
\end{align*}
\end{remark}

\begin{lemma}\label{lem::twosle_proba}
Set $\vcenter{\hbox{\includegraphics[scale=0.3]{figures/link-2.pdf}}}=\{ \link{1}{4}, \link{2}{3} \}$ 
and $\vcenter{\hbox{\includegraphics[scale=0.3]{figures/link-1.pdf}}}=\{\link{1}{2}, \link{3}{4} \}$. 
Let $P_{\vcenter{\hbox{\includegraphics[scale=0.2]{figures/link-2.pdf}}}}$ 
\textnormal{(}resp.~$P_{\vcenter{\hbox{\includegraphics[scale=0.2]{figures/link-1.pdf}}}}$\textnormal{)}
be the probability for Case $\vcenter{\hbox{\includegraphics[scale=0.3]{figures/link-2.pdf}}}$ 
\textnormal{(}resp.~Case $\vcenter{\hbox{\includegraphics[scale=0.3]{figures/link-1.pdf}}}$\textnormal{)},
as in Figure~\ref{fig::twosle}. 
Then we have
\begin{align*}
P_{\vcenter{\hbox{\includegraphics[scale=0.2]{figures/link-2.pdf}}}} 
= \frac{(x_4-x_3)(x_2-x_1)}{(x_4-x_2)(x_3-x_1)} 
\qquad \qquad \text{and} \qquad \qquad
P_{\vcenter{\hbox{\includegraphics[scale=0.2]{figures/link-1.pdf}}}} 
= 1 - P_{\vcenter{\hbox{\includegraphics[scale=0.2]{figures/link-2.pdf}}}} 
= \frac{(x_4-x_1)(x_3-x_2)}{(x_4-x_2)(x_3-x_1)} . 
\end{align*}
\end{lemma}
\begin{proof}
We know that 
$\eta:=\eta_1$ is an $\SLE_4(-2,+2,-2)$ process with force points $(x_2,x_3,x_4)$. If $T$ is 
the continuation threshold of $\eta$, then Case $\vcenter{\hbox{\includegraphics[scale=0.3]{figures/link-2.pdf}}}$ corresponds to $\{\eta(T)=x_4\}$ and 
Case $\vcenter{\hbox{\includegraphics[scale=0.3]{figures/link-1.pdf}}}$ to $\{\eta(T)=x_2\}$. 
Let $(W_t, 0\le t\le T)$ be the Loewner driving function of $\eta$ and $(g_t, 0\le t\le T)$ the corresponding conformal maps. Define, for $t<T$, 
\[ M_t := \frac{(g_t(x_4)-W_t)(g_t(x_3)-g_t(x_2))}{(g_t(x_4)-g_t(x_2))(g_t(x_3)-W_t)} . \]
Using It\^{o}'s formula, one can check that $M_t$ is a local martingale, 
and it is bounded by Remark~\ref{rem::crossratio_trivialbound}: we have
$0\le M_t \le 1$, for $t<T$.
Moreover, by Lemma~\ref{lem::continuouscurve_limit_technical} of Appendix~\ref{subsec:appendix-tech}, we have almost surely, as $t\to T$,
\begin{align*}
M_t \to 1, & \quad \text{when }\eta(t)\to x_2
\qquad  \text{and} \qquad M_t\to 0,\quad \text{when }\eta(t)\to x_4.
\end{align*}
Therefore, the optional stopping theorem implies that
\[P_{\vcenter{\hbox{\includegraphics[scale=0.2]{figures/link-1.pdf}}}}
=\PP[\eta(T)=x_2]=\E\left[M_T\right]=M_0
=\frac{(x_4-x_1)(x_3-x_2)}{(x_4-x_2)(x_3-x_1)}.\]
The formula for the probability $P_{\vcenter{\hbox{\includegraphics[scale=0.2]{figures/link-2.pdf}}}}$
then follows by a direct calculation.
\end{proof}

\subsection{Connection Probabilities for Level Lines}
\label{subsec:GFF_connection_proba}
Fix $N\ge 2$ and $x_1<\cdots<x_{2N}$. Let $\gff$ be the $\GFF$ on $\HH$ with \textit{alternating boundary data}:
\begin{align*}
\lambda \text{ on }(x_{2j-1}, x_{2j}), \text{ for } j \in \{ 1, \ldots, N \} ,
\qquad \text{and} \qquad 
-\lambda \text{ on }(x_{2j}, x_{2j+1}) , \text{ for }  j \in \{ 0, 1, \ldots, N \} ,
\end{align*}
with the convention that $x_0=-\infty$ and $x_{2N+1}=\infty$. For $j \in \{1,\ldots,N \}$, let $\eta_j$ be 
the level line of $\gff$ starting from $x_{2j-1}$. The possible terminal points of $\eta_j$ are the $x_n$'s 
with an even index $n$. The level lines $\eta_1, \ldots, \eta_N$ do not hit each other, so their endpoints 
form a (planar) link pattern $\LA\in \LP_N$. 
In Lemma~\ref{lem::levellines_P_alpha},  for each $\alpha\in\LP_N$, 
we derive the connection probability $P_{\alpha}:=\PP[\LA=\alpha]$.
To this end, we use the next lemmas,
which relate martingales for level lines with solutions of the system
$\mathrm{(PDE)}$~\eqref{eq: multiple SLE PDEs} with $\kappa = 4$.

\begin{lemma}\label{lem::localmart}
Let $\eta=\eta_1$ be the level line of $\gff$ starting from $x_1$, let $W_t$ be its driving function,
and $(g_t, t\ge 0)$ the corresponding family of conformal maps. 
Denote $X_{j1} := g_t(x_j)-W_t$ and $X_{ji} := g_t(x_j)-g_t(x_i)$, for $i,j \in \{2,\ldots,2N\}$.
For any subset $S\subset\{1,\ldots,2N\}$ containing $1$, define
\[M_t^{(S)} := \prod_{1\le i<j\le 2N}X_{ji}^{\delta(i,j)} , \qquad \text{where } \quad
\delta(i,j)=\begin{cases}0,&\text{if }i,j\in S, \text{ or }i,j\not\in S,\\
(-1)^{1+i+j}, &\text{if }i\in S \text{ and }j\not\in S, \text{or } i\not\in S \text{ and }j\in S.\end{cases}\]
Then, $M_t^{(S)}$ is a local martingale. 
\end{lemma}

We remark that the local martingale $M^{(S)}$ in Lemma~\ref{lem::localmart} is in fact the Radon-Nikodym derivative between 
the law of $\eta$ (i.e., the level line of the $\GFF$ with alternating boundary data), 
and the law of a level line of the $\GFF$ with a different boundary data
--- see the discussion in Section~\ref{subsec: Conformal blocks for GFF}.

\begin{proof}
The level line $\eta$ is an $\SLE_4(-2,+2,\ldots, -2)$ process with force points $(x_2, \ldots, x_{2N})$. 
We recall from~\eqref{eq: SLE_kappa_rho} that its driving function satisfies the SDE
\begin{align}\label{eqn::levelline_eta1_sde}
\ud W_t = 2 \ud B_t + \sum_{i=2}^{2N}\frac{-\rho_i \ud t}{g_t(x_i)-W_t}, 
\qquad \text{ where } \quad \rho_i=2(-1)^{i+1}
\end{align}
and $g_t$ is the Loewner map. We rewrite $M_t^{(S)}$ as follows:
\begin{align*}
M_t^{(S)} = \prod_{j=2}^{2N} X_{j1}^{\delta_j} \prod_{2\le i<j\le 2N} X_{ji}^{\delta(i,j)}, 
\qquad \text{where } \quad \delta_j=\delta(1,j).
\end{align*}
By It\^{o}'s formula, we have
\begin{align*}
\frac{\ud M_t^{(S)}}{M_t^{(S)}}
= \; & \sum_{j=2}^{2N} \frac{\delta_j}{X_{j1}} \left(\frac{2 \ud t}{X_{j1}} - \ud W_t\right) 
\, + \sum_{2\le i<j\le 2N} \frac{\delta(i,j)}{X_{ji}} \left(\frac{2 \ud t}{X_{j1}} 
- \frac{2 \ud t}{X_{i1}}\right) \\
\; & + \, \sum_{j=2}^{2N}\frac{2\delta_j(\delta_j-1) \ud t}{X_{j1}^2} 
\, + \sum_{2\le i<j\le 2N}\frac{4\delta_i\delta_j \ud t}{X_{j1}X_{i1}} \\
= \; & \sum_{j=2}^{2N}\frac{2\delta_j^2 \ud t}{X_{j1}^2} 
+ \, \sum_{j=2}^{2N} \sum_{i=2}^{2N} \frac{\delta_j\rho_i \ud t}{X_{j1}X_{i1}} \, + 
\sum_{2\le i<j\le 2N} \left(\frac{-2\delta(i,j)+4\delta_i\delta_j}{X_{j1}X_{i1}}\right) \ud t 
- \, \sum_{j=2}^{2N} \frac{2\delta_j \ud B_t}{X_{j1}} .
\end{align*}
For any $S$ containing $1$, the coefficient of the term $\ud t/X_{j1}^2$ for 
$j \in \{2,\ldots,2N\}$ is 
\[ 2\delta_j^2+\delta_j\rho_j=0 , \]
and the coefficient of the term $\ud t/(X_{j1}X_{i1})$, for 
$i,j \in \{2,\ldots,2N\}$, with $i<j$, is 
\[\delta_j\rho_i+\delta_i\rho_j-2\delta(i,j)+4\delta_i\delta_j=0.\]
Therefore, $M_t^{(S)}$ is a local martingale. 
\end{proof}

\begin{lemma} \label{lem::martingales_PDEs}
Let $\eta=\eta_1$ be the level line of $\gff$ starting from $x_1$, let $(W_t, t \ge 0)$ be its driving function, 
and $(g_t, t \ge 0)$ the corresponding family of 
conformal maps. For a smooth function $\CobloF \colon \chamber_{2N} \to \R$, the ratio 
\[M_t (\CobloF)
:= \frac{\CobloF (W_t, g_t(x_2), \ldots, g_t(x_{2N}) )}{\PartF^{(N)}_{\GFF} (W_t, g_t(x_2), \ldots, g_t(x_{2N}) )} \]
is a local martingale if and only if $\CobloF$ satisfies $\mathrm{(PDE)}$~\eqref{eq: multiple SLE PDEs} with $i = 1$ and $\kappa = 4$. 
\end{lemma}
\begin{proof}
Recall the SDE~\eqref{eqn::levelline_eta1_sde} for $W_t$. 
Lemma~\ref{lem: gff symmetric pff} gives an explicit formula for the function $\PartF:=\PartF^{(N)}_{\GFF}$. 
Using this, one verifies that $\PartF$ satisfies the following differential equation: for $\boldsymbol{x}=(x_1,\ldots, x_{2N}) \in \chamber_{2N}$,
\begin{align}\label{eqn::martingales_PDEs_aux1}
\bigg( 4 \partial_1 + \sum_{j=2}^{2N} \frac{\rho_j}{x_j-x_1} \bigg) 
\PartF(\boldsymbol{x}) = 0 .
\end{align}
Furthermore, $\PartF$ satisfies $\mathrm{(PDE)}$~\eqref{eq: multiple SLE PDEs} with $i = 1$ and $\kappa = 4$:
\begin{align} \label{eqn::martingales_PDEs_aux2}
\LD^{(1)}\PartF(\boldsymbol{x}) = 0,\qquad \qquad \text{where } \quad
\LD^{(1)} := 2\partial_1^2 + \sum_{j=2}^{2N} \left(\frac{2\partial_j}{x_j-x_1} - \frac{1}{2(x_j-x_1)^2}\right) .
\end{align}
We denote $\boldsymbol{Y} := (W_t, g_t(x_2), \ldots, g_t(x_{2N}))$, and 
$X_{j1} := g_t(x_j)-W_t$, and $X_{ji} := g_t(x_j)-g_t(x_i)$, for $i,j \in \{2,\ldots,2N\}$.
By It\^{o}'s formula, any (regular enough) function $F(x_1,\ldots,x_{2N})$ satisfies
\begin{align*}
\ud F(\boldsymbol{Y}) 
&= 2 \partial_1 F(\boldsymbol{Y}) \, \ud B_t + 
\bigg( 2 \partial_1^2 + 
\sum_{j=2}^{2N} \left(\frac{2 \partial_j}{X_{j1}}-\frac{\rho_j\partial_1}{X_{j1}}\right)
\bigg)  F(\boldsymbol{Y})\, \ud t\\
&=2 \partial_1 F(\boldsymbol{Y}) \, \ud B_t + 
\bigg( \LD^{(1)} + 
\sum_{j=2}^{2N} \left(\frac{1}{2X^2_{j1}}-\frac{\rho_j\partial_1}{X_{j1}}\right)
\bigg)  F(\boldsymbol{Y})\, \ud t.
\end{align*}
Combining with \eqref{eqn::martingales_PDEs_aux1} and \eqref{eqn::martingales_PDEs_aux2}, we see that 
\begin{align*}
\frac{\ud M_t(\CobloF)}{M_t(\CobloF)}
&=\frac{\ud \CobloF(\boldsymbol{Y})}{\CobloF(\boldsymbol{Y})}
-\frac{\ud \PartF(\boldsymbol{Y})}{\PartF(\boldsymbol{Y})}
+4\left(\frac{\partial_1\PartF(\boldsymbol{Y})}{\PartF(\boldsymbol{Y})}\right)^2\ud t
-4\left(\frac{\partial_1\CobloF(\boldsymbol{Y})}{\CobloF(\boldsymbol{Y})}\right)\left(\frac{\partial_1\PartF(\boldsymbol{Y})}{\PartF(\boldsymbol{Y})}\right)\ud t\\
&=\left(\frac{2\partial_1\CobloF(\boldsymbol{Y})}{\CobloF(\boldsymbol{Y})}
-\frac{2\partial_1\PartF(\boldsymbol{Y})}{\PartF(\boldsymbol{Y})}\right)\ud B_t+\frac{\LD^{(1)}\CobloF(\boldsymbol{Y})}{\CobloF(\boldsymbol{Y})}\ud t.
\end{align*}
This implies that $M_t(\CobloF)$ is a local martingale if and only if $\LD^{(1)}\CobloF=0$. 
\end{proof}

Now, we give the formula for the connection probabilities for the level lines of the $\GFF$.
To emphasize the main idea, we postpone a technical detail, Proposition~\ref{prop:: strong limit for ratio Zalpha and Ztotal}, 
to Appendix~\ref{subsec:appendix-tech}.

\begin{lemma}\label{lem::levellines_P_alpha}
We have
\begin{align}\label{eq::crossing_probabilities_for_kappa4_with_Z}
P_{\alpha} 
= \frac{\PartF_{\alpha}  (x_1, \ldots, x_{2N}) }{\PartF^{(N)}_{\GFF} (x_1, \ldots, x_{2N}) } > 0 , \quad \text{ for all } \alpha\in \LP_N,
\qquad\text{where } \quad \PartF^{(N)}_{\GFF} := \sum_{\alpha\in\LP_N}\PartF_{\alpha} ,
\end{align}
and $\{\PartF_{\alpha} \colon \alpha\in \LP\}$ is the collection of functions of 
Theorem~\ref{thm::purepartition_existence} with $\kappa=4$. 
\end{lemma}

\begin{proof}
By Theorem~\ref{thm::purepartition_existence}, we have $\PartF_{\alpha} > 0$ for all $\alpha\in\LP_N$, so for all 
$(x_1, \ldots, x_{2N}) \in \chamber_{2N}$, we have
\begin{align}\label{eqn::gff_bound_mart}
0 <
\frac{\PartF_{\alpha}(x_1, \ldots, x_{2N}) }{\PartF^{(N)}_{\GFF}(x_1, \ldots, x_{2N}) } \le 1 . 
\end{align}
We prove the assertion by induction on $N\ge0$. The initial case $N=0$ is a tautology: 
$\PartF_\emptyset = 1 = \PartF^{(0)}_{\GFF}$. 
Let then $N \ge 1$ and assume that formula~\eqref{eq::crossing_probabilities_for_kappa4_with_Z}
holds for all $\hat{\alpha} \in \LP_{N-1}$. Let $\alpha\in \LP_N$. 
Without loss of generality, we may assume that $\link{1}{2} \in \alpha$. Let $\eta$ be the level line of the $\GFF$ $\gff$ starting from 
$x_1$, let $T$ be its continuation threshold, $(W_t, t \ge 0)$ its driving function, and $(g_t,t\geq0)$ 
the corresponding family of conformal maps. Then by Lemma~\ref{lem::martingales_PDEs}, 
\[M_t(\PartF_{\alpha}) 
:= \frac{\PartF_{\alpha}(W_t, g_t(x_2), \ldots, g_t(x_{2N}))}{\PartF^{(N)}_{\GFF}(W_t, g_t(x_2), \ldots, g_t(x_{2N}))} \]
is a local martingale for $t<T$.

As $t\to T$, we know that $\eta(t)\to x_{2n}$ for some $n \in \{1, \ldots, N\}$. First, we consider the case when $\eta(t)\to x_2$.  
On the event $\{\eta(T)=x_2\}$, as $t\to T$, we have by Lemma~\ref{lem::partitionfunction_asy} almost surely 
\begin{align*} 
M_t(\PartF_{\alpha}) 
= \frac{(g_t(x_2)-W_t)^{1/2}}{(g_t(x_2)-W_t)^{1/2}}
\frac{\PartF_{\alpha}(W_t, g_t(x_2), \ldots, g_t(x_{2N}))}{\PartF^{(N)}_{\GFF}(W_t, g_t(x_2), \ldots, g_t(x_{2N}))} \; \longrightarrow \;
\frac{\PartF_{\hat{\alpha}}(g_T(x_3), \ldots, g_T(x_{2N}))}{\PartF^{(N-1)}_{\GFF}(g_T(x_3), \ldots, g_T(x_{2N}))} , 
\end{align*}
where $\hat{\alpha} = \alpha \removeLink \link{1}{2}$. 
Next, on the event $\{\eta(T)=x_{2n}\}$, we have almost surely 
\begin{align*}\label{eqn::gff_subtle_asy}
\lim_{t\to T}\frac{\PartF_{\alpha}(W_t, g_t(x_2), \ldots, g_t(x_{2N}))}{\PartF_{\GFF}^{(N)}(W_t, g_t(x_2), \ldots, g_t(x_{2n}))}=0 ,
\end{align*}
by the bound~\eqref{eqn::partf_alpha_upper} and Proposition~\ref{prop:: strong limit for ratio Zalpha and Ztotal}.
In summary, we have almost surely 
\begin{align*}
M_T(\PartF_{\alpha}):=\lim_{t\to T}M_t(\PartF_{\alpha})=
\one_{\{\eta(T)=x_2\}}
\; \frac{\PartF_{\hat{\alpha}}(g_T(x_3),\ldots, g_T(x_{2N}))}
{\PartF_{\GFF}^{(N-1)}(g_T(x_3),\ldots, g_T(x_{2N}))} .
\end{align*}
On the other hand, by~\eqref{eqn::gff_bound_mart}, $M_t(\PartF_{\alpha})$ is bounded, so the optional stopping theorem gives
\[\frac{\PartF_{\alpha}}{\PartF^{(N)}_{\GFF}} 
= M_0(\PartF_{\alpha}) = \E\left[M_T(\PartF_{\alpha})\right].\]
Combining this with the induction hypothesis $P_{\hat{\alpha}} = \PartF_{\hat{\alpha}}/\PartF^{(N-1)}_{\GFF}$, we obtain
\begin{align}\label{eqn::levellines_mart_bdd_1}
\frac{\PartF_{\alpha}}{\PartF^{(N)}_{\GFF}} 
= \E\left[\one_{\{\eta(T)=x_2\}}P_{\hat{\alpha}}(g_T(x_3),\ldots,g_T(x_{2N}))\right].
\end{align}

Finally, consider the level lines $(\eta_1, \ldots, \eta_N)$ 
of the $\GFF$ $\gff$, where, for each $j$, $\eta_j$ is the level line starting from $x_{2j-1}$. 
Given $\eta := \eta_1$, on the event $\{\eta(T)=x_2\}$, the conditional law of 
$(\eta_2, \ldots, \eta_N)$ is that of the level lines of the $\GFF$ $\hat{\gff}$ with alternating 
boundary data, where $\hat{\gff}$ is $\gff$ restricted to the unbounded component of $\HH\setminus \eta$. Thus, we have
\begin{align}\label{eqn::levellines_mart_bdd_2}
P_{\alpha}=\E\left[\one_{\{\eta(T)=x_2\}}P_{\hat{\alpha}}(g_T(x_3),\ldots,g_T(x_{2N}))\right].
\end{align}
Combining~\eqref{eqn::levellines_mart_bdd_1} and~\eqref{eqn::levellines_mart_bdd_2}, 
we obtain $P_{\alpha} = \PartF_{\alpha}/\PartF^{(N)}_{\GFF}$, 
which is what we sought to prove.
\end{proof}

\subsection{Marginal Probabilities and Proof of Theorem~\ref{thm::multiple_sle_4}}
\label{subsec::gff_levellines_Plk}
Next we calculate the probability for one level line of the $\GFF$ to terminate at a given point.
Again, we postpone a technical result to Appendix~\ref{subsec:appendix-tech}.

\begin{proposition}\label{prop::levellines_proba_lk}
For $a, b \in\{1,\ldots, 2N\}$ such that $a$ is odd and $b$ is even, the probability $P^{(a,b)}$ that the level line 
of the $\GFF$ starting from $x_a$ terminates at $x_{b}$ is given by 
\begin{align*} 
P^{(a,b)} (x_1, \ldots, x_{2N})
= \prod_{\substack{1\le j\le 2N, \\ j\neq a,b}} 
\bigg| \frac{x_j-x_a}{x_j-x_b} \bigg|^{(-1)^j} .
\end{align*}
\end{proposition}

Before proving the proposition, we observe that a special case follows by easy martingale arguments.
\begin{lemma}\label{lem::levellines_proba_neighbors}
The conclusion in Proposition \ref{prop::levellines_proba_lk} holds for $b=a+1$. 
\end{lemma}
\begin{proof} 
To simplify notation, we assume $a=1$; 
the other cases are similar. 
The level line $\eta:=\eta_1$ started from $x_1$ is an $\SLE_4(-2,+2,\ldots, -2)$ process with force points 
$(x_2, \ldots, x_{2N})$. Let $T$ be the continuation threshold of $\eta$. Define, for $t<T$,
\begin{align*}
M_t := \prod_{j = 3}^{2N}\left(\frac{g_t(x_j)-W_t}{g_t(x_j)-g_t(x_2)}\right)^{(-1)^j} .
\end{align*}
By Lemma~\ref{lem::localmart} with $S = \link{1}{2}$, $M_t$ is a local martingale.
Remark~\ref{rem::crossratio_trivialbound} gives, for all $j \in \{3, \ldots, 2N\}$, that
\begin{align*}
\left(\frac{g_t(x_{j+1})-W_t}{g_t(x_{j+1})-g_t(x_2)}\right)\left(\frac{g_t(x_j)-g_t(x_2)}{g_t(x_j)-W_t}\right)  \le 1,
\end{align*}
so $M_t$ is bounded:
we have $0\le M_t \le 1$ for $t<T$. Finally, as $t\to T$, we have almost surely $M_t \to 1$ when $\eta(t)\to x_2$, and
Lemma~\ref{lem::crossratio_zerowhenwrong_general} of Appendix~\ref{subsec:appendix-tech} shows that
$M_t\to 0$ when $\eta(t)\to x_{2n}$ for $n\in\{2,\ldots, N\}$.
Therefore, the optional stopping theorem implies $P^{(1,2)}=\PP[\eta(T)=x_2]=\E[M_T]=M_0$, as desired.
\end{proof}

To prove the general case in Proposition~\ref{prop::levellines_proba_lk}, we use the following lemma.

\begin{lemma}\label{lem::function_F_for_crossing_proba_of_one_level_line}
For any $N \geq 2$ and $a,b \in \{ 1,\ldots,2N \}$ with odd $a$ and even $b$, 
the function ${F^{(a,b)}_N \colon \chamber_{2N} \to \C}$,
\begin{align}\label{eqn::function_F_for_crossing_proba_of_one_level_line}
F^{(a,b)}_N (x_1, \ldots, x_{2N}) := \PartF_{\GFF}^{(N)} (x_1, \ldots, x_{2N})
\prod_{\substack{1\le j\le 2N, \\ j\neq a,b}} \bigg| \frac{x_j-x_a}{x_j-x_b} \bigg|^{(-1)^j} 
\end{align}
belongs to the solution space $\mathcal{S}_N$ defined in~\eqref{eq: solution space}.
\end{lemma}
\begin{proof}
The function $F^{(a,b)}_N$ clearly satisfies the bound~\eqref{eqn::powerlawbound}. Also, because 
$\PartF_{\GFF}^{(N)}$ satisfies $\mathrm{(COV)}$~\eqref{eq: multiple SLE Mobius covariance} and the product
$\prod_j \big| \frac{x_j-x_a}{x_j-x_{b}} \big|^{(-1)^j}$ 
is conformally invariant, $F^{(a,b)}_N$ also satisfies $\mathrm{(COV)}$~\eqref{eq: multiple SLE Mobius covariance}. 
It remains to show $\mathrm{(PDE)}$~\eqref{eq: multiple SLE PDEs}. Without loss of generality, we may assume $a=1$. 
Combining Lemmas~\ref{lem::martingales_PDEs} and~\ref{lem::localmart} (with $S = \link{1}{b}$),
we see that $F^{(1,b)}_N$ satisfes $\mathrm{(PDE)}$~\eqref{eq: multiple SLE PDEs} as well.
Thus, we indeed have $F^{(1,b)}_N \in \mathcal{S}_N$.
\end{proof}

\begin{proof}[Proof of Proposition~\ref{prop::levellines_proba_lk}]
On the one hand, because the function $F^{(a,b)}_N$ defined in~\eqref{eqn::function_F_for_crossing_proba_of_one_level_line} 
belongs to the space $\mathcal{S}_N$ by Lemma~\ref{lem::function_F_for_crossing_proba_of_one_level_line}, 
Proposition~\ref{prop: FK dual elements} 
allows us to write it in the form 
\begin{align*}
F^{(a,b)}_N = \sum_{\alpha \in \LP_N} c_\alpha \PartF_\alpha , \qquad \text{where } \quad c_\alpha = \FKdual_\alpha(F^{(a,b)}_N) .
\end{align*}
On the other hand, by the identity~\eqref{eq::crossing_probabilities_for_kappa4} in Theorem~\ref{thm::multiple_sle_4}, we have
\begin{align*}
P^{(a,b)} = \sum_{\alpha \in \LP_N \colon \link{a}{b}\in\alpha} P_\alpha 
= \sum_{\alpha \in \LP_N \colon \link{a}{b} \in \alpha} \frac{\PartF_\alpha}{\PartF_{\GFF}^{(N)}} .
\end{align*}
Thus, it suffices to show that
\begin{align}\label{eqn::coefficient_indicator}
\FKdual_\alpha(F^{(a,b)}_N) = \one\{\link{a}{b} \in \alpha\}.
\end{align} 
Without loss of generality, we  assume that $a = 1$. We prove~\eqref{eqn::coefficient_indicator} by induction on $N \geq 1$. 
It is clear for $N=1$.
Assume then that $N \geq 2$
and $\FKdual_\beta(F^{(1,b)}_{N-1}) = \one\{\link{1}{b} \in \beta\}$ for all $\beta \in \LP_{N-1}$ and $b \in \{ 2,4,\ldots,2N-2 \}$.
Let $\alpha \in \LP_N$ and choose $i$ such that $\link{i}{i+1} \in \alpha$. We consider two cases.

\begin{itemize}
\item $i,i+1 \notin \{1,b\}$:
By the property~\eqref{eq: total SLE asymptotics} of the function 
$\PartF_{\GFF}$, we have, for any $\xi \in (x_{i-1}, x_{i+2})$, 
\begin{align*}
\; & \lim_{x_{i},x_{i+1}\to\xi} (x_{i+1}-x_{i})^{1/2} \;  F^{(1,b)}_N(x_1,\ldots,x_{2N}) \\
= \; &  \lim_{x_{i},x_{i+1}\to\xi} (x_{i+1}-x_{i})^{1/2} 
\; \PartF_{\GFF}^{(N)} (x_1, \ldots, x_{2N})
\prod_{\substack{1\le j\le 2N, \\ j\neq 1,b} } \bigg| \frac{x_j-x_1}{x_j-x_{b}} \bigg|^{(-1)^j} \\
= \; & \PartF_{\GFF}^{(N-1)} (x_1, \ldots,x_{i-1},x_{i+2} , \ldots, x_{2N})
\prod_{\substack{1\le j\le 2N, \\ j\neq 1,b, i, i+1} } \bigg| \frac{x_j-x_1}{x_j-x_{b}} \bigg|^{(-1)^j} \\
= \; &  F^{(1,b')}_{N-1} (x_1, \ldots,x_{i-1},x_{i+2} , \ldots, x_{2N}) ,
\end{align*}
where $b' = b$ if $i > b$ and $b' = b-2$ if $i < b$. 
Thus, choosing an allowable ordering for the links in $\alpha$ in such a way that
$\link{a_1}{b_1} = \link{i}{i+1}$, the induction hypothesis shows that 
\begin{align*}
\FKdual_\alpha \big(F^{(1,b)}_N\big) = \FKdual_{\alpha \removeLink \link{i}{i+1} } \big( F^{(1,b')}_{N-1}\big) 
=  \one\{\link{1}{b'} \in \alpha \removeLink \link{i}{i+1}  \} =  \one\{\link{1}{b} \in \alpha \} .
\end{align*} 

\item $i\in \{1,b\}$ or $i+1\in\{1,b\}$:  
Then we necessarily have $\link{1}{b} \notin \alpha$.
By symmetry, it suffices to treat the case $i = 1$. Then we have, for any $\xi \in (x_1, x_2)$,
\begin{align*}
\; & \lim_{x_{1},x_{2}\to\xi} (x_{2}-x_{1})^{1/2} \; F^{(1,b)}_N(x_1,\ldots,x_{2N}) \\
= \; & \lim_{x_{1},x_{2}\to\xi} (x_{2}-x_{1})^{1/2} 
\; \PartF_{\GFF}^{(N)} (x_1, \ldots, x_{2N})
\prod_{\substack{1\le j\le 2N, \\ j\neq 1,2,b} } \bigg| \frac{x_j-x_1}{x_j-x_{b}} \bigg|^{(-1)^j} 
\times \bigg|\frac{x_2-x_1}{x_2-x_{b}} \bigg| \quad = \quad 0 .
\end{align*}
\end{itemize}
This proves~\eqref{eqn::coefficient_indicator} and finishes the proof of the lemma.
\end{proof}

Collecting the results from this section and Section~\ref{subsec:GFF_connection_proba}, we now prove Theorem~\ref{thm::multiple_sle_4}.
\multipleslefour*

\begin{proof} 
By Lemma~\ref{lem::levellines_P_alpha}, the connection probabilities 
$P_{\alpha}:=\PP[\LA=\alpha]$ are given by~\eqref{eq::crossing_probabilities_for_kappa4} and they are strictly positive.  
On the event $\{\LA=\alpha\}$, we have $(\eta_1,\ldots,\eta_N) \in X_0^{\alpha}(\HH; x_1,\ldots, x_{2N})$, whose law is
a global $N$-$\SLE_4$ associated to $\alpha$: for each $j \in \{1, \ldots, N\}$, the conditional law of $\eta_j$ given 
$\{\eta_1, \ldots, \eta_{j-1}, \eta_{j+1}, \ldots, \eta_N\}$ is that of
the level line of the $\GFF$ on $\hat{\Omega}_j$ with Dobrushin boundary data, which is that of the chordal $\SLE_4$.
By the uniqueness of the global $N$-$\SLE_4$~\cite[Theorem~1.2]{BeffaraPeltolaWuUniquenessGloableMultipleSLEs},
this global $N$-$\SLE_4$ is the global $N$-$\SLE_4$ constructed in Theorem~\ref{thm::global_existence}.
Finally, Proposition~\ref{prop::levellines_proba_lk} proves \eqref{eqn::levellines_proba_lk}.
\end{proof}

\section{Pure Partition Functions for Multiple ${\boldsymbol{\SLE_4}}$}
\label{sec::pure_pf_for_sle4}
In the previous section, we solved the connection probabilities for the level lines of 
the $\GFF$ in terms of the multiple $\SLE_4$ pure partition functions.
On the other hand, we constructed the multiple $\SLE_\kappa$ pure partition functions for all $\kappa \in (0,4]$ in 
Section~\ref{sec::characterization}, see~\eqref{eqn::purepartition_alpha_def}. 
The purpose of this section is to give another, algebraic formula for them in the case of $\kappa = 4$
(Theorem~\ref{thm: sle4 pure pffs}, also stated below).
This kind of algebraic formulas for connection probabilities were first derived by R.~Kenyon and 
D.~Wilson~\cite{Kenyon-Wilson:Boundary_partitions_in_trees_and_dimers,
Kenyon-Wilson:Double_dimer_pairings_and_skew_Young_diagrams}  
in the context of crossing probabilities in discrete models
(in a general setup, which includes the loop-erased random walk, $\kappa = 2$;
and the double-dimer model, $\kappa = 4$).
In~\cite{KKP:Correlations_in_planar_LERW_and_UST_and_combinatorics_of_conformal_blocks}, 
A.~Karrila, K.~Kyt\"ol\"a, and E.~Peltola also studied the scaling limits of connection probabilities of loop-erased random walks
and identified them with the multiple $\SLE_2$ pure partition functions.

The main virtue of the formula~\eqref{eq: Pure partition function for kappa 4} 
in Theorem~\ref{thm: sle4 pure pffs} is that for each $\alpha \in \LP_N$, it expresses 
the pure partition function $\PartF_\alpha$ as a finite sum of well-behaved functions 
$\CobloF_\beta$, for $\beta \in \LP_N$, with explicit integer coefficients that enumerate certain 
combinatorial objects only depending on $\alpha$ and $\beta$ (given in Proposition~\ref{prop: matrix inversion}). 
Such combinatorial enumerations have been studied, e.g., 
in~\cite{Kenyon-Wilson:Boundary_partitions_in_trees_and_dimers,
Kenyon-Wilson:Double_dimer_pairings_and_skew_Young_diagrams, 
KKP:Correlations_in_planar_LERW_and_UST_and_combinatorics_of_conformal_blocks} 
and they have many desirable properties which can be used in analyzing the pure partition functions.
As an example of this, we verify in Section~\ref{subsec: decay_properties} 
that the decay of the rainbow connection probability agrees with the boundary arm exponents (or 
(half-)watermelon exponents) appearing in the physics literature.

\smallbreak

The auxiliary functions $\CobloF_\alpha$ implicitly appear in the conformal field theory literature  
as ``conformal blocks''~\cite{BPZ:Infinite_conformal_symmetry_in_2D_QFT, 
FFK:Braid_matrices_and_structure_constants_for_minimal_conformal_models,
DMS:CFT, Ribault:CFT_in_the_plane, Flores-Peltola:Monodromy_invariant_correlation_function}\footnote{
The PDE system~\eqref{eq: multiple SLE PDEs} is related to certain quantities being martingales 
for $\SLE_4$ type curves (see Lemma~\ref{lem::martingales_PDEs}).
These partial differential equations arise in conformal field theory as well, from
degenerate representations of the Virasoro algebra, see, e.g.,~\cite{DMS:CFT, Ribault:CFT_in_the_plane}.
The connection of the $\SLE_\kappa$ with conformal field theory  (CFT) is now well-known 
\cite{Bauer-Bernard:Conformal_field_theories_of_SLEs, 
Friedrich-Werner:Conformal_restriction_highest_weight_representations_and_SLE,
Bauer-Bernard:Conformal_transformations_and_SLE_partition_function_martingale, 
Friedrich:On_connections_of_CFT_and_SLE, Friedrich-Kalkkinen:On_CFT_and_SLE, 
KytolaMultipleSLE, KytolaVirasoroSLE}: martingales for $\SLE_\kappa$ curves correspond with 
correlations in a CFT of central charge $c=(3\kappa-8)(6-\kappa)/2\kappa$. 
In that sense, it is natural that the conformal block functions satisfy 
$\mathrm{(PDE)}$~\eqref{eq: multiple SLE PDEs} --- they are (chiral) correlation functions
of a CFT with central charge $c = 1$.
Also the asymptotics property $\mathrm{(ASY)}$~\eqref{eq: multiple SLE asymptotics} 
for $\PartF_\alpha$ can be related to fusion in 
CFT~\cite{Cardy:Boundary_conditions_fusion_rules_and_Verlinde_formula, KytolaMultipleSLE, Dubedat:SLE_and_Virasoro_representations_fusion, KytolaPeoltolaPurePartitionSLE}.}.
In particular, such functions for irrational $\kappa$ are discussed in~\cite{KKP:Conformal_blocks},
where properties of them, such as asymptotics analogous to our findings in Lemma~\ref{lem: CobloF ASY} 
for $\kappa = 4$, are explained in terms of conformal field theory.
In Section~\ref{subsec: Conformal blocks for GFF}, we give a relation between 
the conformal blocks with $\kappa = 4$ and level lines of the $\GFF$.

For each link pattern $\alpha \in \LP_N$, we define
the conformal block function 
\begin{align*}
\CobloF_\alpha \colon \chamber_{2N} \to \Rpos
\end{align*}
as follows.
We write $\alpha$ as an ordered collection~\eqref{eq: begin and end point convention sec 2}.
Then, we set $\vartheta_{\alpha}(i,i) := 0$ and 
\begin{align} \label{eq: CobloF definition}
\CobloF_\alpha(x_1, \ldots, x_{2N}) 
:= & \; \prod_{1\le i<j\le 2N} (x_j - x_i)^{\frac{1}{2} \vartheta_{\alpha}(i,j)}, \\
\text{where } \quad \vartheta_{\alpha}(i,j) := \; & 
\begin{cases}
+1 , & \text{if } i,j \in \{a_1, a_2, \ldots, a_N\} ,
\text{ or } i,j \in \{b_1, b_2, \ldots, b_N\} ,\\
-1 , & \text{otherwise.}
\end{cases}
\nonumber
\end{align}
When 
$\alpha = \unnested_N := \{ \link{1}{2}, \link{3}{4}, \ldots, \link{2N-1}{2N} \}$
is the completely unnested link pattern, 
the formula~\eqref{eq: CobloF definition} equals that of the symmetric partition function
from Lemma~\ref{lem: gff symmetric pff}, so 
$\CobloF_{\unnested_N} = \PartF_{\GFF}^{(N)}$.
Also, it follows from Proposition~\ref{prop: matrix inversion} that 
when $\alpha = \nested_N := \{ \link{1}{2N}, \link{2}{2N-1}, \ldots, \link{N-1}{N} \}$ 
is the rainbow link pattern, then we have 
$\CobloF_{\nested_N} = \PartF_{\nested_N}$. In general, the conformal block functions $\CobloF_\alpha$ 
and the pure partition functions $\PartF_\alpha$, for $\alpha \in \LP_N$, form two linearly independent 
sets that are related by a non-trivial change of basis. 
Theorem~\ref{thm: sle4 pure pffs} expresses this change of basis in terms of 
matrix elements denoted by $\Minv_{\alpha,\beta}$.
For illustration, we list some examples of the explicit formulas for $\PartF_\alpha$ from Theorem~\ref{thm: sle4 pure pffs}:
\begin{itemize}[align=left]
\item[$N = 1$:] This case is trivial,
$\PartF_{\vcenter{\hbox{\includegraphics[scale=0.2]{figures/link-0.pdf}}}}(x_1,x_2) 
= (x_2 - x_1)^{-1/2} 
= \CobloF_{\vcenter{\hbox{\includegraphics[scale=0.2]{figures/link-0.pdf}}}}(x_1,x_2)$.

\item[$N = 2$:] There are two link patterns, and denoting $x_{ji} := x_j - x_i$, we have
(see also Table~\ref{tab::Matrix_ex1} in Section~\ref{subsec: Combinatorics})
\begin{align*}
\PartF_{\vcenter{\hbox{\includegraphics[scale=0.2]{figures/link-2.pdf}}}} (x_1,x_2,x_3,x_4)
 \; = \; \; & 
 \CobloF_{\vcenter{\hbox{\includegraphics[scale=0.2]{figures/link-2.pdf}}}} (x_1,x_2,x_3,x_4)
 \; = \; \left( \frac{x_{43} x_{21}}{x_{41} x_{31} x_{42} x_{32}} \right)^{1/2} , \\
\PartF_{\vcenter{\hbox{\includegraphics[scale=0.2]{figures/link-1.pdf}}}} (x_1,x_2,x_3,x_4)
 \; = \; \; & 
 \CobloF_{\vcenter{\hbox{\includegraphics[scale=0.2]{figures/link-1.pdf}}}} (x_1,x_2,x_3,x_4)
 \; - \; \CobloF_{\vcenter{\hbox{\includegraphics[scale=0.2]{figures/link-2.pdf}}}} (x_1,x_2,x_3,x_4)
  \; = \; \left( \frac{x_{41} x_{32}}{x_{31} x_{21} x_{43} x_{42}} \right)^{1/2} .
\end{align*}
Note that these formulas are consistent with Lemma~\ref{lem::twosle_proba}.

\item[$N = 3$:] There are five link patterns, and we have
(see also Table~\ref{tab::Matrix_ex2} in Section~\ref{subsec: Combinatorics})
\begin{align*}
\PartF_{\vcenter{\hbox{\includegraphics[scale=0.2]{figures/link-7.pdf}}}} 
 \; = \; \; & \CobloF_{\vcenter{\hbox{\includegraphics[scale=0.2]{figures/link-7.pdf}}}} , \\
\PartF_{\vcenter{\hbox{\includegraphics[scale=0.2]{figures/link-6.pdf}}}} 
 \; = \; \; & \CobloF_{\vcenter{\hbox{\includegraphics[scale=0.2]{figures/link-6.pdf}}}} 
 \; - \; \CobloF_{\vcenter{\hbox{\includegraphics[scale=0.2]{figures/link-7.pdf}}}} , \\
\PartF_{\vcenter{\hbox{\includegraphics[scale=0.2]{figures/link-5.pdf}}}} 
 \; =  \; \; & \CobloF_{\vcenter{\hbox{\includegraphics[scale=0.2]{figures/link-5.pdf}}}} 
 \; - \; \CobloF_{\vcenter{\hbox{\includegraphics[scale=0.2]{figures/link-6.pdf}}}} 
 \; + \; \CobloF_{\vcenter{\hbox{\includegraphics[scale=0.2]{figures/link-7.pdf}}}} , \\
\PartF_{\vcenter{\hbox{\includegraphics[scale=0.2]{figures/link-4.pdf}}}} 
 \; = \; \; & \CobloF_{\vcenter{\hbox{\includegraphics[scale=0.2]{figures/link-4.pdf}}}} 
 \; - \; \CobloF_{\vcenter{\hbox{\includegraphics[scale=0.2]{figures/link-6.pdf}}}} 
 \; + \; \CobloF_{\vcenter{\hbox{\includegraphics[scale=0.2]{figures/link-7.pdf}}}} , \\
\PartF_{\vcenter{\hbox{\includegraphics[scale=0.2]{figures/link-3.pdf}}}} 
 \; = \; \; & \CobloF_{\vcenter{\hbox{\includegraphics[scale=0.2]{figures/link-3.pdf}}}} 
 \; - \; \CobloF_{\vcenter{\hbox{\includegraphics[scale=0.2]{figures/link-4.pdf}}}} 
 \; - \; \CobloF_{\vcenter{\hbox{\includegraphics[scale=0.2]{figures/link-5.pdf}}}} 
 \; + \; \CobloF_{\vcenter{\hbox{\includegraphics[scale=0.2]{figures/link-6.pdf}}}} 
 \; - \; 2 \, \CobloF_{\vcenter{\hbox{\includegraphics[scale=0.2]{figures/link-7.pdf}}}} .
\end{align*} 
\end{itemize}

We give now the statement and an outline of the proof for Theorem~\ref{thm: sle4 pure pffs}:

\slepurepffs*

\begin{remark}
A direct consequence of Theorems~\ref{thm::purepartition_existence} and~\ref{thm: sle4 pure pffs} is that the 
right-hand side of~\eqref{eq: Pure partition function for kappa 4} satisfies the power law 
bound~\eqref{eqn::partitionfunction_positive} with $\kappa=4$. 
This is far from clear from the right-hand side of~\eqref{eq: Pure partition function for kappa 4} itself.
\end{remark}

The proof of Theorem~\ref{thm: sle4 pure pffs} uses two inputs. 
First, we verify that the right-hand side of the asserted formula~\eqref{eq: Pure partition function for kappa 4} 
satisfies all of the properties of 
the pure partition functions $\PartF_\alpha$ with $\kappa = 4$:
the normalization $\PartF_\emptyset = 1$, bound~\eqref{eqn::powerlawbound}, 
system $\mathrm{(PDE)}$~\eqref{eq: multiple SLE PDEs}, 
covariance $\mathrm{(COV)}$~\eqref{eq: multiple SLE Mobius covariance}, 
and asymptotics $\mathrm{(ASY)}$~\eqref{eq: multiple SLE asymptotics}.
Second, with these properties verified, we invoke the uniqueness Corollary~\ref{cor::purepartition_unique} to conclude that 
the right-hand side of~\eqref{eq: Pure partition function for kappa 4} must be equal to $\PartF_\alpha$.
We give the complete proof in the end of Section~\ref{subsec: Conformal blocks}.

We point out that 
it is not difficult to check the properties $\mathrm{(PDE)}$~\eqref{eq: multiple SLE PDEs} and
$\mathrm{(COV)}$~\eqref{eq: multiple SLE Mobius covariance} --- this is a direct calculation. 
The main step of the proof is establishing the asymptotics $\mathrm{(ASY)}$~\eqref{eq: multiple SLE asymptotics},
which we perform by combinatorial calculations in Section~\ref{subsec: Conformal blocks}, using notations and results from Section~\ref{subsec: Combinatorics}.
However, before finishing the proof of Theorem~\ref{thm: sle4 pure pffs}, we discuss an application 
to (half-)watermelon exponents. 


\subsection{Decay Properties of Pure Partition Functions with $\kappa = 4$}
\label{subsec: decay_properties}
Consider the rainbow link pattern $\nested_N$ (see Figure~\ref{fig::rainbow}). 
We prove now that, when its first $N$ variables (or both the first $N$ and the last $N$ variables) tend together,
the decay of the pure partition function $\PartF_{\nested_N}$
agrees with the predictions from the physics literature for certain surface critical 
exponents~\cite{Cardy:Conformal_invariance_and_surface_critical_behavior,
Duplantier-Saleur:Exact_critical_properties_of_two-dimensional_dense_self-avoiding_walks,
Nienhuis:Coulomb_gas_formulation_of_2D_phase_transitions,WernerGirsanov,WuAlternatingArmIsing},
known as boundary arm exponents (or (half-)watermelon exponents).

\begin{proposition} \label{prop: decay of rainbow and total pf}
The rainbow pure partition function has the following decay as its $N$ first variables tend together:
\begin{align*}
\PartF_{\nested_N}(\eps, 2\eps, \ldots, N \eps, 1, 2, \ldots, N ) 
\quad \overset{\eps \to 0}{\sim} \quad \eps^{N(N-1)/4} .
\end{align*}
The symmetric partition function 
$\PartF_{\GFF}$ has the decay
\begin{align*}
\PartF^{(N)}_{\GFF}(\eps, 2\eps, \ldots, N \eps, 1, 2, \ldots, N ) 
\quad \overset{\eps \to 0}{\sim} \quad
\begin{cases}
\eps^{-N/4} , & \text{if $N$ is even,} \\
\eps^{-(N-1)/4} , & \text{if $N$ is odd.} 
\end{cases}
\end{align*}
\end{proposition}
\begin{proof}
Theorem~\ref{thm: sle4 pure pffs} and Proposition~\ref{prop: matrix inversion} show that $\PartF_{\nested_N} = \CobloF_{\nested_N}$.
Now, it is clear from the definition~\eqref{eq: CobloF definition} of $\CobloF_{\nested_N}$ as a product 
that the decay from the first $N$ variables $x_j = j \eps$, for $j \in \{1, \ldots, N\}$, is
$\eps^p$, where the power can be read off from~\eqref{eq: CobloF definition}:
$p = \frac{N(N-1)}{2} \times \frac{1}{2} (+1) = \frac{N(N-1)}{4}$. 
Similarly, by the formula~\eqref{eq: gff symmetric pff} of Lemma~\ref{lem: gff symmetric pff}, 
the decay of the symmetric partition function 
$\PartF_{\GFF}$ is also of type $\eps^{p'}$.
To find out the power, we collect the exponents from the differences of the variables 
$x_j = j \eps$ in~\eqref{eq: gff symmetric pff}, for $j \in \{1, \ldots, N\}$:  
\begin{align*}
p' = \sum_{1\leq k<l\leq N} \frac{1}{2} (-1)^{l-k} 
= \frac{1}{2} \sum_{k = 1}^{N-1} \sum_{m = 1}^{N-k} (-1)^{m} = 
\begin{cases}
- N/4 , & \text{if $N$ is even,} \\
- (N-1)/4 , & \text{if $N$ is odd.}
\end{cases}
\end{align*}
This proves the asserted decay.
\end{proof}

We see from Theorem~\ref{thm: sle4 pure pffs} and Proposition~\ref{prop: decay of rainbow and total pf} that for the level lines of the $\GFF$, 
the connection probability associated to the rainbow link pattern
(given in Theorem~\ref{thm::multiple_sle_4}) has the decay 
\begin{align*}
P_{\nested_N}=\frac{\PartF_{\nested_N}(\eps, 2\eps, \ldots, N \eps, 1, 2, \ldots, N )}{\PartF^{(N)}_{\GFF}(\eps, 2\eps, \ldots, N \eps, 1, 2, \ldots, N )}
\quad \overset{\eps \to 0}{\sim} \quad
\eps^{\alpha_N^+} , 
\qquad \text{where } \quad \alpha_N^+=
\begin{cases}
N^2/4 , & \text{if $N$ is even,} \\
(N^2-1)/4 , & \text{if $N$ is odd.}
\end{cases}
\end{align*}
The exponent $\alpha_N^+$ agrees with the $\SLE_4$ boundary arm exponents 
derived in~\cite[Proposition 3.1]{WuAlternatingArmIsing}.

\begin{corollary}\label{cor::watermelon}
The rainbow pure partition function has the following decay as both its $N$ first variables 
and its $N$ last variables tend together:
\begin{align*}
\PartF_{\nested_N}(\eps, 2\eps, \ldots, N \eps, 1 + \eps, 1 + 2 \eps, \ldots, 1 + N \eps ) 
\quad \overset{\eps \to 0}{\sim} \quad
\eps^{N(N-1)/2} . 
\end{align*}
The symmetric partition function 
$\PartF_{\GFF}$ has the decay
\begin{align*}
\PartF^{(N)}_{\GFF}(\eps, 2\eps, \ldots, N \eps, 1 + \eps, 1 + 2 \eps, \ldots, 1 + N \eps ) 
\quad \overset{\eps \to 0}{\sim} \quad
\begin{cases}
\eps^{-N/2} , & \text{if $N$ is even,} \\
\eps^{-(N-1)/2} , & \text{if $N$ is odd.}
\end{cases}
\end{align*}
\end{corollary}
\begin{proof}
Because the two sets $\{ \eps, 2\eps, \ldots, N \eps \}$ and $\{ 1 + \eps, 1 + 2 \eps, \ldots, 1 + N \eps \}$
of variables tend to $0$ and $1$, respectively,
we only have to add up the power-law decay of Proposition~\ref{prop: decay of rainbow and total pf} for both.
\end{proof}

\subsection{First Properties of Conformal Blocks}
\label{subsec: PDE_COV}
Now we verify properties  $\mathrm{(PDE)}$~\eqref{eq: multiple SLE PDEs} and
$\mathrm{(COV)}$~\eqref{eq: multiple SLE Mobius covariance} with $\kappa = 4$ for $\CobloF_\alpha$.

\begin{lemma} \label{lem: CobloF PDE}
The functions $\CobloF_\alpha \colon \chamber_{2N} \to \Rpos$ 
defined in~\eqref{eq: CobloF definition} satisfy
$\mathrm{(PDE)}$~\eqref{eq: multiple SLE PDEs} with $\kappa = 4$.
\end{lemma}
\begin{proof}
For notational simplicity, we write $x_{ij} = x_i - x_j$ and
$\boldsymbol{x} = (x_1,\ldots,x_{2N}) \in \chamber_{2N}$.  
We need to show that for fixed $i \in \{1, \ldots, 2N\}$, we have
\begin{align} \label{eq: multiple SLE PDEs with kappa = 4 fixed i}
2 \frac{\partial_i^2 \CobloF_\alpha(\boldsymbol{x})}{\CobloF_\alpha(\boldsymbol{x})}
+ \sum_{j \neq i} \left(
\frac{2}{x_{ji}} \frac{\partial_j \CobloF_\alpha(\boldsymbol{x})}{\CobloF_\alpha(\boldsymbol{x})}
- \frac{1}{2x_{ji}^{2}}\right) = 0 .
\end{align}
The terms with derivatives are
\begin{align*}
2 \frac{\partial_i^2 \CobloF_\alpha(\boldsymbol{x})}{\CobloF_\alpha(\boldsymbol{x})}
\, = \; \, & \frac{1}{2} \sum_{j, k \neq i} \frac{\vartheta_{\alpha}(i,j)\vartheta_{\alpha}(i,k)}{x_{ij}x_{ik}}
\, - \,  \sum_{j \neq i} \frac{\vartheta_{\alpha}(i,j)}{x_{ij}^2}  \qquad  \text{and} \qquad 
\sum_{j \neq i} \frac{2}{x_{ji}} \frac{\partial_j \CobloF_\alpha(\boldsymbol{x})}{\CobloF_\alpha(\boldsymbol{x})}
\, = \,  \sum_{j \neq i} \frac{1}{x_{ji}} \; \sum_{k \neq j} \frac{\vartheta_{\alpha}(j,k)}{x_{jk}} .
\end{align*}
Using this, the left-hand side of the PDE~\eqref{eq: multiple SLE PDEs with kappa = 4 fixed i}
becomes
\begin{align} \label{eq: PDE check}
\frac{1}{2} \sum_{j, k \neq i} \frac{\vartheta_{\alpha}(i,j)\vartheta_{\alpha}(i,k)}{x_{ij}x_{ik}}
\, - \, \sum_{j \neq i} \frac{\vartheta_{\alpha}(i,j)}{x_{ij}^2}
\, + \, \sum_{j \neq i} \frac{1}{x_{ji}} \; \sum_{k \neq j} \frac{\vartheta_{\alpha}(j,k)}{x_{jk}}
\, - \,  \sum_{j \neq i} \frac{1}{2 x_{ji}^2} .
\end{align}
The last term of~\eqref{eq: PDE check} is canceled by the case $k = j$ in the first term,
and the second term of~\eqref{eq: PDE check} is canceled by the case $k = i$ in the third term.
We are left with 
\begin{align} \label{eq: PDE check 2}
\sum_{j \neq i} \frac{1}{x_{ji}} \; \sum_{k \neq i, j} 
\left( \frac{\vartheta_{\alpha}(j,k)}{x_{jk}} - \frac{\vartheta_{\alpha}(i,j)\vartheta_{\alpha}(i,k)}{2 x_{ik}} \right) .
\end{align}
For a pair $(j,k)$ such that $j \neq i$ and $k \neq i,j$, combining the terms where $j$ and $k$ 
are interchanged, we get a term of the form
\begin{align*}
\left( \frac{\vartheta_{\alpha}(j,k)}{x_{ji}x_{jk}} + \frac{\vartheta_{\alpha}(j,k)}{x_{ki}x_{kj}} \right)
+ \frac{\vartheta_{\alpha}(i,j)\vartheta_{\alpha}(i,k)}{x_{ji}x_{ki}}
\; = \; \frac{\vartheta_{\alpha}(j,k)}{x_{ji}x_{ik}} + \frac{\vartheta_{\alpha}(i,j)\vartheta_{\alpha}(i,k)}{x_{ji}x_{ki}}
\; = \; \frac{\vartheta_{\alpha}(j,k) - \vartheta_{\alpha}(i,j)\vartheta_{\alpha}(i,k)}{x_{ji}x_{ik}} .
\end{align*}
It remains to notice that the numbers $\vartheta_{\alpha}$ defined in~\eqref{eq: CobloF definition}
satisfy the identity $\vartheta_{\alpha}(j,k) - \vartheta_{\alpha}(i,j)\vartheta_{\alpha}(i,k) = 0$ for all $i,j$, and $k$.
This proves~\eqref{eq: multiple SLE PDEs with kappa = 4 fixed i} and finishes the proof.
\end{proof}

\begin{lemma} \label{lem: CobloF COV}
The functions $\CobloF_\alpha \colon \chamber_{2N} \to \Rpos$ 
defined in~\eqref{eq: CobloF definition} satisfy 
$\mathrm{(COV)}$~\eqref{eq: multiple SLE Mobius covariance} with $\kappa = 4$.
\end{lemma}
\begin{proof}
For any conformal map $\varphi \colon \HH \to \HH$, we have the identity
$\frac{\varphi(z) - \varphi(w)}{z-w} = \sqrt{\varphi'(z)} \sqrt{\varphi'(w)}$
for all $z,w \in \overline{\HH}$, 
see e.g.~\cite[Lemma~4.7]{KytolaPeoltolaPurePartitionSLE}.
Using this identity, we calculate
\begin{align*}
\frac{\CobloF_\alpha(\varphi(x_1), \ldots,\varphi(x_{2N}))}{\CobloF_\alpha(x_1, \ldots,x_{2N})}
= \prod_{1\le i<j\le 2N} \left(\frac{\varphi(x_j) - \varphi(x_i)}{x_j-x_i} \right)^{\frac{1}{2} \vartheta_{\alpha}(i,j)}
= \prod_{1\le i<j\le 2N} \big( \varphi'(x_j) \varphi'(x_i) \big)^{\frac{1}{4} \vartheta_{\alpha}(i,j)} . 
\end{align*}
For each $j \in \{1, \ldots, 2N\}$, the factor $\varphi'(x_j)$ comes with the total power
$\frac{1}{4}(N(-1) + (N-1)(+1)) = -\frac{1}{4}$, which equals
$-h = (\kappa-6) / 2\kappa$ with $\kappa = 4$. Thus, $\CobloF_\alpha$ 
satisfy $\mathrm{(COV)}$~\eqref{eq: multiple SLE Mobius covariance} with $\kappa = 4$.
\end{proof}

\subsection{Asymptotics of Conformal Blocks and Proof of Theorem~\ref{thm: sle4 pure pffs}}
\label{subsec: Conformal blocks}
To finish the proof of Theorem~\ref{thm: sle4 pure pffs}, we need to calculate the asymptotics
of the conformal block functions $\CobloF_\alpha$. 
This proof 
is combinatorial, relying on results from~\cite{Kenyon-Wilson:Boundary_partitions_in_trees_and_dimers,
Kenyon-Wilson:Double_dimer_pairings_and_skew_Young_diagrams,
KKP:Correlations_in_planar_LERW_and_UST_and_combinatorics_of_conformal_blocks} 
discussed in Section~\ref{subsec: Combinatorics}.
Recall that we identify link patterns $\alpha \in \LP_N$
with the corresponding Dyck paths~\eqref{eq::DP}. 

\begin{lemma} \label{lem: CobloF ASY}
The collection $\{\CobloF_\alpha \colon \alpha \in \DP\}$
of functions defined in~\eqref{eq: CobloF definition} satisfy the asymptotics property
\begin{align}
\label{eq: CobloF asymptotics}
\lim_{x_j , x_{j+1} \to \xi} 
 \frac{ \CobloF_\alpha(x_1 , \ldots , x_{2N})}{(x_{j+1} - x_j)^{-1/2}}
= \; & \begin{cases}
0 , & \text{if } \slopeat{j} \in \alpha , \\
\CobloF_{ \alpha \removeupwedge{j} } (x_1 , \ldots, x_{j-1}, x_{j+2}, \ldots, x_{2N}) , & \text{if } \upwedgeat{j} \in \alpha , \\
\CobloF_{ \alpha \removedownwedge{j} } (x_1 , \ldots, x_{j-1}, x_{j+2}, \ldots, x_{2N}) , & \text{if } \downwedgeat{j} \in \alpha ,
\end{cases}
\end{align}
for any $j \in \{1, \ldots, 2N-1\}$ and $\xi \in (x_{j-1}, x_{j+2})$.
\end{lemma}
\begin{proof}
Fix $j \in \{1, \ldots, 2N-1\}$. If $\slopeat{j} \in \alpha$, then either both $j$ and $j+1$
are $a$-type indices with labels $a_r,a_s$, or both are $b$-type indices with labels $b_r,b_s$.
In either case, we have $\vartheta_{\alpha}(j,j+1)=1$, so the limit in~\eqref{eq: CobloF asymptotics} is zero.
Assume then that $\upwedgeat{j} \in \alpha$ (resp. $\downwedgeat{j} \in \alpha$).
In this case, we have $j = b_s$ and $j+1 = a_r$ (resp. $j = a_r$ and $j+1 = b_s$)
for some $r,s \in \{1, \ldots, N\}$, so $\vartheta_{\alpha}(j,j+1)=-1$. By definition~\eqref{eq: CobloF definition},
we have
\begin{align*}
\CobloF_\alpha(x_1, \ldots, x_{2N}) 
= & \; \prod_{1\le k<l\le 2N} (x_l - x_k)^{\frac{1}{2} \vartheta_{\alpha}(k,l)} \\
= & \; (x_{j+1} - x_j)^{-1/2}
\prod_{\substack{k < l , \\ k,l \neq j,j+1}} (x_l - x_k)^{\frac{1}{2} \vartheta_{\alpha}(k,l)} \\
& \; \times \prod_{k < j} (x_{j+1} - x_k)^{\frac{1}{2} \vartheta_{\alpha}(j+1,l)} (x_j - x_k)^{\frac{1}{2} \vartheta_{\alpha}(j,l)}
\prod_{l > j+1} (x_l - x_{j+1})^{\frac{1}{2} \vartheta_{\alpha}(k,j+1)} (x_l - x_j)^{\frac{1}{2} \vartheta_{\alpha}(k,j)} .
\end{align*}
The first factor cancels with the normalization factor $(x_{j+1} - x_j)^{1/2}$ in the 
limit~\eqref{eq: CobloF asymptotics}. The second product is independent of $x_j,x_{j+1}$ and
tends to $\CobloF_{ \alpha \removeupwedge{j} } (x_1 , \ldots, x_{j-1}, x_{j+2}, \ldots, x_{2N})$ in the 
limit~\eqref{eq: CobloF asymptotics} (resp.~to $\CobloF_{ \alpha \removedownwedge{j} }$).
Finally, the products in the last line tend to one in the limit~\eqref{eq: CobloF asymptotics},
because we have $\vartheta_{\alpha}(k,j+1) = -\vartheta_{\alpha}(k,j)$, for all $k < j$, and
$\vartheta_{\alpha}(j+1,l) = -\vartheta_{\alpha}(j,l)$, for all $l > j+1$. This proves the lemma.
\end{proof}

\begin{lemma} \label{lem: Pure partition function ASY}
The functions defined by the right-hand side of~\eqref{eq: Pure partition function for kappa 4} satisfy
$\mathrm{(ASY)}$~\eqref{eq: multiple SLE asymptotics} with $\kappa = 4$.
\end{lemma}
\begin{proof}
Denote the functions in question by
\begin{align}\label{eq: Second formula for pure partition function for kappa 4}
\tilde{\PartF}_\alpha(x_1, \ldots, x_{2N}) 
:= \; & \sum_{\beta \DPgeq \alpha} \Minv_{\alpha,\beta} \, \CobloF_\beta(x_1, \ldots, x_{2N}) .
\end{align}
Fix $j \in \{1, \ldots, 2N-1\}$. 
For the asymptotics property $\mathrm{(ASY)}$~\eqref{eq: multiple SLE asymptotics},
we have two cases to consider: either $\link{j}{j+1} \in \alpha$ or $\link{j}{j+1} \notin \alpha$.
As explained in Section~\ref{subsec: Combinatorics}, these can be equivalently written 
in terms of the Dyck path $\alpha \in \DP_N$
as $\upwedgeat{j} \in \alpha$ and $\upwedgeat{j} \notin \alpha$.
The asserted property $\mathrm{(ASY)}$~\eqref{eq: multiple SLE asymptotics} with $\kappa = 4$ 
can thus be written in the following form:
for all $\alpha \in \LP_N$, and for all $j \in \{1, \ldots, 2N-1 \}$ and $\xi \in (x_{j-1}, x_{j+2})$, we have
\begin{align}\label{eq: multiple SLE asymptotics with Dyck paths}
\lim_{x_j , x_{j+1} \to \xi} 
\frac{\tilde{\PartF}_\alpha(x_1 , \ldots , x_{2N})}{(x_{j+1} - x_j)^{-1/2}} 
=\begin{cases}
0 , \quad &
    \text{if } \upwedgeat{j} \notin \alpha , \\
\tilde{\PartF}_{\alpha \removeupwedge{j}}(x_{1},\ldots,x_{j-1},x_{j+2},\ldots,x_{2N}) , &
    \text{if } \upwedgeat{j} \in \alpha .
\end{cases}
\end{align}
We prove the property~\eqref{eq: multiple SLE asymptotics with Dyck paths} for 
$\tilde{\PartF}_\alpha$ separately in the two cases 
$\upwedgeat{j} \in \alpha$ and $\upwedgeat{j} \notin \alpha$.

\textit{Assume first that $\upwedgeat{j} \notin \alpha$.} We split the right-hand side  
of~\eqref{eq: Second formula for pure partition function for kappa 4} into three sums:
\begin{align*}
\tilde{\PartF}_\alpha  = 
\sum_{\beta \DPgeq \alpha \colon \downwedgeat{j} \in \beta } \Minv_{\alpha,\beta}
 \, \CobloF_\beta
 \, + \sum_{\beta \DPgeq \alpha \colon \upwedgeat{j} \in \beta } \Minv_{\alpha,\beta} \,
\CobloF_\beta
 \, + \sum_{\beta \DPgeq \alpha \colon \slopeat{j}\in \beta } \Minv_{\alpha,\beta}
 \, \CobloF_\beta . 
\end{align*}
Using Lemma~\ref{lem: combinatorics}(b), we combine the first and second sums to one sum over $\beta$ 
such that $\downwedgeat{j} \in \beta$, by replacing $\beta$ in the 
second sum by $\beta \wedgelift{j}$. Furthermore, Lemma~\ref{lem: combinatorics}(d) shows that 
the coefficients in these two sums are related by
$\Minv_{\alpha,\beta} = - \Minv_{\alpha,\beta \wedgelift{j}}$.
Therefore, we obtain
\begin{align*}
\tilde{\PartF}_\alpha  = 
\sum_{\beta \DPgeq \alpha \colon \downwedgeat{j} \in \beta } \Minv_{\alpha,\beta}
 \, \big( \CobloF_\beta - \CobloF_{\beta \wedgelift{j} } \big) 
 \, + \sum_{\beta \DPgeq \alpha \colon \slopeat{j} \in \beta } \Minv_{\alpha,\beta}
 \, \CobloF_\beta . 
\end{align*}
Now, it follows from Lemma~\ref{lem: CobloF ASY} that the last sum vanishes in the 
limit~\eqref{eq: multiple SLE asymptotics with Dyck paths}, and that the functions
$\CobloF_\beta$ and $\CobloF_{\beta \wedgelift{j} }$ have the same limit, so they cancel.
In conclusion, the limit~\eqref{eq: multiple SLE asymptotics with Dyck paths} of $\tilde{\PartF}_\alpha$
is zero when $\upwedgeat{j} \notin \alpha$.

\textit{Assume then that $\upwedgeat{j} \in \alpha$.} By Proposition~\ref{prop: matrix inversion},
the system~\eqref{eq: Second formula for pure partition function for kappa 4} with $\alpha \in \DP_N$
is invertible, and
\begin{align}\label{eq: Formula for conformal block function for kappa 4}
\CobloF_\beta(x_1, \ldots, x_{2N}) 
= \; & \sum_{\alpha \in \DP_N} \Mmat_{\beta,\alpha} \, \tilde{\PartF}_\alpha(x_1, \ldots, x_{2N}) ,
\qquad \text{for any } \beta \in \DP_N ,
\end{align}
where $\Mmat_{\beta,\alpha} = \one\{\beta \KWleq \alpha \}$.
We already know by the first part of the proof that 
the limit~\eqref{eq: multiple SLE asymptotics with Dyck paths} of $\tilde{\PartF}_\alpha$
is zero when $\upwedgeat{j} \notin \alpha$. Therefore, taking the 
the limit~\eqref{eq: multiple SLE asymptotics with Dyck paths} of the right-hand side 
of~\eqref{eq: Formula for conformal block function for kappa 4} gives
\begin{align}
& \; \sum_{\alpha \in \DP_N} \one\{\beta \KWleq \alpha \} \, \lim_{x_j , x_{j+1} \to \xi} 
\frac{\tilde{\PartF}_\alpha(x_1 , \ldots , x_{2N})}{(x_{j+1} - x_j)^{-1/2}} 
\nonumber \\
= \; & \sum_{\alpha \colon \upwedgeat{j} \in \alpha} \one\{\beta \KWleq \alpha \} \, \lim_{x_j , x_{j+1} \to \xi} 
\frac{\tilde{\PartF}_\alpha(x_1 , \ldots , x_{2N})}{(x_{j+1} - x_j)^{-1/2}} ,
\qquad \text{for any } \beta \in \DP_N .
\label{eq: Formula1}
\end{align}
We want to calculate the limit in~\eqref{eq: Formula1} for any fixed $\alpha \in \DP_N$ such that
$\upwedgeat{j} \in \alpha$. 

By Lemma~\ref{lem: combinatorics}(c), we have 
$\beta \KWleq \alpha$ if and only if $\wedgeat{j} \in \beta$ and 
$\beta \removewedge{j} \KWleq \alpha \removeupwedge{j}$.
Now, choose $\beta \in \DP_N$ such that $\upwedgeat{j} \in \beta$, and denote
$\hat{\beta} = \beta \removeupwedge{j}$. Then, by Lemma~\ref{lem: combinatorics}(c), we have 
$\one\{\beta \KWleq \alpha \} = \one\{\hat{\beta} \KWleq \hat{\alpha} \}$ and we can re-index
the sum in~\eqref{eq: Formula1} by $\hat{\alpha} = \alpha \removeupwedge{j}$, to obtain
\begin{align}
\sum_{\hat{\alpha} \in \DP_{N-1}} \one\{\hat{\beta} \KWleq \hat{\alpha} \} \, \lim_{x_j , x_{j+1} \to \xi} 
\frac{\tilde{\PartF}_\alpha(x_1 , \ldots , x_{2N})}{(x_{j+1} - x_j)^{-1/2}} .
\label{eq: Formula1.1}
\end{align}
On the other hand, with $\upwedgeat{j} \in \beta$, Lemma~\ref{lem: CobloF ASY}
gives the limit~\eqref{eq: multiple SLE asymptotics with Dyck paths} of the left-hand side 
of~\eqref{eq: Formula for conformal block function for kappa 4}:
\begin{align}
\frac{\CobloF_\beta(x_1, \ldots, x_{2N}) }{(x_{j+1} - x_j)^{-1/2}} 
= \; &  \CobloF_{ \hat{\beta}} (x_1 , \ldots, x_{j-1}, x_{j+2}, \ldots, x_{2N}) 
\nonumber \\
= \; & \sum_{\hat{\alpha} \in \DP_{N-1}} \one\{\hat{\beta} \KWleq \hat{\alpha} \} \, 
\tilde{\PartF}_{\hat{\alpha}}(x_1 , \ldots, x_{j-1}, x_{j+2}, \ldots, x_{2N})  ,
\label{eq: Formula2}
\end{align}
where in the last equality we used~\eqref{eq: Formula for conformal block function for kappa 4} 
for $\hat{\beta} = \beta \removeupwedge{j}$.
Combining~\eqref{eq: Formula1.1}~and~\eqref{eq: Formula2}, we arrive with 
\begin{align}
\; & \sum_{\hat{\alpha} \in \DP_{N-1}} \Mmat_{\hat{\beta},\hat{\alpha}} \, \lim_{x_j , x_{j+1} \to \xi} 
\frac{\tilde{\PartF}_\alpha(x_1 , \ldots , x_{2N})}{(x_{j+1} - x_j)^{-1/2}} 
\nonumber \\
= \; & \sum_{\hat{\alpha} \in \DP_{N-1}} \Mmat_{\hat{\beta},\hat{\alpha}} \, 
\tilde{\PartF}_{\hat{\alpha}}(x_1 , \ldots, x_{j-1}, x_{j+2}, \ldots, x_{2N}) ,
\qquad \text{for any } \hat{\beta} \in \DP_{N-1} ,
\label{eq: final Formula}
\end{align}
where $\Mmat_{\hat{\beta},\hat{\alpha}} = \one\{\hat{\beta} \KWleq \hat{\alpha} \}$ 
and $\alpha \in \LP_N$ is determined by $\hat{\alpha} = \alpha \removeupwedge{j}$.
Recalling that by Proposition~\ref{prop: matrix inversion}, the system~\eqref{eq: final Formula} 
is invertible, we can solve for the asserted limit~\eqref{eq: multiple SLE asymptotics with Dyck paths}.
This concludes the proof.
\end{proof}

\begin{proof} [Proof of Theorem~\ref{thm: sle4 pure pffs}] 
We first note that the functions defined by the right-hand side of~\eqref{eq: Pure partition function for kappa 4} 
satisfy the normalization $\PartF_\emptyset = 1$ and the bound~\eqref{eqn::powerlawbound}
--- the latter follows immediately from the definition~\eqref{eq: CobloF definition} of $\CobloF_\alpha$,
since the coefficients $\Minv_{\alpha,\beta}$ do not depend on the variables $x_1,\ldots, x_{2N}$.
Lemmas~\ref{lem: CobloF PDE}~and~\ref{lem: CobloF COV} show that the functions $\CobloF_\alpha$ 
satisfy the properties $\mathrm{(PDE)}$~\eqref{eq: multiple SLE PDEs} and $\mathrm{(COV)}$~\eqref{eq: multiple SLE Mobius covariance} with $\kappa = 4$, 
whence the right-hand side of~\eqref{eq: Pure partition function for kappa 4} also satisfies these properties by linearity.
Furthermore, Lemma~\ref{lem: Pure partition function ASY} shows that these functions also enjoy
the asymptotics property $\mathrm{(ASY)}$~\eqref{eq: multiple SLE asymptotics} of the pure partition functions $\PartF_\alpha$ with $\kappa = 4$.
Thus, the uniqueness Corollary~\ref{cor::purepartition_unique} shows that the right-hand side of~\eqref{eq: Pure partition function for kappa 4} 
must be equal to $\PartF_\alpha$.
\end{proof}

\subsection{GFF Interpretation}  
\label{subsec: Conformal blocks for GFF}
In this final section, we give an interpretation for the functions  $\CobloF_\alpha$ appearing in
Theorem~\ref{thm: sle4 pure pffs} as partition functions associated to a particular boundary data of the $\GFF$. 

For $\alpha\in\LP_N$, recall that we also denote by $\alpha\in\DP_N$ the corresponding Dyck path~\eqref{eq::DP}. 
Let $\gff_{\alpha}$ be the $\GFF$ on $\HH$ with the following boundary data:
\begin{align}\label{eq: conformal block boundary values with Dyck path}
\lambda (2\alpha(k) - 1) , 
\quad & \text{if } x \in ( x_{k}, x_{k+1} ) , \qquad \text{ for all } k \in \{0,1,\ldots,2N\} .
\end{align}
Note that by this definition, the boundary value of $\gff_{\alpha}$ is always $-\lambda$ 
on $(-\infty, x_1)\cup (x_{2N}, \infty)$, and $+\lambda$ on $(x_1, x_2)\cup (x_{2N-1}, x_{2N})$. 
Define
\begin{align}\label{eqn::conformalblock_gff_heights}
\Height_{\alpha}(k) := \lambda(\alpha(k-1)+\alpha(k)-1) , \qquad \text{ for all }k\in \{1, 2, \ldots, 2N\}. 
\end{align}
Then we always have $\Height_{\alpha}(1)=\Height_{\alpha}(2N)=0$. 

\begin{proposition}\label{prop::levellines_conformalblocks}
Let $\alpha\in\LP_N$. Let $\gff_{\alpha}$ be the 
$\GFF$ on $\HH$ with boundary data given by~\eqref{eq: conformal block boundary values with Dyck path}.
For all $\link{a}{b} \in \alpha$, let $\eta_{a}$ \textnormal{(}resp.~$\eta_{b}$\textnormal{)}
be the level line of $\gff_{\alpha}$ \textnormal{(}resp.~$-\gff_{\alpha}$\textnormal{)}
with height $\Height_{\alpha}(a)$ \textnormal{(}resp.~$-\Height_{\alpha}(b)$\textnormal{)}
Then, the collection $(\eta_1, \ldots, \eta_{2N})$ is a local $N$-$\SLE_4$ 
with partition function $\CobloF_\alpha$. 
\end{proposition}
\begin{proof}
The collection $(\eta_1, \ldots, \eta_{2N})$ clearly satisfies the conformal invariance 
$\mathrm{(CI)}$ and the domain Markov property $\mathrm{(DMP)}$ from the definition of local multiple $\SLE$s
in Section \ref{subsec::global_vs_local}. Thus, we only need to check the marginal law property
$\mathrm{(MARG)}$ for each curve. 
We do this for $\eta_{a}$.
On the one hand, as $\eta_{a}$ is the level line of $\gff_{\alpha}$ with height 
$\Height_{\alpha}(a)$, 
its marginal law is an $\SLE_4(\underline{\rho})$ with force points $\{x_1, \ldots, x_{2N}\}\setminus\{x_{a}\}$, 
where each $x_{a_j}$ (resp.~$x_{b_j}$) is a force point with weight $+2$ (resp.~$-2$). 
Therefore, the driving function $W_t$ of $\eta_{a}$ satisfies the SDEs
\begin{align}
\label{eqn::localmultiple_marginal_sde_for_conformal_block}
\begin{split}
\ud W_t = \; & 2\ud B_t \, + \,  \sum_{i\neq a}\frac{\rho_i \ud t}{W_t-V_t^i} , \qquad 
\text{where } \quad \rho_i = \begin{cases}
+2 , \quad & \text{if } i \in \{a_1, \ldots, a_N \} \setminus \{a\}, \\
-2 , & \text{if } i \in \{b_1, \ldots, b_N\} .
\end{cases} \\
\ud V_t^i = \; & \frac{2\ud t}{V_t^i-W_t},\quad \text{ for } i \neq a ,
\end{split}
\end{align}
where $V_t^i$ are the time evolutions of the force points $x_i$, for $i \neq a$.
On the other hand,~\eqref{eqn::localmultiple_marginal_sde_for_conformal_block} coincides with the SDE 
system~\eqref{eqn::localmultiple_marginal_sde} of $\mathrm{(MARG)}$ with 
$F_j = 2 \partial_{a} \log \CobloF_{\alpha}$, since
by definition~\eqref{eq: CobloF definition} of $\CobloF_{\alpha}$, we have 
\begin{align*}
4\partial_{a} \log \CobloF_{\alpha}
\, = \, \sum_{i \neq a}\frac{2\vartheta_{\alpha}(i,a)}{W_t-V_t^i}
\, = \, \sum_{i \neq a}\frac{\rho_i}{W_t-V_t^i} . 
\end{align*}
This completes the proof.
\end{proof}

Let $\alpha=\unnested_N$ be the completely unnested link pattern. Then, $\gff_{\unnested_N}$ 
is the $\GFF$ on $\HH$ with alternating boundary data, $\Height_{\unnested_N}(k)=0$, for all $k$, and 
$\CobloF_{\unnested_N}=\PartF_{\GFF}^{(N)}$. This is the situation discussed in 
Section~\ref{subsec:GFF_connection_proba}. 
By Theorem~\ref{thm::multiple_sle_4}, all connectivities $\beta \in \LP_N$ for 
the level lines of  $\gff_{\unnested_N}$ have a positive chance.

However, for a general link pattern
$\alpha \in \LP_N \setminus \{\unnested_N\}$, the boundary data for $\gff_\alpha$ is more 
complicated, and its level lines cannot necessarily form all of the different connectivities:
only level lines of $\gff_\alpha$ and level lines of $-\gff_\alpha$ with respective heights $\Height$ and
$-\Height$ of the same magnitude can connect with each other.
For example,
when $\alpha=\nested_N$, then $\gff_{\nested_N}$ is the $\GFF$ with the following boundary data: 
\begin{align*}
\begin{cases}
\lambda ( 2j - 1),\qquad & \text{if } x \in ( x_{j}, x_{j+1} ) , 
\quad \text{for all } j \in \{0,1,\ldots,N\} , \\
\lambda ( 4N - 1 - 2j), \qquad &\text{if } x \in ( x_{j}, x_{j+1} ) ,
\quad \text{for all } j  \in \{N+1,N+2,\ldots,2N \} ,
\end{cases}
\end{align*}
and the heights of the level lines are 
$\Height_{\nested_N}(k) = 2\lambda(k-1)$, for $k \in \{1, \ldots, N\}$ and $\Height_{\nested_N}(k) = 2\lambda(2N-k)$,
for $k \in\{N+1, \ldots, 2N\}$. 
In this case, we have $\CobloF_{\nested_N}=\PartF_{\nested_N}$, and 
for $j\in \{1, \ldots, N\}$, the curve $\eta_j$ merges with $\eta_{2N+1-j}$ almost surely, 
that is, the level lines necessarily form the rainbow connectivity $\nested_N$. 
The marginal law of $\eta_1$ is the $\SLE_4(+2,\ldots,+2,-2,\ldots,-2)$ 
in $\HH$ from $x_1$ to $x_{2N}$ with force points $(x_2, \ldots, x_{2N-1})$, where $x_k$ 
(resp.~$x_l$) is a force point with weight $+2$ for $k \le N$ (resp.~with weight $-2$ for $l \ge N+1$).

\begin{remark}
For each link pattern $\alpha \in \LP_N$, we associate a balanced subset  
$S(\alpha) \subset \{1,\ldots,2N\}$ (that is, a subset containing equally many even and odd indices)
as follows. Write $\alpha = \{ \link{a_1}{b_1}, \ldots, \link{a_N}{b_N} \}$ as an ordered collection 
as in~\eqref{eq: begin and end point convention sec 2}. Define
\[ S(\alpha) := \{a_r \; \colon \; r \in \{1,\ldots,N\} \text{ and } a_r \text{ is odd }\} 
\cup \{b_s \; \colon \; s \in \{1,\ldots,N\} \text{ and } b_s \text{ is even }\} . \]
Let $\gff$ be the $\GFF$ on $\HH$ with alternating boundary data. Then the probability that 
the level lines of $\gff$ connect the points with indices in $S(\alpha)$ among themselves and the points
with indices in the complement $\{1,2,\ldots,2N\} \setminus S(\alpha)$ among themselves equals
\begin{align*}
\frac{\CobloF_\alpha (x_1, \ldots, x_{2N}) }{\PartF_{\GFF}^{(N)} (x_1, \ldots, x_{2N})}.
\end{align*}
This fact was proved in~\cite{Kenyon-Wilson:Boundary_partitions_in_trees_and_dimers}
for interfaces in the double-dimer model. The corresponding claim for the level lines of the $\GFF$ 
can be proved similarly.
\end{remark}

\appendix

\section{Properties of Bound Functions}
\label{subsec:appendix-prop}
We first recall the definition of the bound functions:
we set $\LB_\emptyset := 1$ and, for all $\alpha \in \LP_N$ and for all nice polygons $(\Omega; x_1, \ldots, x_{2N})$, we define
\begin{align*}
\LB_{\alpha}(\Omega; x_1,\ldots, x_{2N}) := \prod_{\link{a}{b} \in \alpha} H_{\Omega}(x_{a}, x_{b})^{1/2} .
\end{align*}
We also note that, by the monotonicity property~\eqref{eqn::poissonkernel_mono} of the boundary
Poisson kernel, for any sub-polygon $(U; x_1, \ldots, x_{2N})$ we have the inequality
\begin{align}\label{eqn::b_alpha_mono}
\LB_{\alpha}(U; x_1, \ldots, x_{2N})\le \LB_{\alpha}(\Omega; x_1, \ldots, x_{2N}).  
\end{align} 
Then we collect some useful properties of the functions $\LB_{\alpha}$ with $\Omega=\HH$:
\begin{align}\label{eqn::b_alpha_def}
\LB_{\alpha} \colon \chamber_{2N} \to \Rpos, 
\qquad \qquad
&\LB_{\alpha}(x_1, \ldots, x_{2N}) := \prod_{\link{a}{b} \in \alpha}  |x_{b}-x_{a}|^{-1}.
\end{align}

\begin{lemma}\label{lem::b_alpha_asy_refined}
The function $\LB_{\alpha}$ 
satisfies the following 
asymptotics: with $\hat{\alpha} = \alpha \removeLink \link{j}{j+1}$, we have
\begin{align*}
\lim_{\substack{\tilde{x}_j , \tilde{x}_{j+1} \to \xi, \\ \tilde{x}_i\to x_i \text{ for } i \neq j, j+1}} 
\frac{\LB_\alpha(\tilde{x}_1 , \ldots , \tilde{x}_{2N})}{(\tilde{x}_{j+1} - \tilde{x}_j)^{-1}} 
=\begin{cases}
0 , \quad &
    \text{if } \link{j}{j+1} \notin \alpha , \\
\LB_{\hat{\alpha}}(x_{1},\ldots,x_{j-1},x_{j+2},\ldots,x_{2N}) , &
    \text{if } \link{j}{j+1} \in \alpha ,
\end{cases}
\end{align*}
for all $\alpha \in \LP_N$, and for all $j \in \{1, \ldots, 2N-1 \}$ and  $x_1 < \cdots < x_{j-1} < \xi < x_{j+2} < \cdots < x_{2N}$.  
\end{lemma}
\begin{proof}
This follows immediately from the definition~\eqref{eqn::b_alpha_def}.
\end{proof}

Then we define, for all $N\ge 1$, the functions
\begin{align}\label{eqn::b_n_def}
\LB^{(N)} \colon \chamber_{2N} \to \Rpos, 
\qquad \qquad
&\LB^{(N)}(x_1, \ldots, x_{2N}) := \prod_{1\leq k<l\leq 2N}(x_l-x_k)^{(-1)^{l-k}}.
\end{align}
We note the connection with the symmetric partition function 
$\PartF_{\GFF}$ for $\kappa=4$ defined in Lemma~\ref{lem: gff symmetric pff}:
\begin{align*}
\LB^{(N)}(x_1, \ldots, x_{2N})=\PartF_{\GFF}^{(N)}(x_1, \ldots, x_{2N})^2 .
\end{align*}
Also, for a nice polygon $(\Omega; x_1, \ldots, x_{2N})$,  we define
\begin{align*}
\LB^{(N)}(\Omega; x_1, \ldots, x_{2N}) := \prod_{i=1}^{2N} |\varphi'(x_i)| \times \LB^{(N)}(\varphi(x_1), \ldots, \varphi(x_{2N})),
\end{align*}
where $\varphi \colon\Omega\to\HH$ is any conformal map such that $\varphi(x_1)<\cdots<\varphi(x_{2N})$.

\begin{lemma}\label{lem::b_total_asy_refined}
For all $n \in \{ 1, \ldots, N \}$
and $\xi < x_{2n+1} < \cdots < x_{2N}$, we have 
\begin{align*}
\lim_{\substack{\tilde{x}_1, \ldots, \tilde{x}_{2n} \to \xi, \\ \tilde{x}_{i} \to x_i \text{ for } 2n < i \le  2N}} 
\frac{\LB^{(N)}(\tilde{x}_1, \ldots, \tilde{x}_{2N})}{\LB^{(n)}(\tilde{x}_1, \ldots, \tilde{x}_{2n})} 
= \LB^{(N-n)}(x_{2n+1},\ldots, x_{2N}).
\end{align*}
\end{lemma}
\begin{proof}
This follows immediately from the definition~\eqref{eqn::b_n_def}.
\end{proof}

\section{Technical Lemmas}
\label{subsec:appendix-tech}
In this appendix, we prove useful results of technical nature.
The main result is the next proposition, which we prove in the end of the appendix. 

\begin{proposition}\label{prop:: strong limit for ratio Zalpha and Ztotal}
Let $\alpha \in \LP_N$ and suppose that $\link{1}{2} \in \alpha$.
Fix an index $n \in \{2,\ldots, N\}$ and real points $x_1 < \cdots < x_{2N}$.
Suppose $\eta$ is a continuous simple curve in $\HH$ starting from $x_1$ and terminating at $x_{2n}$ at time $T$, 
which hits $\R$ only at $\{x_1, x_{2n}\}$. Let $(W_t, 0\le t\le T)$ be its Loewner driving function and $(g_t, 0\le t\le T)$ 
the corresponding conformal maps. Then we have
\begin{align}\label{eq: strong limit for ratio Zalpha and Ztotal}
\lim_{t \to T} 
\frac{\LB_\alpha (W_t, g_t(x_2), \ldots, g_t(x_{2N})) }{\LB^{(N)} (W_t, g_t(x_2), \ldots, g_t(x_{2N})) }
= 0 .
\end{align}
\end{proposition}

For the proof, we need a few lemmas.

\begin{lemma} \label{lem::continuouscurve_limit_technical}
Let $x_1<x_2<x_3<x_4$. 
Suppose $\eta$ is a continuous simple curve in $\HH$ starting from $x_1$ and terminating at $x_4$ at time $T$, 
which hits $\R$ only at $\{x_1, x_4\}$. Let $(W_t, 0\le t\le T)$ be its Loewner driving function and $(g_t, 0\le t\le T)$ 
the corresponding conformal maps. Define, for $t<T$, 
\[ \Delta_t := \frac{(g_t(x_4)-W_t)(g_t(x_3)-g_t(x_2))}{(g_t(x_4)-g_t(x_2))(g_t(x_3)-W_t)} . \]
Then we have $0\le \Delta_t\le 1$, for all $t<T$, and $\Delta_t\to 0$ as $t\to T$. 
\end{lemma}
\begin{proof}
The bound $0\le \Delta_t\le 1$ follows 
from Remark~\ref{rem::crossratio_trivialbound} and the fact that $W_t< g_t(x_2)< g_t(x_3)<g_t(x_4)$.
It remains to check the limit of $\Delta_t$ as $t\to T$. 
To simplify notations, we denote $g_t(x_2)-W_t$ by $X_{21}$ and $g_t(x_3)-g_t(x_2)$ by $X_{32}$, and $g_t(x_4)-g_t(x_3)$ by $X_{43}$. Then we have 
\[\Delta_t=\frac{(X_{43}+X_{32}+X_{21})X_{32}}{(X_{43}+X_{32})(X_{32}+X_{21})}=\frac{X_{32}/X_{21}+X_{32}/X_{43}+X_{32}^2/(X_{21}X_{43})}{1+X_{32}/X_{21}+X_{32}/X_{43}+X_{32}^2/(X_{21}X_{43})}.\]
To show that $\Delta_t\to 0$ as $t \to T$, it suffices to show that
\begin{align}\label{eqn::technical_crossratio_zero}
X_{32}/X_{21}\to 0 \qquad \text{and}\qquad X_{32}/X_{43}\to 0. 
\end{align}
For $z\in\C$, denote by $\PP^z$ the law of Brownian motion in $\C$ started from $z$. 
Let $\tau$ be the first time when $B$ exits $\HH\setminus\eta[0,t]$.
Then by \cite[Remark 3.50]{LawlerConformallyInvariantProcesses}, we have 
\begin{align*}
X_{43}=\lim_{y\to\infty} y \, \PP^{yi}[B_{\tau}\in (x_3, x_4)], \qquad 
X_{32}=\lim_{y\to\infty} y \, \PP^{yi}[B_{\tau}\in (x_2, x_3)],
\end{align*}
and $X_{21}$ is the same limit of the probability that $B_{\tau}$ belongs the union of the right side of $\eta[0,t]$ and $(x_1, x_2)$. 
Property~\eqref{eqn::technical_crossratio_zero} follows from this. 
\end{proof}

\begin{lemma}\label{lem::crossratio_zerowhenwrong_general}
Fix an index $n \in \{2,3, \ldots, N\}$ and real points $x_1 < \cdots < x_{2n}$.
Suppose $\eta$ is a continuous simple curve in $\HH$ starting from $x_1$ and terminating at $x_{2n}$ at time $T$, 
which hits $\R$ only at $\{x_1, x_{2n}\}$. Let $(W_t, 0\le t\le T)$ be its Loewner driving function and $(g_t, 0\le t\le T)$ 
the corresponding conformal maps. Then we have
\begin{align*} 
\lim_{t \to T} \prod_{j=3}^{2n} \left( \frac{g_t(x_j) - W_t}{g_t(x_j )- g_t(x_2)} \right)^{(-1)^{j}} = 0 .
\end{align*}
\end{lemma}
\begin{proof}
For all odd $j \in \{3,5,\ldots, 2n-3\}$, Remark~\ref{rem::crossratio_trivialbound} shows that 
\begin{align*}
0\le \frac{(g_t(x_j)-g_t(x_2))(g_t(x_{j+1})-W_t)}{(g_t(x_j)-W_t)(g_t(x_{j+1})-g_t(x_2))}\le 1.
\end{align*}
Combining this with Lemma~\ref{lem::continuouscurve_limit_technical}, we see that, when $t \to T$, we have
\begin{align*} 
0 \leq \prod_{j=3}^{2n} \left( \frac{g_t(x_j) - W_t}{g_t(x_j )- g_t(x_2)} \right)^{(-1)^{j}} \le 
\; \frac{(g_t(x_{2n-1})-g_t(x_2))(g_t(x_{2n})-W_t)}{(g_t(x_{2n-1})-W_t)(g_t(x_{2n})-g_t(x_2))}  \; \longrightarrow \; 0 . 
\end{align*}
This proves the lemma.
\end{proof}

Next, for any $\alpha \in \LP_N$ and $n \in \{1, \ldots, N\}$, we define the function 
\begin{align}\label{eq: auxiliary function}
F_\alpha^{(n)} (x_1, \ldots, x_{2N})
:= \frac{\LB_{\alpha}(x_1, \ldots, x_{2N})}{\LB^{(n)}(x_1, \ldots, x_{2n})}. 
\end{align}

\begin{lemma}\label{lem::mart_zerowhenwrong_induction}
Let $\alpha \in \LP_N$ and suppose that $\link{1}{2} \in \alpha$. Then for all $n \in \{1, \ldots, N\}$, with $\hat{\alpha} = \alpha \removeLink \link{1}{2}$, we have 
\begin{align*}
F_\alpha^{(n)} (x_1, x_2, x_3, \ldots, x_{2N})
= \prod_{j=3}^{2n} \left( \frac{x_j - x_1}{x_j - x_2} \right)^{(-1)^{j}} \times F_{\hat{\alpha}}^{(n-1)} (x_3, x_4, \ldots, x_{2N}) .
\end{align*}
\end{lemma}
\begin{proof}
This follows immediately from the definition~\eqref{eq: auxiliary function} of $F_\alpha^{(n)}$.
\end{proof}

\begin{lemma}\label{lem::finiteness_property}
For any $\alpha \in \LP_N$, $n \in \{1, \ldots, N\}$, and $\xi < x_{2n+1} < \cdots < x_{2N}$, we have 
\begin{align}\label{eq::finiteness_property}
\limsup_{\substack{\tilde{x}_1, \ldots, \tilde{x}_{2n} \to \xi, \\ \tilde{x}_{i} \to x_i \text{ for } 2n < i \le  2N}} 
F_\alpha^{(n)} (\tilde{x}_1, \ldots, \tilde{x}_{2N}) < \infty .
\end{align}
\end{lemma}
\begin{proof}
We prove the claim by induction on $N \ge 1$. It is clear for $N = 1$, as $F_\emptyset^{(0)} = 1$. 
Assume then that
\begin{align*}
\limsup_{\substack{\tilde{x}_1, \ldots, \tilde{x}_{2\ell} \to y, \\ \tilde{x}_{i} \to x_i \text{ for } 2\ell < i \le  2N-2}} 
F_\beta^{(\ell)} (\tilde{x}_1, \ldots, \tilde{x}_{2N-2}) < \infty 
\end{align*}
holds for all $\beta \in \LP_{N-1}$, $\ell \in \{1, \ldots, N-1\}$, and $y < x_{2\ell+1} < \cdots < x_{2N-2}$. Let 
$\alpha \in \LP_N$, $n \in \{1, \ldots, N\}$, and $\xi < x_{2n+1} < \cdots < x_{2N}$. 
Choose $j$ such that $\link{j}{j+1} \in \alpha$. We consider three cases.

\begin{itemize}
\item $j+1 \leq 2n$: In this case, by Lemma \ref{lem::mart_zerowhenwrong_induction}, we have
\begin{align*}
F_\alpha^{(n)} (\tilde{x}_1, \ldots, \tilde{x}_{2N}) = \prod_{\substack{1 \leq i \leq 2n, \\ i \neq j, j+1 }} 
\bigg| \frac{\tilde{x}_i - \tilde{x}_j}{\tilde{x}_i - \tilde{x}_{j+1}} \bigg|^{(-1)^{i+j+1}} \times
F_{\alpha \removeLink \link{j}{j+1}}^{(n-1)} (\tilde{x}_1, \ldots, \tilde{x}_{j-1}, \tilde{x}_{j+2}, \ldots, \tilde{x}_{2N})  .
\end{align*}
Using Remark~\ref{rem::crossratio_trivialbound}, we see that 
if $j$ is odd, then we have  
\[ \prod_{\substack{1 \leq i \leq 2n, \\ i \neq j, j+1 }} 
\bigg| \frac{\tilde{x}_i - \tilde{x}_j}{\tilde{x}_i - \tilde{x}_{j+1}} \bigg|^{(-1)^{i+j+1}} 
= \prod_{\substack{1\le m\le n , \\ m\neq (j+1)/2} }
\bigg|\frac{(\tilde{x}_{2m-1}-\tilde{x}_{j+1})(\tilde{x}_{2m}-\tilde{x}_{j})}{(\tilde{x}_{2m-1}-\tilde{x}_{j})(\tilde{x}_{2m}-\tilde{x}_{j+1})}\bigg| \le 1 ,\]
and if $j$ is even, 
then we have 
\begin{align*}
 \prod_{\substack{1 \leq i \leq 2n, \\ i \neq j, j+1 }} 
\bigg| \frac{\tilde{x}_i - \tilde{x}_j}{\tilde{x}_i - \tilde{x}_{j+1}} \bigg|^{(-1)^{i+j+1}} 
= \; \bigg|\frac{(\tilde{x}_1-\tilde{x}_j)(\tilde{x}_{2n}-\tilde{x}_{j+1})}{(\tilde{x}_1-\tilde{x}_{j+1})(\tilde{x}_{2n}-\tilde{x}_j)}\bigg| \;
\prod_{\substack{1\le m<n , \\  m\neq j/2} }
\bigg|\frac{(\tilde{x}_{2m}-\tilde{x}_{j+1})(\tilde{x}_{2m+1}-\tilde{x}_{j})}{(\tilde{x}_{2m}-\tilde{x}_{j})(\tilde{x}_{2m+1}-\tilde{x}_{j+1})}\bigg|\le 1.
\end{align*}
Thus, we have 
\[F_\alpha^{(n)} (\tilde{x}_1, \ldots, \tilde{x}_{2N}) \le 
F_{\alpha \removeLink \link{j}{j+1}}^{(n-1)} (\tilde{x}_1, \ldots, \tilde{x}_{j-1}, \tilde{x}_{j+2}, \ldots, \tilde{x}_{2N}) ,
\]
so by the induction hypothesis, 
$F_\alpha^{(n)}$ remains finite in the limit~\eqref{eq::finiteness_property}.

\item $j > 2n$: In this case, we have
\begin{align*}
F_\alpha^{(n)} (\tilde{x}_1, \ldots, \tilde{x}_{2N}) = (\tilde{x}_{j+1} - \tilde{x}_j)^{-1} \;
F_{\alpha \removeLink \link{j}{j+1}}^{(n)} (\tilde{x}_1, \ldots, \tilde{x}_{j-1}, \tilde{x}_{j+2}, \ldots, \tilde{x}_{2N}) ,
\end{align*}
which by the induction hypothesis remains finite in the limit~\eqref{eq::finiteness_property}.

\item $j = 2n$: In this case, we have
\begin{align*}
F_\alpha^{(n)} (\tilde{x}_1, \ldots, \tilde{x}_{2N}) 
= & \; 
\left( \frac{\tilde{x}_{2n} - \tilde{x}_{2n-1}}{\tilde{x}_{2n+1} - \tilde{x}_{2n}} \right) \;
\prod_{i = 1}^{2n - 2 } \left( \frac{\tilde{x}_{2n-1} - \tilde{x}_i}{\tilde{x}_{2n} - \tilde{x}_i} \right)^{(-1)^{i}} 
\\
& \; \times F_{\alpha \removeLink \link{j}{j+1}}^{(n-1)} (\tilde{x}_1, \ldots, \tilde{x}_{j-1}, \tilde{x}_{j+2}, \ldots, \tilde{x}_{2N}) .
\end{align*}
By Remark \ref{rem::crossratio_trivialbound}, we have 
\[\prod_{i = 1}^{2n - 2 } \left( \frac{\tilde{x}_{2n-1} - \tilde{x}_i}{\tilde{x}_{2n} - \tilde{x}_i} \right)^{(-1)^{i}} 
= \; \prod_{m = 1}^{n-1}\frac{(\tilde{x}_{2n}-\tilde{x}_{2m-1})(\tilde{x}_{2n-1}-\tilde{x}_{2m})}{(\tilde{x}_{2n-1}-\tilde{x}_{2m-1})(\tilde{x}_{2n}-\tilde{x}_{2m})} \; \le \; 1 . \]
By the induction hypothesis, the limit~\eqref{eq::finiteness_property} of $F_{\alpha \removeLink \link{j}{j+1}}^{(n-1)}$ is finite, so
we see that $F_\alpha^{(n)}$ also remains finite in the limit~\eqref{eq::finiteness_property} 
(in fact, the limit of $F_\alpha^{(n)}$ is zero in this case).
\end{itemize}

This completes the proof.
\end{proof}

\begin{proof}[Proof of Proposition \ref{prop:: strong limit for ratio Zalpha and Ztotal}]
Write $\LB_\alpha / \LB^{(N)}  = \big( \LB^{(n)} / \LB^{(N)} \big) \big( \LB_\alpha / \LB^{(n)} \big)$.
Then, Lemma~\ref{lem::b_total_asy_refined} shows that
in the limit~\eqref{eq: strong limit for ratio Zalpha and Ztotal},
we have 
$ \LB^{(N)}/\LB^{(n)} \to \LB^{(N-n)} > 0$. Thus, it suffices to show that 
$F_\alpha^{(n)}=\LB_\alpha / \LB^{(n)} \to 0$ 
in this limit.  
Indeed, combining Lemmas \ref{lem::crossratio_zerowhenwrong_general}--\ref{lem::finiteness_property}, 
we see that in the limit $t \to T$, we have
\begin{align*}
\frac{\LB_\alpha(W_t, g_t(x_2), \ldots, g_t(x_{2N}))}
{\LB^{(n)}(W_t, g_t(x_2), \ldots, g_t(x_{2n}))} 
= & \;  F_\alpha^{(n)}(W_t, g_t(x_2), \ldots, g_t(x_{2N})) \\
= & \;  \prod_{j=3}^{2n} \left( \frac{g_t(x_j) - W_t}{g_t(x_j) - g_t(x_2)} \right)^{(-1)^{j}}
\times  F_{\hat{\alpha}}^{(n-1)} (g_t(x_3), \ldots, g_t(x_{2N}))
\; \longrightarrow \; 0 ,
\end{align*}
where $\hat{\alpha} = \alpha \removeLink \link{1}{2}$.
This concludes the proof.
\end{proof}

{\small
\singlespacing

\newcommand{\etalchar}[1]{$^{#1}$}

}


\begin{thebibliography}{CDCH{\etalchar{+}}14}

\bibitem[BB03]{Bauer-Bernard:Conformal_field_theories_of_SLEs}
Michel Bauer and Denis Bernard.
\newblock Conformal field theories of stochastic {L}oewner evolutions.
\newblock {\em Comm. Math. Phys.}, 239(3):493--521, 2003.

\bibitem[BB04]{Bauer-Bernard:Conformal_transformations_and_SLE_partition_function_martingale}
Michel Bauer and Denis Bernard.
\newblock Conformal transformations and the {SLE} partition function
  martingale.
\newblock {\em Ann. Henri Poincar\'e}, 5(2):289--326, 2004.

\bibitem[BBK05]{KytolaMultipleSLE}
Michel Bauer, Denis Bernard, and Kalle Kyt{\"o}l{\"a}.
\newblock Multiple {S}chramm-{L}oewner evolutions and statistical mechanics
  martingales.
\newblock {\em J.~Stat. Phys.}, 120(5-6):1125--1163, 2005.

\bibitem[BH19]{BenoistHonglerIsingCLE}
St{\'e}phane Benoist and Cl{\'e}ment Hongler.
\newblock The scaling limit of critical Ising interfaces is $\CLE(3)$.
\newblock To appear in {\em Ann. Probab.}, 2019.
\newblock {\em Preprint in} arXiv:1604.06975.

\bibitem[BPW18]{BeffaraPeltolaWuUniquenessGloableMultipleSLEs}
Vincent Beffara, Eveliina Peltola, and Hao Wu. 
\newblock On the uniqueness of global multiple $\SLE$s.
\newblock {\em Preprint in} arXiv:1801.07699, 2018.

\bibitem[BPZ84a]{BPZ:Infinite_conformal_symmetry_in_2D_QFT}
Alexander~A. Belavin, Alexander~M. Polyakov, and Alexander~B. Zamolodchikov.
\newblock Infinite conformal symmetry in two-dimensional quantum field theory.
\newblock {\em Nuclear Phys.~B}, 241(2):333--380, 1984.

\bibitem[BPZ84b]{BelavinPolyakovZamolodchikovConformalSymmetry}
Alexander~A. Belavin, Alexander~M. Polyakov, and Alexander~B. Zamolodchikov.
\newblock Infinite conformal symmetry of critical fluctuations in two
  dimensions.
\newblock {\em J.~Stat. Phys.}, 34(5-6):763--774, 1984.

\bibitem[CN07]{CamiaNewmanPercolation}
Federico Camia and Charles~M. Newman.
\newblock Critical percolation exploration path and {${\rm SLE}\sb 6$}: a proof
  of convergence.
\newblock {\em Probab. Theory Related Fields}, 139(3-4):473--519, 2007.

\bibitem[Car84]{Cardy:Conformal_invariance_and_surface_critical_behavior}
John~L. Cardy.
\newblock Conformal invariance and surface critical behavior.
\newblock {\em Nuclear Phys.~B}, 240(4):514--532, 1984.

\bibitem[Car89]{Cardy:Boundary_conditions_fusion_rules_and_Verlinde_formula}
John~L. Cardy.
\newblock Boundary conditions, fusion rules and the {V}erlinde formula.
\newblock {\em Nuclear Phys.~B}, 324(3):581--596, 1989.

\bibitem[CDCH{\etalchar{+}}14]{CDCHKSConvergenceIsingSLE}
Dmitry Chelkak, Hugo Duminil-Copin, Cl{\'e}ment Hongler, Antti Kemppainen, and
  Stanislav Smirnov.
\newblock Convergence of {I}sing interfaces to {S}chramm's {SLE} curves.
\newblock {\em Comptes Rendus Mathematique}, 352(2):157--161, 2014.

\bibitem[CHI15]{ChelkakHonglerLzyurovConformalInvarianceCorrelationIsing}
Dmitry Chelkak, Cl\'ement Hongler, and Konstantin Izyurov.
\newblock Conformal invariance of spin correlations in the planar {I}sing
  model.
\newblock {\em Ann. of Math.}, 181(3):1087--1138, 2015.

\bibitem[CI13]{ChelkakLzyurovSpinorIsing}
Dmitry Chelkak and Konstantin Izyurov.
\newblock Holomorphic spinor observables in the critical {I}sing model.
\newblock {\em Comm. Math. Phys.}, 322(2):303--332, 2013.

\bibitem[CS12]{ChelkakSmirnovIsing}
Dmitry Chelkak and Stanislav Smirnov.
\newblock Universality in the $2{D}$ {I}sing model and conformal invariance of
  fermionic observables.
\newblock {\em Invent. Math.}, 189(3):515--580, 2012.

\bibitem[DFMS97]{DMS:CFT}
Philippe Di~Francesco, Pierre Mathieu, and David S{\'e}n{\'e}chal.
\newblock {\em Conformal field theory}.
\newblock Springer-Verlag, New York, 1997.

\bibitem[DF85]{DF-four_point_correlation_functions}
Viktor~S.~Dotsenko and Vladimir~A.~Fateev.
\newblock {Four-point correlation functions and the operator algebra in 2D
  conformal invariant theories with $c \geq 1$}.
\newblock {\em Nucl. Phys.~B}, 251:691--734, 1985.

\bibitem[Dub06]{Dubedat:Euler_integrals_for_commuting_SLEs}
Julien Dub{\'e}dat.
\newblock Euler integrals for commuting {SLE}s.
\newblock {\em J.~Stat. Phys.}, 123(6):1183--1218, 2006.

\bibitem[Dub07]{DubedatCommutationSLE}
Julien Dub{\'e}dat.
\newblock Commutation relations for {S}chramm-{L}oewner evolutions.
\newblock {\em Comm. Pure Appl. Math.}, 60(12):1792--1847, 2007.

\bibitem[Dub09]{DubedatSLEFreefield}
Julien Dub{\'e}dat.
\newblock S{LE} and the free field: {P}artition functions and couplings.
\newblock {\em J.~Amer. Math. Soc.}, 22(4):995--1054, 2009.

\bibitem[Dub15a]{Dubedat:SLE_and_Virasoro_representations_localization}
Julien Dub{\'e}dat.
\newblock {SLE} and {V}irasoro representations: {L}ocalization.
\newblock {\em Comm. Math. Phys.}, 336(2):695--760, 2015.

\bibitem[Dub15b]{Dubedat:SLE_and_Virasoro_representations_fusion}
Julien Dub{\'e}dat.
\newblock {SLE} and {V}irasoro representations: {F}usion.
\newblock {\em Comm. Math. Phys.}, 336(2):761--809, 2015.

\bibitem[DS87]{Duplantier-Saleur:Exact_critical_properties_of_two-dimensional_dense_self-avoiding_walks}
Bertrand Duplantier and Hubert Saleur.
\newblock Exact critical properties of two-dimensional dense self-avoiding
  walks.
\newblock {\em Nuclear Phys.~B}, 290:291--326, 1987.

\bibitem[DS11]{DuplantierSheffieldLQGKPZ}
Bertrand Duplantier and Scott Sheffield.
\newblock Liouville quantum gravity and {KPZ}.
\newblock {\em Invent. Math.}, 185(2):333--393, 2011.

\bibitem[FFK89]{FFK:Braid_matrices_and_structure_constants_for_minimal_conformal_models}
Giovanni Felder, J\"urg Fr\"ohlich, and Gerhard Keller.
\newblock Braid matrices and structure constants for minimal conformal models.
\newblock {\em Comm. Math. Phys.}, 124(4):647--664, 1989.

\bibitem[FL13]{FieldLawlerReversedRadialSLEBrownianLoop}
Laurence~S. Field and Gregory~F. Lawler.
\newblock{Reversed Radial SLE and the Brownian Loop Measure}.
\newblock {\em J.~Stat. Phys.}, 150:1030--1062 (2013)

\bibitem[FK15a]{FloresKlebanSolutionSpacePDE1}
Steven~M. Flores and Peter Kleban.
\newblock A solution space for a system of null-state partial differential
  equations: {P}art 1.
\newblock {\em Comm. Math. Phys.}, 333(1):389--434, 2015.

\bibitem[FK15b]{FloresKlebanSolutionSpacePDE2}
Steven~M. Flores and Peter Kleban.
\newblock A solution space for a system of null-state partial differential
  equations: {P}art 2.
\newblock {\em Comm. Math. Phys.}, 333(1):435--481, 2015.

\bibitem[FK15c]{FloresKlebanSolutionSpacePDE3}
Steven~M. Flores and Peter Kleban.
\newblock A solution space for a system of null-state partial differential
  equations: {P}art 3.
\newblock {\em Comm. Math. Phys.}, 333(2):597--667, 2015.

\bibitem[FK15d]{FloresKlebanSolutionSpacePDE4}
Steven~M. Flores and Peter Kleban.
\newblock A solution space for a system of null-state partial differential
  equations: {P}art 4.
\newblock {\em Comm. Math. Phys.}, 333(2):669--715, 2015.

\bibitem[FP19{\etalchar{+}}]{Flores-Peltola:Monodromy_invariant_correlation_function}
Steven~M. Flores and Eveliina Peltola.
\newblock Monodromy invariant {CFT} correlation functions of first column {K}ac
  operators.
\newblock {\em In preparation}.

\bibitem[FSK15]{FSK-Multiple_SLE_connectivity_weights_for_rectangles_hexagons_and_octagons}
Steven~M. Flores, Jacob J.~H. Simmons, and Peter Kleban.
\newblock Multiple-{SLE}$_\kappa$ connectivity weights for rectangles, hexagons,
  and octagons.
\newblock {\em Preprint in} arXiv:1505.07756, 2015.

\bibitem[FSKZ17]{FSKZ-A_formula_for_crossing_probabilities_of_critical_systems_inside_polygons}
Steven~M. Flores, Jacob J.~H. Simmons, Peter Kleban, and Robert~M. Ziff.
\newblock A formula for crossing probabilities of critical systems inside polygons.
\newblock {\em J.~Phys.~A}, 50(6):064005, 2017.

\bibitem[Fri04]{Friedrich:On_connections_of_CFT_and_SLE}
Roland Friedrich.
\newblock On connections of conformal field theory and stochastic {L}oewner
  evolution.
\newblock {\em Preprint in} arXiv:math-ph/0410029, 2004.

\bibitem[FK04]{Friedrich-Kalkkinen:On_CFT_and_SLE}
Roland Friedrich and Jussi Kalkkinen.
\newblock On conformal field theory and stochastic {L}oewner evolution.
\newblock {\em Nuclear Phys.~B}, 687(3):279--302, 2004.

\bibitem[FW03]{Friedrich-Werner:Conformal_restriction_highest_weight_representations_and_SLE}
Roland Friedrich and Wendelin Werner.
\newblock Conformal restriction, highest weight representations and {SLE}.
\newblock {\em Comm. Math. Phys.}, 243(1):105--122, 2003.

\bibitem[Gra07]{GrahamSLE}
Kevin Graham.
\newblock On multiple {S}chramm-{L}oewner evolutions.
\newblock {\em J.~Stat. Mech. Theory Exp.}, (3):P03008, 21, 2007.

\bibitem[HK13]{HonglerKytolaIsingFree}
Cl{\'e}ment Hongler and Kalle Kyt{\"o}l{\"a}.
\newblock Ising interfaces and free boundary conditions.
\newblock {\em J.~Amer. Math. Soc.}, 26(4):1107--1189, 2013.

\bibitem[HS13]{HonglerSmirnovIsingEnergy}
Cl{\'e}ment Hongler and Stanislav Smirnov.
\newblock The energy density in the planar Ising model.
\newblock {\em Acta Math.}, 211(2):191--225, 2013.

\bibitem[H{\"o}r67]{HormanderHypoelliptic}
Lars H{\"o}rmander.
\newblock Hypoelliptic second-order differential equations.
\newblock {\em Acta Math.}, 119:147--171, 1967.

\bibitem[H{\"o}r90]{HormanderFA}
Lars H{\"o}rmander.
\newblock {\em The analysis of linear partial differential operators I: Distribution theory and Fourier analysis}, volume~256 of
  {\em Grundlehren der mathematischen Wissenschaften}.
\newblock Springer-Verlag, Berlin Heidelberg, second edition, 1990.


\bibitem[Izy15]{IzyurovObservableFree}
Konstantin Izyurov.
\newblock Smirnov's observable for free boundary conditions, interfaces and
  crossing probabilities.
\newblock {\em Comm. Math. Phys.}, 337(1):225--252, 2015.

\bibitem[Izy17]{IzyurovIsingMultiplyConnectedDomains}
Konstantin Izyurov.
\newblock Critical {I}sing interfaces in multiply-connected domains.
\newblock {\em Probab. Theory Related Fields}, 167(1):379--415, 2017.

\bibitem[JL18]{LawlerSmoothness}
Mohammad Jahangoshahi and Gregory F.~Lawler.
\newblock On the smoothness of the partition function for multiple Schramm-Loewner evolutions.
\newblock {\em J.~Stat. Phys.}, 173(5):1353--1368, 2018.

\bibitem[Ken08]{KenyonDimer}
Richard~W. Kenyon.
\newblock Height fluctuations in the honeycomb dimer model.
\newblock {\em Comm. Math. Phys.}, 281(3):675--709, 2008.

\bibitem[KKP17]{KKP:Correlations_in_planar_LERW_and_UST_and_combinatorics_of_conformal_blocks}
Alex Karrila, Kalle Kyt{\"o}l{\"a}, and Eveliina Peltola.
\newblock Boundary correlations in planar {LERW} and {UST}.
\newblock {\em Preprint in} arXiv:1702.03261, 2017.

\bibitem[KKP19]{KKP:Conformal_blocks}
Alex Karrila, Kalle Kyt{\"o}l{\"a}, and Eveliina Peltola.
\newblock Conformal blocks, $q$-combinatorics, and quantum group symmetry.
\newblock To appear in {\em Ann. Henri Poincar\'e D}, 2019.
\newblock {\em Preprint in} arXiv:1709.00249.

\bibitem[KL05]{KozdronLawlerLERWestimates}
Michael~J. Kozdron and Gregory~F. Lawler.
\newblock Estimates of random walk exit probabilities and application to loop-erased random walk.
\newblock {\em Electron. J.~Probab.}, 10(44), 1442--1467, 2005.

\bibitem[KL07]{KozdronLawlerMultipleSLEs}
Michael~J. Kozdron and Gregory~F. Lawler.
\newblock The configurational measure on mutually avoiding {SLE} paths.
\newblock In {\em Universality and renormalization}, volume~50 of {\em Fields
  Inst. Commun.}, pages 199--224. Amer. Math. Soc., Providence, RI, 2007.
  
  
\bibitem[Kon03]{Kontsevich:CFT_SLE_and_phase_boundaries}
Maxim Kontsevich.
\newblock CFT, $\SLE$, and phase boundaries.
\newblock {\em Oberwolfach Arbeitstagung}, 2003.

\bibitem[KP16]{KytolaPeoltolaPurePartitionSLE}
Kalle Kyt{\"o}l{\"a} and Eveliina Peltola.
\newblock Pure partition functions of multiple {SLE}s.
\newblock {\em Comm. Math. Phys.}, 346(1):237--292, 2016.

\bibitem[KP19]{Kytola-Peltola:Conformally_covariant_boundary_correlation_functions_with_quantum_group}
Kalle Kyt{\"o}l{\"a} and Eveliina Peltola.
\newblock Conformally covariant boundary correlation functions with a quantum
  group.
 \newblock To appear in {\em J.~Eur. Math. Soc.}, 2019.
\newblock {\em Preprint in} arXiv:1408.1384.

\bibitem[KS18]{KemppainenSmirnovFKIsing}
Antti Kemppainen and Stanislav Smirnov.
\newblock Configurations of {FK} {I}sing interfaces and hypergeometric {SLE}.
\newblock {\em Math. Res. Lett.}, 25(3):875--889, 2018.


\bibitem[KW11a]{Kenyon-Wilson:Boundary_partitions_in_trees_and_dimers}
Richard~W. Kenyon and David~B. Wilson.
\newblock Boundary partitions in trees and dimers.
\newblock {\em Trans. Amer. Math. Soc.}, 363(3):1325--1364, 2011.

\bibitem[KW11b]{Kenyon-Wilson:Double_dimer_pairings_and_skew_Young_diagrams}
Richard~W. Kenyon and David~B. Wilson.
\newblock Double-dimer pairings and skew {Y}oung diagrams.
\newblock {\em Electron. J.~Combin.}, 18(1):130--142, 2011.

\bibitem[Kyt07]{KytolaVirasoroSLE}
Kalle Kyt\"ol\"a.
\newblock Virasoro module structure of local martingales of {SLE} variants.
\newblock {\em Rev. Math. Phys.}, 19(5):455--509, 2007.

\bibitem[Law05]{LawlerConformallyInvariantProcesses}
Gregory~F. Lawler.
\newblock {\em Conformally invariant processes in the plane}, volume~114 of
  {\em Mathematical Surveys and Monographs}.
\newblock Amer. Math. Soc., Providence, RI, 2005.

\bibitem[Law09a]{LawlerPartitionFunctionsSLE}
Gregory~F. Lawler.
\newblock Partition functions, loop measure, and versions of {SLE}.
\newblock {\em J.~Stat. Phys.}, 134(5-6):813--837, 2009.

\bibitem[Law09b]{LawlerSLENotes}
Gregory~F. Lawler.
\newblock Schramm-Loewner Evolution (SLE), in {\em Statistical Mechanics}, S. Sheffield and T. Spencer ed., 
volume~16 of {\em AMS IAS/Park City Mathematics Series}, 2009.
  
\bibitem[Law11]{LawlerSLEMultiplyConnected}
Gregory~F. Lawler.
\newblock Defining SLE in multiply connected domains with the Brownian loop measure.
\newblock {\em Preprint in} arXiv:1108.4364, 2011.

\bibitem[LSW03]{LawlerSchrammWernerConformalRestriction}
Gregory~F. Lawler, Oded Schramm, and Wendelin Werner.
\newblock Conformal restriction: the chordal case.
\newblock {\em J.~Amer. Math. Soc.}, 16(4):917--955 (electronic), 2003.

\bibitem[LSW04]{LawlerSchrammWernerLERWUST}
Gregory~F. Lawler, Oded Schramm, and Wendelin Werner.
\newblock Conformal invariance of planar loop-erased random walks and uniform
  spanning trees.
\newblock {\em Ann. Probab.}, 32(1B):939--995, 2004.

\bibitem[LW04]{LawlerWernerBrownianLoopsoup}
Gregory~F. Lawler and Wendelin Werner.
\newblock The {B}rownian loop soup.
\newblock {\em Probab. Theory Related Fields}, 128(4):565--588, 2004.

\bibitem[LV17]{Lenells-Viklund:Coulomb_gas_integrals_for_commuting_SLEs1}
Jonatan Lenells and Fredrik Viklund.
\newblock Schramm's formula and the {G}reen's function for multiple {$\mathrm{SLE}$}.
\newblock {\em Preprint in} arXiv:1701.03698, 2017.

\bibitem[LV18]{Lenells-Viklund:Coulomb_gas_integrals_for_commuting_SLEs2}
Jonatan Lenells and Fredrik Viklund.
\newblock Asymptotic analysis of {D}otsenko-{F}ateev integrals.
\newblock {\em Preprint in} arXiv:arXiv:1812.07528, 2018.


\bibitem[MS16]{MillerSheffieldIG1}
Jason Miller and Scott Sheffield.
\newblock Imaginary geometry {I}: {I}nteracting {SLE}s.
\newblock {\em Probab. Theory Related Fields}, 164(3-4):553--705, 2016.

\bibitem[Nie87]{Nienhuis:Coulomb_gas_formulation_of_2D_phase_transitions}
Bernard Nienhuis.
\newblock Coulomb gas formulation of two-dimensional phase transitions.
\newblock In C.~Domb and J.~L. Lebowitz ed., volume~11 of {\em Phase transitions and
  critical phenomena}, pages 1--53. Academic Press, London, 1987.

\bibitem[PW18]{PeltolaWuIsing}
Eveliina Peltola and Hao Wu. 
\newblock Crossing probabilities of multiple {I}sing interfaces.
\newblock {\em Preprint in} arXiv:1808.09438, 2018.

\bibitem[Rib14]{Ribault:CFT_in_the_plane}
Sylvain Ribault.
\newblock Conformal field theory on the plane.
\newblock {\em Preprint in} arXiv:1406.4290, 2014.

\bibitem[RS05]{RohdeSchrammSLEBasicProperty}
Steffen Rohde and Oded Schramm.
\newblock Basic properties of {SLE}.
\newblock {\em Ann. of Math.}, 161(2):883--924, 2005.

\bibitem[RY99]{RevuzYorMartBM}
Daniel Revuz and Marc Yor.
\newblock {\em Continuous martingales and {B}rownian motion}, volume~293 of
  {\em Grundlehren der mathematischen Wissenschaften}.
\newblock Springer-Verlag, Berlin Heidelberg, third edition, 1999.

\bibitem[Rud91]{RudinFA}
Walter Rudin.
\newblock {\em Functional Analysis}.
\newblock International Series in Pure and Applied Mathematics. McGeaw-Hill, second edition, 1991.

\bibitem[Sch00]{SchrammScalinglimitsLERWUST}
Oded Schramm.
\newblock Scaling limits of loop-erased random walks and uniform spanning
  trees.
\newblock {\em Israel J.~Math.}, 118:221--288, 2000.

\bibitem[SS09]{SchrammSheffieldDiscreteGFF}
Oded Schramm and Scott Sheffield.
\newblock Contour lines of the two-dimensional discrete {G}aussian free field.
\newblock {\em Acta Math.}, 202(1):21--137, 2009.

\bibitem[SS13]{SchrammSheffieldContinuumGFF}
Oded Schramm and Scott Sheffield.
\newblock A contour line of the continuum {G}aussian free field.
\newblock {\em Probab. Theory Related Fields}, 157(1-2):47--80, 2013.

\bibitem[SW05]{SchrammWilsonSLECoordinatechanges}
Oded Schramm and David~B. Wilson.
\newblock S{LE} coordinate changes.
\newblock {\em New York J.~Math.}, 11:659--669 (electronic), 2005.

\bibitem[SZ12]{Shigechi-Zinn:Path_representation_of_maximal_parabolic_Kazhdan-Lusztig_polynomials}
Keiichi Shigechi and Paul Zinn-Justin.
\newblock Path representation of maximal parabolic {K}azhdan-{L}usztig
  polynomials.
\newblock {\em J.~Pure Appl. Algebra}, 216(11):2533--2548, 2012.

\bibitem[She07]{SheffieldGFFMath}
Scott Sheffield.
\newblock Gaussian free fields for mathematicians.
\newblock {\em Probab. Theory Related Fields}, 139(3-4):521--541, 2007.

\bibitem[Smi01]{SmirnovPercolationConformalInvariance}
Stanislav Smirnov.
\newblock Critical percolation in the plane: conformal invariance, {C}ardy's
  formula, scaling limits.
\newblock {\em C.~R.~Acad. Sci. Paris S\'er. I Math.}, 333(3):239--244, 2001.

\bibitem[Smi06]{SmirnovConformalInvariance}
Stanislav Smirnov.
\newblock Towards conformal invariance of $2{D}$ lattice models.
\newblock In {\em International {C}ongress of {M}athematicians. {V}ol. {II}},
  pages 1421--1451. Eur. Math. Soc., Z\"urich, 2006.

\bibitem[Smi10]{SmirnovConformalInvarianceAnnals}
Stanislav Smirnov.
\newblock Conformal invariance in random cluster models. I. Holomorphic fermions in the {I}sing model.
\newblock {\em Ann. of Math.}, 172(2):1435--1467, 2010.

\bibitem[Str08]{StroockPDEs}
Daniel W.~Stroock.
\newblock {\em An introduction to partial differential equations for probabilists}, volume~112 of
  {\em Cambridge Studies in Advanced Mathematics}.
\newblock Cambridge University Press, Cambridge, 2008.


\bibitem[Tao09]{TaoEpsilon}
Terence Tao.
\newblock {\em An epsilon of room, I: Real analysis. Pages from year three of a mathematical blog},
volume~117 of {\em Graduate Studies in Mathematics}. 
\newblock Amer. Math. Soc., Providence, RI, 2009.



\bibitem[Wer04]{WernerGirsanov}
Wendelin Werner.
\newblock Girsanov's transformation for {${\rm SLE}(\kappa,\rho)$} processes,
  intersection exponents and hiding exponents.
\newblock {\em Ann. Fac. Sci. Toulouse Math. (6)}, 13(1):121--147, 2004.

\bibitem[Wu17]{WuIsingHyperSLE}
Hao Wu.
\newblock Hypergeometric {SLE}: conformal Markov characterization and applications.
\newblock {\em Preprint in} arXiv:1703.02022, 2017. 

\bibitem[WW17]{WangWuLevellinesGFFI}
Menglu Wang and Hao Wu.
\newblock Level lines of {G}aussian free field {I}: zero-boundary {GFF}.
\newblock {\em Stochastic Process. Appl.}, 127(4):1045--1124, 2017.

\bibitem[Wu18]{WuAlternatingArmIsing}
Hao Wu.
\newblock Alternating arm exponents for the critical planar {I}sing model.
\newblock {\em Ann. Probab.}, 46(5): 2863--2907, 2018.

\bibitem[Zha08a]{ZhanReversibility}
Dapeng Zhan.
\newblock Reversibility of chordal {SLE}.
\newblock {\em Ann. Probab.}, 36(4):1472--1494, 2008.

\bibitem[Zha08b]{Zhan:Scaling_limits_of_planar_LERW_in_finitely_connected_domains}
Dapeng Zhan.
\newblock The scaling limits of planar {LERW} in finitely connected domains.
\newblock {\em Ann. Probab.}, 36(2):467--529, 2008.

\end{thebibliography}
\end{document}